\documentclass[12pt,amsymb,fullpage]{amsart}
\usepackage{amssymb,amscd,pstricks}

\newtheorem{theorem}{Theorem}[section]
\newtheorem{defn}[theorem]{Definition}

\newtheorem{lemma}[theorem]{Lemma}

\newtheorem{eple}[theorem]{Example}
\newtheorem{rmk}[theorem]{Remarks}
\newtheorem{dsc}[theorem]{Discussion}
\newtheorem{nota}[theorem]{Notation}

\newsavebox{\indbin}
\savebox{\indbin}{\begin{picture}(0,0)
\newlength{\gnu}
\settowidth{\gnu}{$\smile$} \setlength{\unitlength}{.5\gnu}
\put(-1,-.65){$\smile$} \put(-.25,.1){$|$}
\end{picture}}

\newcommand{\be}{\begin{enumerate}}
\newcommand{\bd}{\begin{defn}}
\newcommand{\bt}{\begin{theorem}}
\newcommand{\bl}{\begin{lemma}}
\newcommand{\ee}{\end{enumerate}}
\newcommand{\ed}{\end{defn}}
\newcommand{\et}{\end{theorem}}
\newcommand{\el}{\end{lemma}}

\begin{document}
\title{A Theory of Branches for Algebraic Curves}
\author{Tristram de Piro}
\address{Mathematics Department, The University of Camerino, Camerino, Italy}
 \email{tristam.depiro@unicam.it}
\thanks{The author was supported by a Marie Curie Research Fellowship}
\begin{abstract}
This paper develops some of the methods of the "Italian School" of
algebraic geometry in the context of infinitesimals. The results
of this paper have no claim to originality, they can be found in
\cite{Sev}, we have only made the arguments acceptable by modern
standards. However, as the question of rigor was the main
criticism of their approach, this is still a useful project. The
results are limited to algebraic curves. As well as being
interesting in their own right, it is hoped that these may also
help the reader to appreciate their sophisticated approach to
algebraic surfaces and an understanding of singularities. The
constructions are also relevant to current research in Zariski
structures, which have played a major role both in model theoretic
applications to diophantine geometry and in recent work on
non-commutative geometry.
\end{abstract}
\maketitle
\begin{section}{Introduction, Preliminary Definitions, Lemmas and Notation}

We begin this section with the preliminary reminder to the reader
that the following results are concerned with \emph{algebraic}
curves. However, the constructions involved are geometric and rely
heavily on the techniques of Zariski structures, originally
developed in \cite{Z} and \cite{HZ}. One might, therefore,
speculate that the results could, in themselves, be used to
develop further the general theory of such structures. Our
starting point is the main Theorem 17.1 of \cite{Z}, also
formulated for Zariski geometries in \cite{HZ};\\

\begin{theorem}{Main Theorem 17.1 of \cite{Z}}\\

Let $M$ be a Zariski structure and $C$ a presmooth Zariski curve
in $M$. If $C$ is non-linear, then there exists a nonconstant
continuous map;\\

$f:C\rightarrow P^{1}(K)$\\

Moreover, $f$ is a finite map ($f^{-1}(x)$ is finite for every
$x\in C$), and for any $n$, for any definable $S\subseteq C^{n}$,
the image $f(S)$ is a constructible subset (in the sense of
algebraic geometry) of $[P^{1}(K)]^{n}$.
\end{theorem}

\begin{rmk}
Here, $K$ denotes an algebraically closed field, we refer the
reader to the original paper for the remaining terminology of the
Theorem.
\end{rmk}

Using this theorem, one can already see that there exists a close
connection between a geometric theory of \emph{algebraic} curves
and Zariski curves. We begin this section by pointing out some of
the main obstacles to developing the results of this paper in the
context of Zariski curves, leaving the resolution of the main
technical problems for another occasion. This discussion continues
up to $(\dag\dag)$, when we introduce the main notation and
preliminary lemmas of the
paper.\\

In Section 2, the first major obstacle that we encounter is a
suitable generalisation of the notion of a linear system. Using
Theorem 1.1, any linear system $\Sigma$ of algebraic hypersurfaces
on $[P^{1}(K)]^{n}$ will define a linear system of Zariski
hypersurfaces on $C^{n}$ , by composing with the finite cover $f$
$(*)$. One would then expect to be able to develop much of the
theory of $g_{n}^{r}$ given in Section $2$ for such systems,
applied to a Zariski curve $S\subseteq C^{n}$. This follows from
the following observations. First, there exists a generalised
Bezout's theorem, holding in Zariski structures, see the paper
\cite{Z}, hence one might hope to obtain an analogous result to
Theorem 2.3. Secondly, one would expect that the local
calculations, by algebraic power series, which we used in Lemma
2.10, would transfer to intersections on $S$. This uses the fact
that, at a point where $f$ is unramified, $S$ and the algebraic
curve $f(S)$ are \emph{locally} isomorphic, in the sense of
infinitesimal neighborhoods, and, at a point $p$ where $f$ is
ramified, with multiplicity $r$, we have the straightforward relation;\\

$I_{italian}^{(\Sigma)}(p,S,f^{*}(\phi_{\lambda}))=rI_{italian}^{(\Sigma)}(f(p),f(S),\phi_{\lambda})$\\

where I have used the notation of Theorem 2.3, $\phi_{\lambda}$
denotes a hypersurface in $[P^{1}(K)]^{n}$ and
$f^{*}(\phi_{\lambda})$ denotes its inverse image in $S^{n}$.\\

In Section 3, the notion of a multiple point is introduced.  As
this is defined locally, one would expect this definition to be
generalisable to Zariski curves. Also, the geometric notion of $2$
algebraic curves being biunivocal has a natural generalisation to
Zariski curves. However, the main result of Section 3, that any
algebraic curve is birational(biunivocal) to an algebraic curve
without multiple points, is not so easily transferred. This
follows from the simple observation that there exist Zariski
curves $S$, which \emph{cannot} be embedded in a projective space
$P^{n}(K)$, see Section 10 of \cite{HZ}. This failure of Theorem
3.3 could be explained alternatively, by noting that the
combinatorics involved must fail for multiplicity calculations
using Zariski curves. The solution to this problem seems to be
quite difficult and one must presumably attempt to resolve it by
introducing a larger class of hypersurfaces than that defined by a
linear system $\Sigma$, as
in $(*)$ above. See also the remarks below.\\

In Section 4, the method of Conic projections is not immediately
transferable to Zariski structures, as one needs to define what is
meant by a \emph{line} in this context. However, we note the
following geometric property of a line in relation to other
algebraic curves;\\

\begin{theorem}(Luroth)
Let $f:l\rightarrow C$ be a finite morphism of a line $l$ onto a
projective algebraic curve $C$. Then $f$ is biunivocal.
\end{theorem}

\begin{proof}
See \cite{H}.
\end{proof}

This theorem in fact characterises a line $l$ up to birationality.
Its proof requires a global topological property of the line,
namely that the genus of $l$ is zero. Although the above property
may be formulated for Zariski curves, it does not guarantee the
\emph{existence} of such a curve $S\subseteq (C^{n})^{eq}$, which
is not trivially biunivocal to an \emph{algebraic} line. One would
expect the solution of this problem to require more advanced
techniques, such as a geometric definition of the genus of a
Zariski curve. Severi, in fact, gives such a definition for
algebraic curves in \cite{Sev} and one might hope that his
definition would generalise to Zariski curves. One could then hope
to extend the methods of Conic projections in this context.\\

The results of Section 5 rely centrally on the main result of
Section 3, hence their generalisation to Zariski curves require a
resolution of the problems noted above. One should observe that
the Italian geometers definition of a branch is very different to
a local definition using algebraic power series, see also the
remarks later in this section, hence cannot be straightforwardly
generalised using Theorem 1.1. Given a Zariski curve $S\subset
C^{n}$ and $p\in S$, let
$\{\gamma_{p}^{1},\ldots,\gamma_{p}^{n}\}$ enumerate the branches
of $f(p)\in f(S)$. We can then define;\\

$\overline{\gamma}_{p}^{j}=\{x\in S\cap{\mathcal
V}_{p}:f(x)\in\gamma_{p}^{j}\}$\\

In the case when $f$ is unramified (in the sense of Zariski
structures) at $p$, this would give an adequate definition of a
branch $\overline{\gamma}_{p}^{j}$ for the Zariski curve $C$.
However, the definition is clearly inadequate when $p$ is a point
of ramification and requires a generalisation of the methods of
Section 5. \\

The results of Section 6 depend mainly on Cayley's classification
of singularities. As the main technical tool used in the proof of
this result is the method of algebraic power series, by using
Theorem 1.1, one would expect the results given in this Section to
generalise more easily to Zariski curves.\\

$(\dag\dag)$ We will work in the language ${\mathcal L}_{spec}$,
as defined in \cite{depiro3}. $P(L)$ will denote $\bigcup_{n\geq
1}P^{n}(L)$ where $L$ is an algebraically closed field. Unless
otherwise stated, we will assume that the field has characteristic
$0$, this is to avoid problems concerning Frobenius. The results
that we prove in this paper hold in arbitrary characteristic, if
we avoid
exceptional cases, however we need to make certain modifications to the proof. We will discuss these modifications in the final section of the paper.\\

 We assume the existence of a universal specialisation $P(K)\rightarrow_{\pi} P(L)$ where $K$ is
algebraically closed and $L\subset K$. Given $l\in L$, we will
denote its infinitesimal neighborhood by ${\mathcal V}_{l}$, that
is $\pi^{-1}(l)$. As we noted in the paper \cite{depiro3}, it is
not strictly necessary to consider a universal specialisation when
defining the non-standard intersection  multiplicity of curves,
one need only consider a prime model of the theory $T_{spec}$.
However, some of our proofs will require more refined
infinitesimal arguments which are not first order in the language
${\mathcal L}_{spec}$ and therefore cannot immediately be
transferred to a prime model. We will only refer to the
non-standard model when using infinitesimal arguments.We assume
the reader is familiar with the arguments employed in the papers
\cite{depiro1}, \cite{depiro2} and \cite{Z}. Of particular
importance are the following notions;\\

(i). The technique of Zariski structures. We assume that
$P^{n}(L)$ may be considered as a Zariski structure in the
topology induced by algebraically closed subvarieties. When
referring to the dimension of an algebraically closed subvariety
$V$, we will use the model theoretic definition as given in the
paper \cite{depiro1}. We will assume the reader is acquainted with
the notions surrounding the meaning of "generic" in this context.\\

(ii). The method of algebraic power series and their relation to
infinitesimals. This technique was explored extensively in the
papers \cite{depiro1} and \cite{depiro2}. In the following paper,
we will use power series methods to parametrise the branches of
algebraic curves without being overly rigorous. By a branch in
this context, we refer to the "Newtonian definition" rather than
the one used by the Italian school of algebraic geometry,
 which is the subject of this paper.\\

(iii). The non-standard statement and proof of Bezout's theorem
for projective algebraic curves in $P^{2}(L)$. This was given in
the paper \cite{depiro2}.\\

We will assume that $L$ has infinite transcendence degree and
therefore has the property;\\

Given any subfield $L_{0}\prec L$ of finite transcendence degree
and an integer $n\geq 1$, we can find $\bar a_{n}\in P^{n}(L)$
which is generic over $L_{0}$. (*)\\

We refer the reader to (i) above for the relevant definition of
generic. In general, when referring to a generic point, we will
mean generic with respect to some algebraically closed field of
finite transcendence degree. This field will be the algebraic
closure of the parameters defining any algebraic object given in
the specific context.\\

By a projective algebraic curve, we will mean a closed irreducible
algebraic subvariety $C$ of $P^{n}(L)$ for some $n\geq 1$ having
dimension $1$. Occasionally, we will consider the case when $C$
has distinct irreducible components $\{C_{1},\ldots,C_{r}\}$,
which will be made clear in a given situation. In both cases, we
need only consider the usual Zariski topology on $P^{n}(L)$ as
given in $(i)$ above. By a plane projective algebraic curve, we
will mean a projective algebraic curve contained in $P^{2}(L)$. By
a projective line $l$ in $P^{n}$, we will mean a projective
algebraic curve isomorphic to $P^{1}(L)$. Any distinct points
$\{p_{1},p_{2}\}$ in $P^{n}(L)$ determine a unique projective line
denoted by $l_{p_{1}p_{2}}$. We will call the line generic if
there exists a generic pair $\{p_{1}p_{2}\}$ determining it. We
occasionally assume the existence of the closed algebraic variety
$I\subset {(P^{n}\times P^{n})\setminus\Delta}\times P^{n}$, which
parametrises the family of lines in $P^{n}$, defined by;\\

$I(a,b,y)\equiv y\in l_{ab}$\\

In order to see that this does define a closed algebraic variety,
take the standard open cover $\{U_{i}:=(X_{i}\neq 0):0\leq i\leq
n\}$ of $P^{n}$. Then observe that $I\cap ({(U_{i}\times
U_{j})\setminus\Delta}\times P^{n})$ is
locally trivialisable by the maps;\\

$\Theta_{ij}:{(U_{i}\times U_{j})\setminus\Delta}\times
P^{1}\rightarrow I\cap
(U_{i}\times U_{j}\times P^{n})$\\

$\Theta_{ij}(a,b,[t_{0}:t_{1}])=$\\

$[t_{0}{a_{0}\over a_{i}}+t_{1}{b_{0}\over
b_{j}}:\ldots:t_{0}+t_{1}{b_{i}\over
b_{j}}:\ldots:t_{0}{a_{j}\over
a_{i}}+t_{1}:\ldots:t_{0}{a_{n}\over
a_{i}}+t_{1}{b_{n}\over b_{j}}]$\\

and the transition functions
$\Theta_{ijkl}=\Theta_{ij}\circ\Theta_{kl}^{-1}$ are algebraic. In
general, we will leave the reader to check in the course of the
paper that certain naturally defined algebraic varieties are in
fact algebraic. By a projective plane $P$ in $P^{n}$, we will mean
a closed irreducible projective subvariety of $P^{n}$, isomorphic
to $P^{r}(L)$, for some $0\leq r\leq n$. We did, however, use a
special notation for a plane projective curve, as defined above.
Any sequence of points $\bar a$ determines a unique projective
plane $P_{\bar a}$, defined as the intersection of all planes
containing $\bar a$. We will call a sequence
$\{p_{0},\ldots,p_{r}\}$ linearly independent if
$P_{p_{0},\ldots,p_{r}}$ is isomorphic to $P^{r}(L)$. As before,
if $U\subset (P^{n})^{r+1}$ defines the open subset of linearly
independent elements, we can define the cover $I\subset U\times
P^{n}$,
which parametrises the family of $r$-dimensional planes in $P^{n}$;\\

$I(p_{0},\ldots,p_{r},y)\equiv y\in P_{p_{1},\ldots,p_{r}}$\\

As before, one can see that this is a closed projective algebraic
subvariety of $U\times P^{n}$.\\

By a non-singular projective algebraic curve, we mean a projective
algebraic curve $C\subset P^{n}$ which is non-singular in the
sense of \cite{H} (p32). Given any point $p\in C$, we then define
its tangent line $l_{p}$ as follows;\\

By Theorem 8.17 of \cite{H}, there exist homogeneous polynomials\\
$\{G_{1},\ldots,G_{n-1}\}$ such that $C$ is defined in an affine
open neighborhood $U$ of $p$ in $P^{n}$ by the homogenous ideal
$J=<G_{1},\ldots,G_{n-1}>$ $(*)$. Let $dG_{j}={\partial
G_{j}\over\partial X_{0}}X_{0}+\ldots+{\partial G_{j}\over\partial
X_{n}}X_{n}$ be the differential of $G_{j}$. Then $dG_{j}$ defines
a family of hyperplanes, parametrised by an open neighborhood of
$p$ in $C$. If $x_{i}={X_{i}\over X_{0}}$ is a choice of affine
coordinate system containing $p$ and
$G_{j}^{res}(x_{1},\ldots,x_{n})={G_{j}\over X_{0}^{deg(G_{j})}}$,
we define $dG_{j}^{res}={\partial G_{j}^{res}\over\partial
x_{1}}dx_{1}+\ldots+{\partial G_{j}^{res}\over\partial
x_{n}}dx_{n}$ and $J^{res}=<G_{1}^{res},\ldots,G_{n-1}^{res}>$.
Then $dG_{j}^{res}$ defines a family of \emph{affine} hyperplanes,
parametrised by an open neighborhood of $p$ in $C$. We claim that
$dG_{j}^{res}(p)=dG_{j}(p)^{res}$, this follows by
an easy algebraic calculation, using the fact that;\\

${\partial G_{j}\over\partial X_{0}}(p)p_{0}+\ldots+{\partial
G_{j}\over\partial X_{n}}(p)p_{n}=0$\\

The differentials $\{dG^{res}_{1},\ldots,dG^{res}_{n-1}\}$ are
independent at $p$, by the same Theorem 8.17 of \cite{H}, hence
$\bigcap_{1\leq j\leq n-1}dG_{j}(p)$, defines a line $l_{p}\subset
P^{n}$, which we call the tangent line. We need to show that the
definition is independent of the choice of
$\{G_{1},\ldots,G_{n-1}\}$. Suppose that we are given another
choice $\{H_{1},\ldots,H_{n-1}\}$. Again, using Theorem 8.17 of
\cite{H}, we can find a matrix $(f_{ij})_{1\leq i,j\leq n-1}$ of
polynomials in $R(U)$, for some affine open neighborhood
$U$ of $p$ in $P^{n}$ $(**)$, such that the matrix of values $(f_{ij}(p))_{1\leq i,j\leq n-1}$ is invertible and; \\

$H_{i}^{res}=\sum_{1\leq j\leq n-1}f_{ij}G_{j}^{res}$ (mod $(J^{res}(U))^{2}$)\\

Then, by properties of differentials, and the fact that,
 for $g\in (J^{res}(U))^{2}$, $dg(p)\equiv 0$, we have;\\

$dH_{i}^{res}(p)=\sum_{1\leq j\leq n-1}f_{ij}(p)dG_{j}^{res}(p)$\\

This implies that $\bigcap_{1\leq j\leq
n-1}dG_{j}(p)=\bigcap_{1\leq j\leq n-1}dH_{j}(p)$ as required.\\

We define the tangent variety $Tang(C)$ to be $\bigcup_{x\in
C}l_{x}$. We claim that this is a closed projective subvariety of
$P^{n}$. In order to see this, using the notation of the above
argument, let $\{U_{i}\}$ be an affine cover of $P^{n}$ and let
$\{G_{i1},\ldots,G_{ij},\ldots,G_{i,n-1}\}$ be homogeneous
polynomials with the properties $(*)$ and $(**)$ given above. Let
$W\subset P^{n}\times P^{n}$ be the closed
projective variety, defined on $P^{n}\times U_{i}$ by;\\

$W_{i}(\bar X,\bar Y)\equiv U_{i}(\bar Y)\wedge\bigwedge_{1\leq
j\leq n-1}G_{ij}(\bar Y)=0\wedge\bigwedge_{1\leq j\leq n-1}
dG_{ij}(\bar
Y)\centerdot\bar X=0$\\

Then, by completeness, the variety $V(\bar X)\equiv (\exists \bar
Y)W(\bar X,\bar Y)$ is a closed projective variety of $P^{n}$. The
above argument shows that $V=Tang(C)$.\\

We will require a more sophisticated notation when considering
hypersurfaces $H$ of $P^{n}(L)$ for $n\geq 1$. Namely, by a
hypersurface of degree $m$, we will mean a homogenous polynomial
$F(X_{0},X_{1},\ldots,X_{n})$ of degree $m$ in the variables
$[X_{0}:X_{1}:\ldots:X_{n}]$. In the case when $F$ is irreducible
or has distinct irreducible factors $\{F_{1},\ldots,F_{r}\}$, such
a hypersurface may be considered as the union of $r$ distinct
irreducible closed algebraic subvarieties of codimension $1$ in
$P^{n}(L)$. We will then refer to the hypersurface as having
distinct irreducible components. Again, we can understand such
hypersurfaces using the usual Zariski topology as given in $(i)$
above. By a hyperplane, we will mean an irreducible hypersurface
isomorphic to $P^{n-1}(L)$. Equivalently, a hyperplane is defined
by a homogeneous polynomial of degree $1$. In general, let
$F=F_{1}^{n_{1}}\centerdot\ldots F_{j}^{n_{j}}\centerdot\ldots
F_{r}^{n_{r}}$ be the factorisation of $F$ into irreducibles for
$1\leq j\leq r$. We will want to take into account the
"non-reduced" character of $F$ if some $n_{j}\geq 2$. Therefore,
given an irreducible homogeneous polynomial $G$ of degree $m$ and
an integer $s\geq 1$, we will refer to $G^{s}$ as an $s$-fold
component of degree $ms$. Geometrically, we interpret
$G^{s}$ as follows;\\

Consider the space of \emph{all} homogenous polynomials of degree
$ms$ parametrised by $P^{N}(L)$. Let $W\subset P^{N}\times P^{n}$
be the irreducible projective variety defined by $W(\bar x,\bar
y)$ iff $\bar y\in Zero(F_{\bar x})$, where $F_{\bar x}$ is the
homogenous polynomial (defined uniquely up to scalars) by the
parameter $\bar x$. The coefficients of the homogeneous polynomial
$G^{s}$ determine uniquely an element $\bar a$ in $P^{N}$ and we
can consider $G^{s}$ as the fibre $W(\bar a)\subset P^{n}$. In the
Zariski topology, this consists of the variety defined by the
irreducible polynomial $G$. However, in the language ${\mathcal
L}_{spec}$, we can realise its non-reduced nature using the following lemmas;\\

\begin{lemma}

Let $S\subset P^{n}$ be an irreducible hypersurface and $C$ a
projective algebraic curve, then, if $C$ is not contained in $S$,
$S\cap C$ consists of a finite non-empty set of points.

\end{lemma}

\begin{proof}
The proof follows immediately from the Projective Dimension
Theorem, see for example \cite{H}.
\end{proof}

\begin{lemma}

Let $G=0$ define an irreducible hypersurface of degree $m$. Let
$l$ be a generic line, then $l$ intersects $G=0$ in precisely $m$
points.

\end{lemma}

\begin{proof}

Let $Par_{l}=P^{n}\times P^{n}\setminus\Delta$ be the parameter
space for lines in $P^{n}$ as defined above. Let $Par_{m}$ be the
projective parameter space for all homogeneous forms of degree
$m$. Then we can form the variety $W\subset Par_{l}\times
Par_{m}\times P^{n}$ given by;\\

$W(l,x,y)\ \ iff\ \ y\in l\cap Zero(G_{x})$,\\

 where $G_{x}$ is the homogenous polynomial corresponding to the parameter
 $x$. If $l$ is chosen to be generic over the parameters defining $G_{x}$, then $l\cap Zero(G_{x})$ is finite.
 It follows that there exists an open subset $U\subset Par_{l}\times Par_{m}$ consisting of parameters
 $\{l,x\}$ such that $l\cap G_{x}$ has finite
 intersection. As $U$ is smooth, the finite cover $W$ restricted to $U$ is equidimensional and we may
 apply Zariski structure arguments, as was done in \cite{depiro2}.
Considering $W$ as a finite cover of $U$, if $y\in l\cap
Zero(G_{x})$, we define $Mult_{y}(l,G_{x})$ to be
$Mult_{(l,x,y)}(W/U)$ in the sense of Zariski structures. Using
the notation in \cite{depiro2}, we can also define
$LeftMult_{y}(l,G_{x})$ and $RightMult_{y}(l,G_{x})$. In more
geometric language;\\

$LeftMult_{y}(l,G_{x})=Card({\mathcal V}_{y}\cap l'\cap Zero(G_{x}))$\\

where $l'$ is a generic infinitesimal variation of $l$ in the
nonstandard model $P(K)$.\\

$RightMult_{y}(l,G_{x})=Card({\mathcal V}_{y}\cap l\cap Zero(G_{x'}))$\\

where $x'\in {\mathcal V}_{x}$ is generic in $Par_{m}$, considered
as a variety in the nonstandard model $P(K)$.\\

We now claim that there \emph{exists} a line $l$ with exactly $m$
points of intersection with $G_{x_{0}}=G$. If $G_{x_{0}}$ defines
a projective algebraic curve in $P^{2}(L)$, the result follows
immediately from an application of Bezout's theorem and the fact
that there exists a line $l$ having only (algebraically)
transverse intersections with $G_{x_{0}}$. This result was shown
in \cite{depiro2}. Otherwise, we obtain the case that $G_{x_{0}}$
defines a plane projective curve by repeated application of
Bertini's Theorem for Hyperplane Sections, see \cite{N}. Namely,
we can find a $2$-dimensional plane $P$ such that $P\cap
G_{x_{0}}$ defines a plane projective algebraic curve $C_{x_{0}}$.
We may assume that such a plane $P$ is
determined by a generic triple $\{a,b,c\}$. We have an isomorphism $\theta:P^{2}(L)\rightarrow P$ given by;\\

$\theta([Y_{0},Y_{1},Y_{2}])=[{a_{0}\over a_{i}}Y_{0}+{b_{0}\over b_{j}}Y_{1}+{c_{0}\over c_{k}}Y_{2}:\ldots:{a_{n}\over a_{i}}Y_{0}+{b_{n}\over b_{j}}Y_{1}+{c_{n}\over c_{k}}Y_{2}]$\\

If the form $G_{x_{0}}$ is given by $\sum_{i_{0}+\ldots+i_{n}=m}
x_{i_{0}\ldots i_{n}}X_{0}^{i_{0}}\ldots X_{n}^{i_{n}}$, then the
corresponding curve $C_{x_{0}}$ is given in coordinates
$\{Y_{0},Y_{1},Y_{2}\}$ by;\\

$C_{x_{0}}(Y_{0},Y_{1},Y_{2})=\sum_{j_{0}+j_{1}+j_{2}=m}F_{j_{0}j_{1}j_{2}}(a,b,c,\bar
x_{0})Y_{0}^{j_{0}}Y_{1}^{j_{1}}Y_{2}^{j_{2}}$\\

where $F_{j_{0}j_{1}j_{2}}$ is an algebraic function of
$\{a,b,c,\bar x\}$ and, in particular, linear in the variables
$\{\bar x\}$. By the assumption that $C_{x_{0}}$ does not vanish
identically on $P$, we must have that $F_{j_{0}j_{1}j_{2}}$ is not
identically zero for some $(j_{0}j_{1}j_{2})$ with
$j_{0}+j_{1}+j_{2}=m$. It follows that $C_{x_{0}}$ defines a plane
projective curve of degree $m$. Now, by the above argument, we may
find a line $l_{0}\subset P$ having exactly $m$ intersections with
$C_{x_{0}}$. Therefore, $l_{0}$ has exactly $m$
intersections with $G_{x_{0}}$ as well.\\

Now consider the fibre $U(x_{0})=\{l:l\cap G_{x_{0}}\ is\
finite\}$ and restrict the cover $W$ to $U(x_{0})$. By the above
calculation, we have that;\\

$\sum_{y\in l_{0}\cap G_{x_{0}}}Mult_{l_{0},y}(W/U(x_{0}))=\sum_{y\in l_{0}\cap G_{x_{0}}}LeftMult_{y}(l_{0},G_{x_{0}})\geq m\ (*)$\\

By elementary properties of Zariski structures, if $l$ is chosen
to be generic over the parameters defining $G_{x_{0}}$, then $(*)$
holds, with $l$ replacing $l_{0}$, and, moreover,
$LeftMult_{y}(l,G_{x_{0}})=1$ for each intersection $y\in l\cap
G_{x_{0}}$. This implies that $l$ has \emph{at least} $m$ points
of intersection with $G_{x_{0}}=G$. In order to obtain equality,
we now consider the fibre $U(l)=\{x:l\cap G_{x}\ is\ finite\}$ and
restrict the cover $W$ to $U(l)$. It will follow from the result
given in the next section, the Hyperspatial Bezout Theorem, that,
for \emph{any} $x\in U(l)$;\\

$\sum_{y\in l\cap G_{x}}Mult_{x,y}(W/U(l))=\sum_{y\in l\cap
G_{x}}RightMult_{y}(l,G_{x})=deg(l)deg(G_{x})$\\

where, for a projective algebraic curve, degree is given by
Definition 1.12. We clearly have that $deg(l)=1$ and, by
assumption, that $deg(G_{x_{0}})=m$. Hence, we must have that $l$
intersects $G_{x_{0}}$ in exactly $m$ points and, moreover,
$RightMult_{y}(l,G_{x_{0}})=1$ for each intersection $y\in l\cap
G_{x_{0}}$
as well.\\

The lemma is now proved but we can give an algebraic formulation
of the result. Namely, if $l$ is determined by the generic pair
$\{a,b\}$, then we have an isomorphism $\theta:P^{1}(L)\rightarrow
l$ given by;\\

$\theta([Y_{0},Y_{1}])=[{a_{0}\over a_{i}}Y_{0}+{b_{0}\over b_{j}}Y_{1}:\ldots:{a_{n}\over a_{i}}Y_{0}+{b_{n}\over b_{j}}Y_{1}]$\\

If the form $G_{x_{0}}$ is given by $\sum_{i_{0}+\ldots+i_{n}=m}
x_{i_{0}\ldots i_{n}}X_{0}^{i_{0}}\ldots X_{n}^{i_{n}}$, then the
equation of $G_{x_{0}}$ on  $l$  is given in coordinates
$\{Y_{0},Y_{1}\}$ by the homogenous polynomial;\\

$P_{x_{0}}(Y_{0},Y_{1})=\sum_{j_{0}+j_{1}=m}F_{j_{0}j_{1}}(a,b,\bar
x_{0})Y_{0}^{j_{0}}Y_{1}^{j_{1}}$\\

We clearly have, by the same reasoning as above, that $P_{x_{0}}$
has degree $m$. The result of the lemma gives that $P_{x_{0}}$ has
exactly $m$ roots in $P^{1}(L)$. Hence, these roots are all
distinct. It follows that the scheme theoretic intersection of $l$
and $G=0$ consists of $m$ distinct reduced points. That is $l$
intersects $G=0$ algebraically transversely.\\

\end{proof}

We now extend the lemma to include reducible varieties. As before
$G$ is an irreducible algebraic form of degree $m$ and we consider
the power $G^{s}$ for some $s\geq 1$. We will say that $y\in l\cap
G^{s}=0$ is counted $r$-times if $RightMult_{y}(l,G^{s})=r$ in the
sense defined above.

\begin{lemma}

Let $G=0$ define an irreducible hypersurface of degree $m$ and let
$s\geq 1$. Let $l$ be a generic line, then $l$ intersects
$G^{s}=0$ in $m$ points each counted $s$-times.

\end{lemma}

\begin{proof}
We first show that $l$ is generic with respect to $G=0$. Let
$\lambda=\{\lambda_{i}\}$ be the parameters defining $G=0$ and let
$\mu=\{\mu_{j}\}$ be the parameters defining $G^{s}=0$. We claim
that $\lambda$ and $\mu$ are interdefinable, considered as
elements of the projective spaces $Par_{m}$ and $Par_{ms}$, in the
structure $P(L)$, which was considered in \cite{depiro1}. The fact
that $\mu\in dcl(\lambda)$ is clear. Conversely, let $\alpha$ be
an automorphism fixing $\mu$ and let $G_{\alpha(\lambda)}$ be the
algebraic form of degree $m$ obtained from $G=G_{\lambda}$. As
$(G_{\lambda})^{s}=G_{\mu}$, we have that
$(G_{\alpha(\lambda)})^{s}=G_{\mu}$. Hence,
$Zero(G_{\alpha(\lambda)})=Zero(G_{\lambda})$. By the projective
Nullstellenstatz, and the fact that both algebraic forms are
irreducible, we must have that
$<G_{\lambda}>=<G_{\alpha(\lambda)}>$. Hence, there must exist a
unit $U$ in the ring $L[X_{0},\ldots,X_{n}]$ such that
$G_{\lambda}=UG_{\alpha(\lambda)}$. As the only such units are
scalars, we obtain immediately that $\alpha(\lambda)=\lambda$ in
$Par_{ms}$. As $P(L)$ is sufficiently saturated, we obtain that
$\lambda\in dcl(\mu)$. By the previous lemma, we obtain that $l$
intersects $G=0$ in exactly $m$ points, hence $l$ intersects
$G^{s}=0$ in exactly $m$ points as well. Therefore, the first part
of the lemma is shown. We now apply the Hyperspatial Bezout
Theorem, see below, to obtain that;\\

$\sum_{y\in l\cap
G^{s}=0}RightMult_{y}(l,G^{s})=deg(l)deg(G^{s})=ms$\\

Hence, the lemma is shown by proving that for \emph{any} $y\in
l\cap G^{s}=0$, $RightMult_{y}(l,G^{s})\geq s$. As before, we
choose a generic plane $P$ containing $l$, such that
$C_{x_{0}}=P\cap G_{x_{0}}$ defines a projective algebraic curve
of degree $m$. Using an explicit presentation of an isomorphism
$\theta:P^{2}(L)\rightarrow L$, as was done above, we clearly have
that the intersection $(P\cap (G_{x_{0}})^{s}=0)$ consists of the
non-reduced curve $C_{x_{0}}^{s}=0$, which has degree $ms$. We now
apply a result from the paper \cite{depiro2},(Lemma 4.16), which
gives that, for $y\in l\cap C_{x_{0}}$;\\

$RightMult_{y}(l,C_{x_{0}}^{s},Par_{Q_{ms}})=s\centerdot
RightMult_{y}(l,C_{x_{0}},Par_{Q_{ms}})\geq s$\\

where we have use the fact that the parameter defining
$C_{x_{0}}^{s}$ moves in the projective parameter space $Q_{ms}$
defining plane projective curves of degree $ms$ in $P^{2}(L)$. We
now claim that, for $y\in l\cap G_{x_{0}}=0$;\\

$RightMult_{y}(l,G_{x_{0}}^{s})\geq RightMult_{y}(l,C_{x_{0}}^{s},Par_{Q_{ms}})$ $(*)$\\

Recall that, given $G_{x}$ an algebraic form of degree $m$, the
restriction of $G_{x}$ to the plane $P_{abc}$ is given by the
formula;\\

$C_{x}(Y_{0},Y_{1},Y_{2})=\sum_{j_{0}+j_{1}+j_{2}=m}F_{j_{0}j_{1}j_{2}}(a,b,c,\bar
x)Y_{0}^{j_{0}}Y_{1}^{j_{1}}Y_{2}^{j_{2}}$\\

We first claim that each $F_{j_{0}j_{1}j_{2}}$ is \emph{not}
identically zero $(**)$. Let $H_{x}$ be a hyperplane given in
coordinates by;\\

$H_{x}=\sum_{r=0}^{n}x_{i}X_{i}=0$\\

Then the restriction of $H_{x}$ to $P_{abc}$ is given by the plane;\\

$P_{x}=(\sum_{r=0}^{n}{x_{r}a_{r}\over
a_{i}})Y_{0}+(\sum_{r=0}^{n}{x_{r}b_{r}\over
b_{j}})Y_{1}+(\sum_{r=0}^{n}{x_{r}c_{r}\over c_{k}})Y_{2}$\\

By elementary linear algebra, we can find hyperplanes
$\{H_{0},H_{1},H_{2}\}$ whose restriction to $P_{abc}$ define the
planes $\{Y_{0}=0,Y_{1}=0,Y_{2}=0\}$. It follows by direct
calculation that the algebraic form of degree $m$ defined by
$H_{0}^{j_{0}}H_{1}^{j_{1}}H_{2}^{j_{2}}$ restricts to the curve
of degree $m$ defined by
$Y_{0}^{j_{0}}Y_{1}^{j_{1}}Y_{2}^{j_{2}}=0$. Hence, $(**)$ is
shown. Now consider the function;\\

$\overline F=\{F_{j_{0}j_{1}j_{2}}\}:Par_{m}\rightarrow
Par_{Q_{m}}$\\

By earlier remarks, the algebraic function $\overline F$ is linear
in the variables $\{x_{i}\}$ and defined over $\{a,b,c\}$. Hence,
its image defines a plane $P\subset Par_{Q_{m}}$. By the above
calculation, this plane $P$ contains the linearly independent set
$\{p_{j_{0}j_{1}j_{2}}:j_{0}+j_{1}+j_{2}=m\}$ where
$C_{p_{j_{0}j_{1}j_{2}}}$ defines the curve
$Y_{0}^{j_{0}}Y_{1}^{j_{1}}Y_{2}^{j_{2}}=0$. Hence, $\overline F$
is surjective. Moreover, by elementary facts about linear maps,
the fibres of $\overline F$ are equidimensional. We can then show
$(*)$. Suppose that;\\

$RightMult_{y}(l,C_{x_{0}}^{s},Par_{Q_{ms}})=k$\\

Then one can find $x'\in{\mathcal V}_{x_{1}}\cap Par_{Q_{ms}}$,
generic over the parameter $x_{1}$ defining $C_{x_{0}}^{s}$, such
that $C_{x'}$ intersects $l$ in the distinct points
$\{y_{1},\ldots,y_{k}\}\subset{\mathcal V}_{y}$. As $x_{1}$ is
regular for the cover $\overline F$, if $x_{2}$ defines
$G_{x_{0}}^{s}$, we can find $x''\in{\mathcal V}_{x_{2}}\cap
Par_{ms}$, such that $\overline{F}(x'')=x'$. The algebraic form
defined by $G_{x''}$ then intersects the plane $P_{abc}$ in the
curve $C_{x'}$, hence it must intersect the line $l$ in the
distinct points $\{y_{1},\ldots,y_{k}\}\subset{\mathcal V}_{y}$ as
well. This implies that;\\

$RightMult_{y}(l,G_{x_{0}}^{s})\geq k$\\

Hence $(*)$ is shown. The lemma is then proved. In this lemma we
have not shown anything interesting algebraically. Namely, if one
considers the restriction of $G_{x_{0}}^{s}$ to $l$, we obtain the
homogeneous polynomial $P_{x_{0}}^{s}$. By the previous lemma,
$P_{x_{0}}$ has $m$ distinct roots in $P^{1}(L)$, hence
$P_{x_{0}}^{s}$ has $m$ distinct roots with multiplicity $s$.
Therefore, the scheme theoretic intersection of $l$ with
$G_{x_{0}}^{s}$ consists of $m$ distinct copies of the non-reduced
scheme $L[t]/(t)^{s}$. The usefulness of the result will be shown
in the following lemmas.
\end{proof}

\begin{rmk}
Note that the latter part of the argument in fact shows that, for
\emph{any} line intersecting $G_{x_{0}}^{s}$ in finitely many
points, we must have that each point of intersection $y$ is
counted at least $s$-times.

\end{rmk}

\begin{lemma}
Let $F=0$ define a hypersurface of degree $k$. Let
$F=F_{1}^{n_{1}}\centerdot\ldots\centerdot
F_{j}^{n_{j}}\centerdot\ldots\centerdot F_{r}^{n_{r}}$ be its
factorisation into irreducibles, with $degree(F_{j})=m_{j}$. Then
there exists a line $l$, intersecting each component $F_{j}$ in
exactly $m_{j}$ points, each counted $n_{j}$ times, with the
property that the sets $\{(F_{j}\cap l):1\leq j\leq r\}$ are
pairwise disjoint. Moreover, the set of lines having this property
form a Zariski open subset of $Par_{l}$, defined over the
parameters of $F=0$.

\end{lemma}

\begin{proof}
Let $(x_{1},\ldots,x_{r})\in Par_{m_{1}}\times\ldots\times
Par_{m_{r}}$ be the tuple defining each reduced irreducible
component of $F=0$. By an elementary argument, extending the proof
in Lemma 1.6, the tuple is interalgebraic with the tuple $x\in
Par_{k}$ defining the hypersurface $F=0$. Let
$\theta_{(x_{j},m_{j},n_{j})}(y)\subset Par_{l}$ be the statement
that a line $l_{y}$ intersects $F_{j}$ in exactly $m_{j}$
points, each counted $n_{j}$ times;\\

$\exists_{z_{1}\neq\ldots\neq z_{m_{j}}}[\bigwedge_{1\leq i\leq
m_{j}}z_{i}\in l_{y}\cap F_{j,x_{j}}\wedge
RightMult_{z_{i}}(l_{y},F_{j,x_{j}})=n_{j}\wedge.\ \ \forall
w(w\in l_{y}\cap F_{j,x_{j}}\rightarrow\bigvee_{1\leq i\leq
m_{j}}z_{i}=w)]$\\

By definability of multiplicity in Zariski structures, Lemmas 1.5
and Lemmas 1.6, each $\theta_{(x_{j},m_{j},n_{j})}(y)$ is
definable over $x_{j}$ and is a Zariski dense algebraic subset of
$Par_{l}$. By the previous remark, the complement of
$\theta_{(x_{j},m_{j},n_{j})}(y)$ in the set of lines having
finite intersection with $F_{j}$
is given by;\\

$\exists w[w\in l_{y}\cap F_{j,x_{j}}\wedge
RightMult_{w}(l_{y},F_{j,x_{j}})\geq n_{j}+1]$\\

It follows that $\theta_{(x_{j},m_{j},n_{j})}(y)$ defines a
Zariski open subset of $Par_{l}$. Now let;\\

$\theta(y)=\bigwedge_{1\leq j\leq
r}\theta_{(x_{j},m_{j},n_{j})}(y)$\\

Then $\theta(y)$ defines a Zariski open subset of $Par_{l}$.
Finally, let $W$ be the the union of the pairwise intersections of
the irreducible components $F_{j}$. Then, by elementary dimension
theory, $W$ is Zariski closed of dimension at most $n-2$. Hence,
the condition on $Par_{l}$ that a line passes through $W$ defines
a proper closed set over the parameters $(x_{1},\ldots,x_{r})$.
Let $U(y)$ be the Zariski open complement of this set in
$Par_{l}$. Then any line $l$ satisfying $\theta(y)\wedge U(y)$ has
the properties required of the lemma.

\end{proof}

\begin{defn}
We will call a line $l$ satisfying the conclusion of the lemma
\emph{transverse} to $F$.

\end{defn}

We will now give an alternative characterisation of
transversality.

\begin{lemma}
Let $F=0$ define a hypersurface of degree $k$. Then a line $l$ is
transverse to $F$ iff $l$ intersects $F$ in finitely many points
and, for each $y\in l\cap (F=0)$, $LeftMult_{y}(l,F)=1$. Moreover,
the notion of transversality may be formulated by a predicate in
the language ${\mathcal L}_{spec}$, $Transverse_{k}\subset
Par_{k}\times Par_{l}$;\\

$Transverse_{k}(\lambda,l)$ iff $l$ is transverse to $F_{\lambda}$\\
\end{lemma}

\begin{proof}
The first part of the proof is clear using previous results of
this section. For the second part, use the results in Section 3 of
\cite{depiro3}.
\end{proof}

We now have;\\

\begin{lemma}{A Nullstellensatz for Non-Reduced Hypersurfaces}\\

Let $F=0$ define a hypersurface of degree $k$ and let
$Zero(F)=Zero(F_{\lambda_{1}})\cup\ldots\cup
Zero(F_{\lambda_{r}})$ be its geometric factorisation into
irreducibles, using the Zariski topology. Let
$\sigma(\lambda_{j},m_{j},n_{j})\subset Par_{m_{j}}$, for $1\leq
j\leq r$, be the predicates defined in ${\mathcal L}_{spec}$ by;\\

A transverse line to $F_{\lambda_{j}}$ intersects
$F_{\lambda_{j}}$ in
exactly $m_{j}$ points, each counted $n_{j}$ times.\\

Then the original homogeneous polynomial $F$ is characterised, up
to scalars, by the sequence;\\

$(Zero(F_{\lambda_{1}}),\ldots,Zero(F_{\lambda_{r}}),\sigma(\lambda_{1},m_{1},n_{1}),\ldots,\sigma(\lambda_{r},m_{r},n_{r}))$

\end{lemma}

\begin{proof}
By the proof of previous results from this section, the formulae
$\sigma(\lambda_{j},m_{j},n_{j})$, for $1\leq j\leq r$, determine
the multiplicity of each component $F_{\lambda_{j}}$. The result
then follows by uniqueness of factorisation.
\end{proof}

We will refer to a hypersurface as generic if the parameter
defining it is generic in the parameter space of all hypersurfaces
of the same degree.

\begin{defn}

The degree of a projective algebraic curve $C$ is the number of
intersections with a generic hyperplane.

\end{defn}

We need to check this is a good definition. Let $(P^{n})^{*}$ be
the dual space of $P^{n}$. $(P^{n})^{*}$ is the parameter
space for all hyperplanes $H$ in $P^{n}$. We have that;\\

 $\{a\in (P^{n})^{*}:dim(C\cap H_{a})\geq 1\}$\\

  is closed in $(P^{n})^{*}$, hence, for generic $a\in (P^{n})^{*}$,
  $C\cap H_{a}$ is finite (and non-empty). Choosing \emph{some} generic $a$ in $(P^{n})^{*}$,
  let $m$ be the number of intersections of $H_{a}$ with $C$. Let $\theta(x)\subset
(P^{n})^{*}$ be the statement;\\

$\exists_{x_{1}\neq\ldots\neq x_{m}}(\bigwedge_{1\leq i\leq
m}x_{i}\in C\cap H_{x}\wedge \forall y(y\in C\cap H_{x}\rightarrow
\bigvee_{1\leq i\leq m}x_{i}=y))$\\

$\theta(x)$ is algebraic and defined over $\emptyset$, hence, as
it contains $a$, must be Zariski dense in $(P^{n})^{*}$. In
particular, it contains \emph{any} generic $a$ in $(P^{n})^{*}$.\\

\begin{defn}

The degree of a hypersurface $F$ is the degree of the homogenous
polynomial defining it. (See the above remarks)

\end{defn}

In the following paper, the notion of \emph{birationality} between
projective algebraic curves, will be central.

\begin{defn}

We define a linear system $\Sigma$ on $P^{r}$ to be the collection
of algebraic forms of degree $k$, for some $k\geq 1$,
corresponding to a plane, which we will denote by $Par_{\Sigma}$,
contained in $Par_{k}$, the parameter space of homogeneous
polynomials of degree $k$. If $Par_{\Sigma}$ has dimension $n$, we
define a basis of $\Sigma$ to be an ordered set of $n+1$ forms
corresponding to a maximally independent set of parameters in
$Par_{\Sigma}$. Equivalently, a basis of $\Sigma$ is an ordered
system;\\

$\{\phi_{0}(X_{0},\ldots,X_{r}),\ldots,\phi_{n}(X_{0},\ldots,X_{r})\}$\\

of homogeneous polynomials of degree $k$ belonging to $\Sigma$
which are independent, that is there do not exist
parameters $\{\lambda_{0},\ldots,\lambda_{n}\}$ such that;\\

$\lambda_{0}\phi_{0}+\ldots +\lambda_{n}\phi_{n}\equiv 0$\\

\end{defn}

\begin{defn}

Given a linear system $\Sigma$ of dimension $n$ on $P^{r}$, we define the base locus of the system $\Sigma$ by;\\

 $Base(\Sigma)=\{\bar
x\in P^{r}: \phi_{0}(\bar x)=\ldots=\phi_{n}(\bar x)=0\}$\\

for any basis of $\Sigma$. Given any $2$ bases
$\{\phi_{0},\ldots,\phi_{j},\ldots,\phi_{n}\}$ and\\
$\{\psi_{0},\ldots,\psi_{i},\ldots,\psi_{n}\}$ of $\Sigma$, we can
find an invertible matrix of scalars
$(\lambda_{ij})_{0\leq i,j\leq n}$ such that;\\

$\psi_{i}=\sum_{j=0}^{n} \lambda_{ij}\phi_{j}$\\

Hence, the base locus of $\Sigma$ is well defined. As a basis
corresponds to a maximally independent sequence in $Par_{\Sigma}$,
we could equivalently define;\\

$Base(\Sigma)=\{\bar x\in P^{r}: \phi(\bar x)=0\}$\\

for \emph{every} algebraic form $\phi$ belonging to $\Sigma$.

\end{defn}

\begin{lemma}

Let $\Sigma$ be a linear system of dimension $n$ and degree $k$ on
$P^{r}$. Then, a choice of basis $B$ for $\Sigma$ defines a
morphism $\Phi_{\Sigma,B}:{P^{r}\setminus Base(\Sigma)}\rightarrow
P^{n}$ with the property that $Image(\phi_{\Sigma,B})$ is not
contained in any hyperplane section of $P^{n}$. Moreover, given
any $2$ bases $\{B,B'\}$ of $\Sigma$, there exists a homography
$\theta_{B,B'}:P^{n}\rightarrow P^{n}$ such that;\\

$\Phi_{\Sigma,B'}=\theta_{B,B'}\circ\Phi_{\Sigma,B}$\\

\end{lemma}

\begin{proof}
Given a choice of basis $B=\{\phi_{0},\ldots,\phi_{n}\}$ for $\Sigma$, one checks that the map defined by;\\

$\Phi_{\Sigma,B}([X_{0}:\ldots:X_{r}])=[\phi_{0}(X_{0},\ldots,X_{r}):\ldots:\phi_{n}(X_{0},\ldots,X_{r})]$\\

is a morphism with the required properties. Now suppose that
$\{\phi_{0},\ldots,\phi_{j},\ldots,\phi_{n}\}$ and
$\{\psi_{0},\ldots,\psi_{i},\ldots,\psi_{n}\}$ are $2$ bases $B$
and $B'$ for $\Sigma$. Let $(\lambda_{ij})_{0\leq i,j\leq n}$ be
the matrix of scalars as given in Definition 1.15. Then one can
define a homography $\theta_{B,B'}$ by;\\

$\theta_{B,B'}([Y_{0}:\ldots:Y_{n}])=[\sum_{j=0}^{n}\lambda_{0j}Y_{j}:\ldots:\sum_{j=0}^{n}\lambda_{ij}Y_{j}:\ldots:\sum_{j=0}^{n}\lambda_{nj}Y_{j}]$\\

It is clear that this homography has the required property of the
lemma.

\end{proof}

\begin{defn}
We define a rational map from $P^{r}$ to $P^{n}$ to be a morphism
defined by a choice of basis for a linear system $\Sigma$.
\end{defn}

\begin{rmk}
Given a linear system $\Sigma$, we will generally refer to a
morphism given by Lemma 1.16 as simply $\Phi_{\Sigma}$, leaving
the reader to remember that a choice of basis is involved. As any
$2$ such choices differ by a homography, any properties of one
morphism transfer directly to the other, so one hopes that this
terminology will not cause confusion. More geometrically, observe
that, if $x\in P^{r}\setminus Base(\Sigma)$, then the set of
algebraic forms in $\Sigma$, vanishing at $x$, defines a
hyperplane $H_{x}\subset\Sigma$. A choice of basis
$\{\phi_{0},\ldots,\phi_{n}\}$ for $\Sigma$ identifies this
hyperplane $H_{x}$ with a point $[\phi_{0}(x),\ldots,\phi_{n}(x)]$
of the dual space $P^{n*}$.

\end{rmk}

\begin{defn}
We say that two projective algebraic curves $C_{1}$ and $C_{2}$
are birational if there exists $U\subset C_{1}$ and $V\subset
C_{2}$, with $U$ and $V$ open in $C_{1}$ and $C_{2}$ respectively,
such that $U$ and $V$ are isomorphic as algebraic varieties. We
will use the notation $\Phi:C_{1}\leftrightsquigarrow C_{2}$ for a
birational map.
\end{defn}

We will require the following presentation of birational maps;\\

\begin{lemma}
Let $C_{1}\subset P^{r}$ and $C_{2}\subset P^{n}$ be birational
projective algebraic curves, as in Definition 1.16, with the
property that no hyperplane section of $P^{r}$ or $P^{n}$ contains
$C_{1}$ or $C_{2}$ respectively. Then we can find linear systems
$\{\Sigma,\Sigma'\}$, rational maps;\\

$\ \ \phi_{\Sigma}:{P^{r}\setminus Base(\Sigma)}\rightarrow P^{n}\
\ \ \ \ \ \ \ \ \ \ \ \phi_{\Sigma'}:{P^{n}\setminus
Base(\Sigma')}\rightarrow
P^{r}$\\

and open subsets $\{U',V'\}$ of $\{C_{1},C_{2}\}$, which are
disjoint from\\
 $\{Base(\Sigma),Base(\Sigma')\}$, such that the
restrictions $\phi_{\Sigma}:U'\rightarrow V'$ and
$\phi_{\Sigma'}:V'\rightarrow U'$ are (inverse) isomorphisms.

\end{lemma}

\begin{proof}

As usual, let $[X_{0}:\ldots:X_{r}]$ and $[Y_{0}:\ldots:Y_{n}]$ be
homogeneous coordinates for $P^{r}$ and $P^{n}$ respectively.
Taking the hyperplanes $X_{0}=0$ and $Y_{0}=0$, we can find affine
presentations;\\

${L[x_{1},\ldots,x_{r}]\over J_{1}}=R({C_{1}\setminus C_{1}\cap
(X_{0}=0)})=R(U'')$\\

${L[y_{1},\ldots,y_{n}]\over J_{2}}=R({C_{2}\setminus C_{2}\cap
(Y_{0}=0)})=R(V'')$\\

where $\{U'',V''\}$ are open subsets of $\{C_{1},C_{2}\}$,
$\{J_{1},J_{2}\}$ are prime ideals.\\

We can then find $U'\subset U''\cap U$ and $V'\subset V''\cap V$
such that $U'$ and $V'$ are isomorphic as algebraic subvarieties
of $C_{1}$ and $C_{2}$ (consider the elements of $U$ which are
mapped to $V''$ by the original isomorphism). Now choose
polynomials
$F(\bar x)$ and $G(\bar y)$ such that;\\

${L[x_{1},\ldots,x_{r}]_{F}\over J_{1}'}=R(U')\ \ \ \ \ \ \ \ {L[y_{1},\ldots,y_{n}]_{G}\over J_{2}'}=R(V')$\\

As $R(U')\cong R(V')$, we can find rational functions
$\{\phi_{1}(\bar x),\ldots,\phi_{n}(\bar x)\}$ and
$\{\psi_{1}(\bar y),\ldots,\psi_{r}(\bar y)\}$ (with denominators
powers of $F$ and $G$ respectively) defining morphisms;\\

$\Phi:{A^{r}\setminus \{F=0\}}\rightarrow A^{n}\ \ \ \ \ \ \Psi:{A^{n}\setminus \{G=0\}}\rightarrow A^{r}$\\

and representing the isomorphism $U'\cong V'$. We now show how to
convert $\Phi$ into $\Phi_{\Sigma}$. By equating denominators, we
are able to write\\
 $\{\phi_{1}(\bar x),\ldots,\phi_{n}(\bar x)\}$ as
$\{{p_{1}(\bar x)\over q(\bar x)},\ldots,{p_{n}(\bar x)\over
q(\bar x)}\}$. Now make the substitutions $x_{i}={X_{i}\over
X_{0}}$ in $\{p_{1}(\bar x),\ldots,p_{n}(\bar x),q(\bar x)\}$ and
multiply
 through by the highest power of $X_{0}$ to obtain homogeneous
 polynomials of the same degree\\
  $\{P_{1}(\bar X),\ldots,P_{n}(\bar
 X),Q(\bar X)\}$. Let $\Sigma$ be the linear system defined by the plane spanned by
these homogeneous polynomials and define $\Phi_{\Sigma}$ by;\\

 $\rho Y_{0}=Q(X_{0},\ldots,X_{r}),\rho Y_{1}=P_{1}(X_{0},\ldots,X_{r}),\ldots,\rho Y_{n}=P_{n}(X_{0},\ldots,X_{r})$\\

where $\rho$ is a constant of proportionality. (This is an
alternative notation for a map between projective spaces, used
frequently in papers by the Italian geometers Castelnouvo,
Enriques and Severi). By the assumption on $\{C_{1},C_{2}\}$
concerning hyperplane sections, the homogeneous polynomials
$\{Q,P_{1},\ldots,P_{n}\}$ form a basis for $\Sigma$, hence this
defines a rational map. Similarily, one can find a linear system
$\Sigma'$ and convert $\Psi$ into a rational map $\phi_{\Sigma'}$.
The rest of the properties of the lemma follow immediately from
the construction.
\end{proof}

We should also note the following equivalent criteria for
birationality of projective algebraic curves;\\

\begin{lemma}
Let $C_{1}$ and $C_{2}$ be projective algebraic curves. Then
$C_{1}$ and $C_{2}$ are birational iff;\\

(i). There is an isomorphism of function fields $L(C_{1})\cong L(C_{2})$.\\
(ii). (In characteristic $0$) There exist $a_{1}$ generic in
$C_{1}$, $a_{2}$ generic in $C_{2}$ and an algebraic relation
$Rational(x,y)$ such that $a_{1}\in dcl_{Rational}(a_{2})$ and
$a_{2}\in dcl_{Rational}(a_{1})$.\\
.\ \ \ \ (In charateristic $p$) One can use the same criteria but
must pay attention to the presence of the Frobenius map, see the
     paper \cite{depiro1} for details on how to resolve this issue.\\

\end{lemma}

\begin{defn}
Let $\Phi:C_{1}\leftrightsquigarrow C_{2}$ be a birational map, as
in Definition 1.19. We define the correspondence
$\Gamma_{\Phi}\subset C_{1}\times C_{2}$
associated to $\Phi$ to be;\\

$\overline{(Graph(\Phi)\subset U\times V)}$\\

where, for $W$ an algebraic subset of $C_{1}\times C_{2}$, we let
$\bar W$ denote its Zariski closure.

\end{defn}

\begin{defn}
Let $C_{1}$ and $C_{2}$ be projective algebraic curves. We will
say that $2$ birational maps $\Phi_{1}:C_{1}\leftrightsquigarrow
C_{2}$ and $\Phi_{2}:C_{1}\leftrightsquigarrow C_{2}$ are
equivalent if there exists $U\subset C_{1}$ such that $\Phi_{1}$
and $\Phi_{2}$ are both defined and agree on $U$. Clearly,
equivalence of birational maps is an equivalence relation.
\end{defn}

\begin{lemma}
Let $\Phi_{1}:C_{1}\leftrightsquigarrow C_{2}$ and
$\Phi_{2}:C_{1}\leftrightsquigarrow C_{2}$ be equivalent
birational maps, then $\Gamma_{\Phi_{1}}=\Gamma_{\Phi_{2}}$.
\end{lemma}

\begin{proof}
Immediate from the definitions.

\end{proof}

\begin{defn}
We will denote the equivalence class of a birational map $\Phi$ by
$[\Phi]$. By the above lemma, we can associate a correspondence
$\Gamma_{[\Phi]}$ to an equivalence class of birational maps
\end{defn}

\begin{lemma}{Obstruction to Birationality at Singular Points}\\

Let $\Gamma_{[\Phi]}$ be a birational correspondence between
$C_{1}$ and $C_{2}$. If $x$ is a non-singular point of $C_{1}$,
there exists a unique corresponding point $y$ of $C_{2}$ and
vice-versa.

\end{lemma}

\begin{proof}
Let $U\subset C_{1}$ be the set of non-singular points of $C_{1}$.
We can consider $\Gamma_{\Phi}$ as a cover of $U$. As $U$ is
smooth, we may apply the technique of Zariski structures for this
cover. By birationality, if $x\in U$ is generic, there exists a
unique $(xy)\in \Gamma_{\Phi}$ and, moreover, $Mult(y/x)=1$. This
last fact was shown, for example, in the paper \cite{depiro1} and
given originally in \cite{Z}. By further properties of Zariski
structures, again see either of the above, the total multiplicity
of points in the cover over $U$ is preserved, in particular, for
\emph{any} $x\in U$, there exists a unique corresponding
$(xy)\in\Gamma_{\Phi}$.
\end{proof}

\begin{rmk}
Note that non-singularity is not necessarily preserved when
associating $y$ to $x$ in the above lemma. This motivates the
following definition.
\end{rmk}

\begin{defn}
Given a birational correspondence $\Gamma_{[\Phi]}$, we define the
canonical sets $U_{[\Phi]}\subset
\Gamma_{[\Phi]}$,$V_{[\Phi]}\subset
C_{1}$ and $W_{[\Phi]}\subset C_{2}$ to be the sets;\\

$U_{[\Phi]}=\{(x,y)\in\Gamma_{[\Phi]}:NonSing(x),NonSing(y)\}$\\

$V_{[\Phi]}=\pi_{1}(U_{[\Phi]})$\\

$W_{[\Phi]}=\pi_{2}(U_{[\Phi]})$\\

\end{defn}

\begin{lemma}
Given a birational correspondence $\Gamma_{[\Phi]}$, there exists
an isomorphism $\Phi_{1}:V_{[\Phi]}\rightarrow W_{[\Phi]}$ such
that $\Gamma_{[\Phi]}=\Gamma_{[\Phi_{1}]}$
\end{lemma}
\begin{proof}
By an elementary result in algebraic geometry, see for example
\cite{H}, a morphism $\Phi:U\subset C_{1}\rightarrow P^{n}$, where
$U$ is an open subset of $C_{1}$, extends uniquely to the
non-singular points of $C_{1}$. Combining this with Lemma 1.26, we
obtain immediately the result.
\end{proof}

We now need to relate the canonical sets of a birational
correspondence $\Gamma_{[\Phi]}$ with a \emph{particular}
presentation of $[\Phi]$ given by Lemma 1.20;\\

\begin{defn}
Let $\phi_{\Sigma}$ be as in Lemma 1.20, we define the canonical
open sets associated to $\phi_{\Sigma}$ to be;\\

$V_{\phi_{\Sigma}}={V_{[\Phi]}\setminus Base(\Sigma)}$\\

$W_{\phi_{\Sigma}}=\phi_{\Sigma}(V_{\phi_{\Sigma}})$\\

\end{defn}

\begin{lemma}
We have the following relations between canonical sets;\\

$V_{\phi_{\Sigma}}\subset V_{[\Phi]}\subset NonSing(C_{1})$\\

$W_{\phi_{\Sigma}}\subset W_{[\Phi]}\subset NonSing(C_{2})$\\

$\phi_{\Sigma}:V_{\phi_{\Sigma}}\rightarrow W_{\phi_{\Sigma}}$ is
an isomorphism.\\

$V_{\phi_{\Sigma}}$ and $Base(\Sigma)$ are disjoint.

\end{lemma}

\begin{proof}
The proof is an easy exercise.
\end{proof}

\begin{rmk}
It would be desirable to find a particular presentation
$\Phi_{\Sigma}$ of a birational class $[\Phi]$, for which
$Base(\Sigma)$ is disjoint from the canonical set $V_{[\Phi]}$. In
general, one can easily prove the weaker result;\\

There exist $2$ presentations $\Phi_{\Sigma_{1}}$ and
$\Phi_{\Sigma_{2}}$ of a birational class $[\Phi]$, such that;\\

$V_{[\Phi]}=V_{\Phi_{\Sigma_{1}}}\cup V_{\Phi_{\Sigma_{2}}}$\\

$W_{[\Phi]}=W_{\Phi_{\Sigma_{1}}}\cup W_{\Phi_{\Sigma_{2}}}$\\

We also note that the choice of $\Sigma$ presenting a birational
class $[\Phi]$ is far from unique. For example, let
$Id:P^{1}\rightarrow P^{1}$ be the identity map. This isomorphism
can be represented by any of the birational maps;\\

$\phi_{n}[X_{0}:X_{1}]=[X_{0}^{n}:X_{0}^{n-1}X_{1}]$ $(n\geq 1)$\\

Let $\Sigma_{n}$ be the linear system of dimension $1$ and degree
$n$ defined by the pair of homogeneous polynomials
$\{X_{0}^{n},X_{0}^{n-1}X_{1}\}$. Then
$\phi_{n}=\Phi_{\Sigma_{n}}$ (with this choice of basis).

\end{rmk}

We finally note the following well known theorem, see for example \cite{H};\\

\begin{theorem}
Let $C$ be a projective algebraic curve, then $C$ is birational to
a plane projective algebraic curve.
\end{theorem}
\end{section}
\begin{section}{A Basic Theory of $g_{n}^{r}$}

We begin this section with the definition of intersection
multiplicity used by the original Italian school of algebraic
geometry.\\

\begin{defn}

Let $C\subset P^{w}$ be a projective algebraic curve. Let
$Par_{F}$ be the projective parameter space for \emph{all}
hypersurfaces of a given degree $e$ and let $U=\{\lambda\in
Par_{F}:|C\cap F_{\lambda}|<\infty\}$ be the open subvariety of
$Par_{F}$ corresponding to hypersurfaces of degree $e$ having
finite intersection with $C$. For $\lambda\in U$, $p\in C\cap
F_{\lambda}$, we define;\\

$I_{italian}(p,C,F_{\lambda})=Card(C\cap F_{\lambda'}\cap
{\mathcal V}_{p})$\
for $\lambda'\in {\mathcal V}_{\lambda}$ generic in $U$.\\

\end{defn}

\begin{rmk}
That this \emph{is} a rigorous definition follows from general
properties of Zariski structures. The definition is the same as
$RightMult_{p}(C,F_{\lambda})$ which we considered in the previous
section. We will often abbreviate the notation
$I_{italian}(p,C,F_{\lambda})=s$ by saying that $p$ is
\emph{counted} $s$ times for the intersection of $C$ with
$F_{\lambda}$.\\
\end{rmk}

The basic theory of $g_{n}^{r}$ relies principally on the following result;\\

\begin{theorem}{Hyperspatial Bezout}\\

Let $C\subset P^{w}$ be a projective algebraic curve of degree $d$
and $F_{\lambda}$ a hypersurface of degree $e$ having finite
intersection with $C$. Then;\\

$\sum_{p\in C\cap F_{\lambda}}I_{italian}(p,C,F_{\lambda})=de$\\

\end{theorem}

We first require the following lemma, preserving
the notation from Definition 2.1;\\

\begin{lemma}
Let $H_{\lambda}$ be a generic hyperplane, then;\\

$I_{italian}(p,C,H_{\lambda})=1$ for \emph{all} $p\in C\cap H_{\lambda}$\\

and each $p\in C\cap H_{\lambda}$ is non-singular.\\

 In the Italian terminology, each point $p$ of intersection is counted
 \emph{once} or the intersection is transverse. Using the methods developed
 in Section 1, it is not difficult to prove that each point of intersection is transverse
 (using the scheme theoretic definition).

\end{lemma}

\begin{proof}

Suppose, for contradiction, that
$I_{italian}(p_{1},C,H_{\lambda})\geq 2$ for some $p_{1}\in C\cap
H_{\lambda}$. Let $\{p_{1},\ldots,p_{d}\}$ be the total set of
intersections, where $d$ is the degree of $C$, see Definition
1.12. Then we can find $\lambda'\in {\mathcal V}_{\lambda}$
generic in $U$ and a distinct pair $\{p_{1}^{1},p_{1}^{2}\}$ in
${\mathcal V}_{p_{1}}\cap C\cap H_{\lambda'}$. By properties of
Zariski structures, see \cite{depiro1} or \cite{Z}, we can also
find $\{p_{2}^{1},\ldots,p_{d}^{1}\}$ such that $p_{j}^{1}\in
C\cap H_{\lambda'}\cap {\mathcal V}_{p_{j}}$ for $2\leq j\leq d$.
It follows from the definition of an infinitesimal neighborhood
that, for $p\neq q$ with $\{p,q\}\subset P^{w}$, ${\mathcal
V}_{p}$ and ${\mathcal V}_{q}$ are disjoint. Hence,
$\{p_{1}^{1},p_{1}^{2},p_{2}^{1},\ldots,p_{d}^{1}\}$ defines a
distinct set of intersections of $C$ with $H_{\lambda'}$. It
follows that $C$ and $H_{\lambda'}$ have at least $d+1$
intersections, contradicting the facts that $\lambda'$ is generic
in $U$ and the degree of $C$ is equal to $d$. For the second part
of the lemma, observe that the set of nonsingular points
$NonSing(C)$ is a dense open subset of $C$, defined over the field
of definition of $C$. The condition that a hyperplane passes
through at least one point of $({C\setminus NonSing(C)})$ is
therefore a union of finitely many proper hyperplanes
$P_{1}\cup\ldots\cup P_{m}$ contained in $Par_{H}$, also defined
over the field of definition of $C$. As $H_{\lambda}$ was chosen
to be generic, its parameter $\lambda$ cannot lie inside
$P_{1}\cup\ldots\cup P_{m}$. Hence, the result is shown.

\end{proof}

We now complete the proof of Theorem 2.3;\\

\begin{proof}

Choose $\{\lambda_{1},\ldots,\lambda_{e}\}$ independent generic
tuples in $P^{w*}=Par_{H}$, the parameter space for hyperplanes on
$P^{w}$. Let $F_{e}$ be the form of degree $e$ defined
by;\\

$F_{e}=H_{\lambda_{1}}\centerdot\ldots\centerdot
H_{\lambda_{e}}=\Sigma_{i_{0}+\ldots+i_{w}=e}\lambda_{i_{0}\ldots i_{n}}Y_{0}^{i_{0}}\ldots Y_{w}^{i_{w}}=0$\\

We first claim that the intersections $C\cap H_{\lambda_{j}}$ are
pairwise disjoint sets for $1\leq j\leq e$. The condition that a
hyperplane $H_{\mu}$ passes through at least one point of the
intersection $\{(C\cap H_{\lambda_{1}})\cup\ldots\cup (C\cap
H_{\lambda_{j}})\}$ is a union of finitely many proper closed
hyperplane conditions on $Par_{H}$, defined over the parameters of
$C$ and the tuple $\{\lambda_{1},\ldots,\lambda_{j}\}$, for $1\leq
j\leq e-1$. As the tuples $\{\lambda_{1},\ldots,\lambda_{e}\}$
were chosen to be generically independent in $Par_{H}$, the result
follows. Now, by Lemma 2.4 and the definition of $F_{e}$, we
obtain a total number $de$ of intersections between $C$ and
$F_{e}$. We claim that for each point $p$ of intersection;\\

$I_{italian}(p,C,F_{e})=1$ $(*)$\\

This does not follow immediately from Lemma 2.4 as the parameter
$\{\lambda_{i_{0}\ldots i_{n}}\}$ defining $F_{e}$ is allowed to
vary in the parameter space $Par_{e}$ of \emph{all} forms of
degree $e$. We prove the claim by reducing the problem to one
about plane projective curves.\\

We use Lemma 1.20 and Theorem 1.33 to find a plane projective
algebraic curve $C_{1}\subset P^{2}$ and a linear system $\Sigma$
such that $\Phi_{\Sigma}:C_{1}\leftrightsquigarrow C$. Let
$\{\phi_{0},\ldots,\phi_{w}\}$ be a basis for $\Sigma$, defining
the birational map $\Phi_{\Sigma}$. We may suppose that each
$\phi_{i}$ is homogenous of degree $k$ in the variables
$\{X_{0},X_{1},X_{2}\}$ for $P^{2}$. Let
$\{V_{[\Phi]},V_{\Phi_{\Sigma}},W_{[\Phi]},W_{\Phi_{\Sigma}}\}$ be
the canonical sets associated to $\Gamma_{\Phi_{\Sigma}}$, see
Definitions 1.28 and 1.30. Note that the canonical sets are all
definable over the data of $\Phi_{\Sigma}$. Hence, we may, without
loss of generality, assume that the point $p$ given in $(*)$ above
lies in $W_{\Phi_{\Sigma}}$ and its corresponding $p'\in C_{1}$
lies in $V_{\Phi_{\Sigma}}$. In particular, $p'$ defines a
non-singular point of the curve $C_{1}$. Now, given an algebraic
form $F_{\mu}$ of degree $e$;\\

$F_{\mu}=\Sigma_{i_{0}+\ldots+i_{w}=e}\mu_{i_{0}\ldots
i_{n}}Y_{0}^{i_{0}}\ldots Y_{w}^{i_{w}}=0$\\

we obtain a corresponding algebraic curve $\psi_{\mu}$ of degree
$ke$ on $P^{2}$ given by the equation;\\

$\psi_{\mu}=\Sigma_{i_{0}+\ldots+i_{w}=e}\mu_{i_{0}\ldots
i_{n}}\phi_{0}^{i_{0}}\ldots \phi_{w}^{i_{w}}=0$ $(\dag)$\\

We claim that;\\

$I_{italian}(p,C,F_{\mu})\leq I_{italian}(p',C_{1},\psi_{\mu})$ $(**)$\\

For suppose that $p$ is counted $s$-times for the intersection of
$C$ with $F_{\mu}$, then we can find $\mu'\in{\mathcal V}_{\mu}$
generic in $U$ such that $C\cap F_{\mu'}\cap{\mathcal V}_{p}$
consists of the distinct points $\{p_{1},\ldots,p_{s}\}$. By
elementary properties of infinitesimals,
$\{p_{1},\ldots,p_{s}\}\subset W_{\Phi_{\Sigma}}$, hence we can
find a corresponding distinct set $\{p_{1}',\ldots,p_{s}'\}$ in
$V_{\Phi_{\Sigma}}$. By the defining property of a specialisation
,the fact that the correspondence $\Gamma_{[\Phi_{\Sigma}]}$ is
closed and the definition of $\psi_{\mu'}$ we must have that
$\{p_{1}'\ldots,p_{s}'\}\subset C_{1}\cap{\mathcal
V_{p}}\cap\psi_{\mu'}$. As the map $\theta:Par_{e}\rightarrow
Par_{ke}$, defined by $(\dag)$, is algebraic, it follows that
$\psi_{\mu'}$ defines an infinitesimal variation of $\psi_{\mu}$
in the space of all algebraic curves of degree $ke$ on $P^{2}$.
Hence, it follows that $p'$ is counted at least $s$-times for the
intersection of $C_{1}$ with $\psi_{\mu}$. Therefore, $(**)$ is
shown. Now, given a hyperplane $H_{\mu}$;\\

$H_{\mu}=\mu_{0}Y_{0}+\ldots+\mu_{w}Y_{w}=0$\\

we obtain a corresponding algebraic curve $\phi_{\mu}$ of degree
$k$ on $P^{2}$, defined by the equation;\\

$\phi_{\mu}=\mu_{0}\phi_{0}+\ldots+\mu_{w}\phi_{w}=0$\\

 Corresponding to the factorisation $F_{\lambda}=F_{e}=H_{\lambda_{1}}\centerdot\ldots\centerdot H_{\lambda_{e}}$,
we obtain the factorisation
$\psi_{\lambda}=\phi_{\lambda_{1}}\centerdot\ldots\centerdot
\phi_{\lambda_{e}}$. Therefore, in order to show $(*)$,
it will be sufficient to prove that;\\

$I_{italian}(p',C_{1},\phi_{\lambda_{1}}\centerdot\ldots\centerdot
\phi_{\lambda_{e}})=1$ $(***)$\\

Let $p$ belong (uniquely) to the intersection $C\cap H_{\lambda_{j}}$. We claim first that;\\

$I_{italian}(p',C_{1},\phi_{\lambda_{j}})=1$\\

We clearly have that the linear system $\Sigma$ consists of
$\{\phi_{\lambda}:\lambda\in P^{w*}\}$. Hence, as by construction
$p'$ does not belong to $Base(\Sigma)$ and is non-singular, the
result in fact follows from Lemma 2.4 and a local result given
later in this section, Lemma 2.10, which is independent of this
theorem. By results of \cite{depiro2} on plane projective curves,
$(***)$ follows. Hence, $(*)$ is shown as
well.\\

We have now proved that there exists a form $F_{e}$ of degree $e$
which intersects $C$ in exactly $de$ points with multiplicity. The
theorem now follows immediately from the corresponding result in
Zariski structures that, for a finite equidimensional cover
 $G\subset Par_{e}\times P^{w}$;\\

$\Sigma_{x\in G(\lambda)}Mult_{(\lambda,x)}(G/Par_{F})$ is
preserved.\\

\end{proof}

Using this theorem, we develop a basic theory of $g_{n}^{r}$ on
$C$, a projective algebraic curve in $P^{w}$. Suppose that we are
given a linear system $\Sigma$ of dimension $r$, consisting of
algebraic forms $\phi_{\lambda}$, parametrised by $Par_{\Sigma}$.
We will assume that $C\cap\phi_{\lambda}$ has finite intersection
for each $\lambda\in Par_{\Sigma}$, which we will abbreviate by
saying that $\Sigma$ has finite intersection with $C$.
Then, for $\lambda\in Par_{\Sigma}$, we obtain the weighted set of points;\\

$W_{\lambda}=\{n_{p_{1}},\ldots,n_{p_{m}}\}$\\

where\\

 $\{p_{1},\ldots,p_{m}\}=C\cap \phi_{\lambda}$\\

 and\\

 $I_{italian}(p_{j},C,\phi_{\lambda})=n_{p_{j}}$ for $1\leq j\leq
 m$.\\

 By Theorem 2.3, the total weight $n_{p_{1}}+\ldots +n_{p_{m}}$ of these
 points is always equal to $de$.\\

It follows that, as $\lambda$ varies in $Par_{\Sigma}$, we obtain
a series of weighted sets $Series(\Sigma)=\{W_{\lambda}:\lambda\in
Par(\Sigma)\}$. We now make the following definition;\\

\begin{defn}

We define $order(Series(\Sigma))$ to be the total weight of any of
the sets in $Series(\Sigma)$. We define $dim(Series(\Sigma))$ to
be $dim(\Sigma)$. We define $g_{n}^{r}(\Sigma)$ to be the series
of weighted sets parametrised by $Par_{\Sigma}$ where;\\

 $n=order(Series(\Sigma))$ and
$r=dim(Series(\Sigma))$.\\

\end{defn}

We now make the following local analysis of $g_{n}^{r}(\Sigma)$.\\

\begin{defn}

Let $\Sigma$ be a linear system having finite intersection with
$C\subset P^{w}$. If $\phi_{\lambda}$ belongs to $\Sigma$ and
$p\in C\cap\phi_{\lambda}$, we define;\\

$I_{italian}^{\Sigma}(p,C,\phi_{\lambda})=Card(C\cap\phi_{\lambda'}\cap{\mathcal
V}_{p})$ for $\lambda'\in{\mathcal V}_{\lambda}$ generic in $Par_{\Sigma}$.\\

\end{defn}

\begin{rmk}
This is a good definition as $Par_{\Sigma}$ is smooth. The
difference between $I_{italian}$ and $I_{italian}^{\Sigma}$ is
that in the first case we can vary the parameter $\lambda$ over
$\emph{all}\ $ forms of degree $e$, while, in the second case, we
restrict the parameter to forms of the linear system defined by
$\Sigma$.
\end{rmk}

\begin{defn}

We will refer to ${C\setminus Base(\Sigma)}$ as the set of mobile
points for the system $\Sigma$.

\end{defn}

We now make the following preliminary definition;\\

\begin{defn}

We will define a coincident mobile point for $\Sigma$ to be a
point $p\in C\cap\phi_{\lambda}$, for some $\lambda\in
Par_{\Sigma}$, such that $p$ lies outside $Base(\Sigma)$ and with
the further property that;\\

$I_{italian}^{\Sigma}(p,C,\phi_{\lambda})\lvertneqq
I_{italian}(p,C,\phi_{\lambda})$\\

\end{defn}

It is an important property of \emph{linear} systems in
characteristic $0$ that there do not exist coincident mobile
points. (See the final section for the corresponding result in
arbitrary characteristic.) We will prove this in the following
lemmas, the notation of Definition 2.6 will be maintained until
Lemma 2.17.\\

\begin{lemma}{Non-Existence of Coincident Mobile Points}\\

Let $p\in ({C\setminus Base(\Sigma)})\cap\phi_{\lambda}$ be a
non-singular point. Then;\\

$I_{italian}(p,C,\phi_{\lambda})=I_{italian}^{\Sigma}(p,C,\phi_{\lambda})$\\

\end{lemma}

\begin{proof}
We prove this by induction on $m=I_{italian}(p,C,\phi_{\lambda})$.
The case $m=1$ is clear
as we \emph{always} have that;\\

$I_{italian}^{\Sigma}(p,C,\phi_{\lambda})\leq
I_{italian}(p,C,\phi_{\lambda})$\\

Suppose that $I_{italian}(p,C,\phi_{\lambda})=m+1$. Let
$\lambda'\in Par_{\Sigma}\cap {\mathcal V}_{\lambda}$ be generic
and let $\{p_{1},\ldots,p_{r}\}$ enumerate ${\mathcal V}_{p}\cap
C\cap\phi_{\lambda'}$. Suppose $r\geq 2$, then, by results of
\cite{depiro2}, (summability of specialisation), we have that;\\

$I_{italian}(p_{j},C,\phi_{\lambda'})\leq m$ for each $1\leq j\leq
r$.\\

 By properties of infinitesimal neighborhoods, each $p_{j}$
is non-singular and lies in ${C\setminus Base(\Sigma)}$. Hence, by
the induction hypothesis, it follows that;\\

$I_{italian}^{\Sigma}(p_{j},C,\phi_{\lambda'})=
I_{italian}(p_{j},C,\phi_{\lambda'})$ for each $1\leq j\leq r$\\

 Again, using the same result from \cite{depiro2}, (summability of specialisation),
  we have that;\\

$I_{italian}^{(\Sigma)}(p,C,\phi_{\lambda})=\sum_{1\leq j\leq
r}I_{italian}^{(\Sigma)}(p_{j},C,\phi_{\lambda'})$\\

where the $(\Sigma)$ notation is used to show that the result
holds for either of the above defined multiplicities. Hence;\\

$I_{italian}(p,C,\phi_{\lambda})=I_{italian}^{\Sigma}(p,C,\phi_{\lambda})$\\

We may, therefore, assume that;\\

For \emph{any} generic $\lambda'\in {\mathcal V}_{\lambda}\cap
Par_{\Sigma}$, there exists a \emph{unique} $p'\in {\mathcal
V}_{p}\cap C\cap\phi_{\lambda'}$, with
$I_{italian}(p',C,\phi_{\lambda'})=I_{italian}(p,C,\phi_{\lambda})$
$(*)$\\

 Now, as $p\notin Base(\Sigma)$, we can find $\phi_{\mu}$
with $p\notin C\cap\phi_{\mu}$. We consider the pencil of
algebraic forms defined by $\{\phi_{\lambda},\phi_{\mu}\}$. We may
assume $p$ is in finite position on the curve $C$. For ease of
notation, we will assume that $w=2$ and make the generalisation to
arbitrary dimension $w$ at the end of the lemma. Let $f(X,Y)=0$
define the curve $C$ in affine coordinates such that $p$
corresponds to the point $(0,0)$. We rewrite the pencil of curves
$(\phi_{\lambda},\phi_{\mu})$ in affine coordinates, which gives
the $1$-parameter family;\\

$g(X,Y;t)=\sum_{i+j\leq deg(\Sigma)} (\lambda_{ij}+t\mu_{ij})X^{i}Y^{j}=0$\\

The following calculation is somewhat informal, see remark $(ii)$
of Section 1;\\

As $p\notin Base(\Sigma)$, the function
${\phi_{\lambda}\over\phi_{\mu}}$ is defined at $p=(0,0)$ and
gives an
algebraic morphism;\\

${\phi_{\lambda}\over\phi_{\mu}}:W\subset C\rightarrow
Par_{t}$\\

on some open $W\subset C$, with ${\phi_{\lambda}\over\phi_{\mu}}(0,0)=0$.\\

The cover $graph({\phi_{\lambda}\over\phi_{\mu}})\subset W\times
Par_{t}$ of $Par_{t}$ is Zariski unramified at $(0,0,0)$ as, given
generic $t\in {\mathcal V}_{0}$, there exists a unique $(x,y)$
such that $(x,y)\in {\mathcal V}_{(0,0)}\cap C\cap
(\phi_{\lambda}+t\phi_{\mu})$, by $(*)$ and the fact that
$Par_{t}$ is smooth. By results of \cite{depiro1}, (Theorem 6.10),
if we assume the ground field $L$ has characteristic $0$, the
cover is etale at $(0,0,0)$. We will make the modification for
non-zero
characteristic in the final section of this paper. $(1)$\\

 Now, as $f(X,Y)=0$ is non-singular at $p=(0,0)$, we can apply the
implicit function theorem to obtain a parametrisation in algebraic
power series $(x(t),y(t))$ of the branch at $(0,0)$ such that
$f(x(t),y(t))=0$. (See the remark $(ii)$ in Section 1 and
\cite{depiro2} for the correct interpretation of these power
series on appropriate etale covers and the corresponding
definition of a branch.) We then obtain a map defined by algebraic
power series $\theta:A^{1,et}\rightarrow Par_{t}$ given by;\\

$\theta(t)={\phi_{\lambda}\over\phi_{\mu}}(x(t),y(t))$\\

where $A^{1,et}$ is an etale cover of $A^{1}$ over the
distinguished point $(0)$.
 We have
that $\theta$ is etale at $(0^{lift},0)$, as the composition of
etale maps is etale. By the inverse function theorem, we may find
an etale isomorphism $\rho:A^{1,et}\rightarrow A^{1,et}$ such that
$\theta(\rho(t))=t$ and
$(x_{1}(t),y_{1}(t))=(x(\rho(t)),y(\rho(t)))$ also parametrises the branch at $(0,0)$.\\

We therefore have that;\\

$g(x_{1}(t),y_{1}(t);t)=0$ (**) (2)\\

Now, by assumption, $I_{italian}(p,C,\phi_{\lambda})\geq 2$. Hence, by results of \cite{depiro2};\\

 $g(X,Y;0)$ is algebraically tangent to $f(X,Y)=0$ at $(0,0)$ (3).\\

By the chain rule and $(**)$, we have that;\\

${\partial g_{t}\over\partial
X}_{(x_{1}(t),y_{1}(t))}x_{1}'(t)+{\partial g_{t}\over\partial
Y}_{(x_{1}(t),y_{1}(t))}y_{1}'(t)+{\partial g\over\partial
t}_{(x_{1}(t),y_{1}(t))}=0$ $(***)$ (4)\\

Hence, at $t=0$, we have that ${\partial g\over\partial
t}_{(0,0)}=0$, that is $p=(0,0)$ belongs to $\phi_{\mu}=0$, which
is a contradiction. The calculation $(***)$ holds for formal power
series in $L[[t]]$. In particular, it holds for algebraic power
series. We now give a brief justification for this
calculation;\\

Let $v=ord_{t}$ be the standard valuation on the power series ring
$L[[t]]$. Given a power series $f\in L[[t]]$ and a sequence of
power series $\{f_{n}:n\in{\mathcal Z}_{\geq 0}\}$, we will say
that $\{f_{n}\}$ converges to $f$, abbreviated by
$\{f_{n}\}\rightarrow f$, if;\\

$(\forall m\in{\mathcal Z}_{\geq 0})(\exists n(m)\in{\mathcal
Z}_{\geq 0})(\forall k\geq
n(m))[v(f-f_{k})\geq m]$\\

Now choose sequences $\{x_{1}^{n}(t),y_{1}^{n}(t)\}$ of
\emph{polynomials} in $L[t]$ such that
$\{x_{1}^{n}(t)\}\rightarrow x_{1}(t)$ and
$\{y_{1}^{n}(t)\}\rightarrow y_{1}(t)$. We claim
that;\\

$\{g_{n}(t)=g(x_{1}^{n}(t),y_{1}^{n}(t);t)\}\rightarrow
g(x_{1}(t),y_{1}(t);t)$\\

This follows by standard continuity arguments for polynomials in
the non-archimidean topology induced on $L[[t]]$ by $v$. We then
have that;\\

$g_{n}'(t)={\partial g\over\partial
X}_{(x_{1}^{n}(t),y_{1}^{n}(t),t)}x_{1}^{n}(t)'+{\partial
g\over\partial
Y}_{(x_{1}^{n}(t),y_{1}^{n}(t),t)}y_{1}^{n}(t)'+{\partial
g\over\partial t}_{(x_{1}^{n}(t),y_{1}^{n}(t),t)}$\\

This follows from the fact that the chain rule and product rule
hold in the polynomial ring $L[t]$, even in non-zero
characteristic. We now claim that $\{x_{1}^{n}(t)'\}\rightarrow
x_{1}(t)'$ and $\{y_{1}^{n}(t)'\}\rightarrow y_{1}(t)'$. This
holds by the definition of convergence and the fact that, for a
power series $f\in L[[t]]$, if $ord_{t}(f)=r$, then
$ord_{t}(f')\geq r-1$. Using standard continuity arguments and
uniqueness of limits, one obtains the result $(4)$. One can also
give a geometric interpretation of the calculation $(4)$ using
duality arguments. We will
discuss this problem on another occasion.\\

In order to finish the argument, we claim that;\\

${\partial g_{0}\over\partial X}_{(0,0)}x_{1}'(0)+{\partial
g_{0}\over\partial Y}_{(0,0)}x_{2}'(0)=0$ $(5)$\\

This follows from $(3)$ and the fact that algebraic tangency can
be characterised by the property that $Dg_{0}$ at $(0,0)$ contains
the tangent line $l_{p}$ of $C$. This is clear if $g_{0}$ is
non-singular at $p$, in particular if $g_{0}$ has a non-reduced
component at $p$. Otherwise, it follows easily from \cite{depiro2}
or \cite{H}. Hence, at $t=0$, we have that ${\partial
g\over\partial t}_{(0,0)}=0$, that is $p=(0,0)$ belongs to
$\phi_{\mu}=0$, which is a contradiction.\\

We now consider the case for arbitrary dimension $w$. We will use
Theorem 1.33 to find a plane projective curve $C_{1}\subset P^{2}$
birational to $C$. Using Lemma 1.20, we can find a linearly
independent system $\Sigma'$ and a birational presentation
$\Psi_{\Sigma'}:C_{1}\leftrightsquigarrow C$. We will assume that
the point $p$ under consideration lies inside the canonical set
$W_{\Psi_{\Sigma'}}$ with corresponding $p'\in
V_{\Psi_{\Sigma'}}$. This can in fact always be arranged, see the
section on Conic Projections. However, for the moment, we can, if
necessary, replace the set $NonSing(C)$ by $W_{\Psi_{\Sigma'}}$.
Now, we follow through the calculation given above for $w=2$. The
argument up to $(1)$ is unaffected.  We first justify the
calculation $(2)$. Let $f(X,Y)=0$ be an affine representation of
$C_{1}$, such that the point $p'$ corresponds to $(0,0)$. Then, we
may obtain a local power series representation $(x(t),y(t))$ of
$f(X,Y)$ at $(0,0)$ and, applying $\Psi_{\Sigma'}$, a local power
series representation
$\Psi_{\Sigma'}(x(t),y(t))=(x_{1}(t),\ldots,x_{w}(t))$ of the
corresponding $p\in C$. We may then, applying the same argument,
obtain the relation $g(x_{1}(t),\ldots,x_{w}(t);t)=0$, where;\\

$g(X_{1},\ldots,X_{w};t)=\Sigma_{i_{1}+\ldots+i_{w}\leq
deg(\Sigma)}\lambda_{i_{1}\ldots i_{w}}X_{1}^{i_{1}}\ldots
X_{w}^{i_{w}}+t\mu_{i_{1}\ldots i_{w}}X_{1}^{i_{1}}\ldots
X_{w}^{i_{w}}$\\

is an affine representation of the pencil
$\phi_{\lambda}+t\phi_{\mu}$.\\

We now need to justify the calculation in $(3)$. Write
$\phi_{\lambda}$ in the form;\\

$\phi_{\lambda}=\Sigma_{i_{0}+\ldots
i_{w}=deg(\Sigma)}\lambda_{i_{0}\ldots
i_{w}}X_{0}^{i_{0}}\ldots X_{n}^{i_{w}}=0$\\

Let $\Sigma'=\{\psi_{0},\ldots,\psi_{w}\}$, then the assumption
that $\phi_{\lambda}$ passes through $p$ implies that the curve;\\

$D=\Sigma_{i_{0}+\ldots+i_{w}=deg(\Sigma)}\lambda_{i_{0}\ldots
i_{w}}\psi_{0}^{i_{0}}\ldots\psi_{i_{w}}^{i_{w}}=0$\\

passes through the corresponding $p'$ of $C_{1}$. By the fact
that\\
 $I_{italian}(p,C,\phi_{\lambda})\geq 2$, we can vary the
coefficients $\{\lambda_{i_{0}\ldots i_{w}}\}$ of $\phi_{\lambda}$
to obtain distinct intersections $\{x'',x'''\}$ in
$C\cap\phi_{\lambda'}\cap {\mathcal V}_{p}$. By properties of
infinitesimals, these intersections lie in the fundamental set
$W_{\Psi_{\Sigma'}}$. Hence, we can find corresponding
intersections $\{x'''',x'''''\}$ in ${\mathcal V}_{p'}\cap
V_{\psi_{\Sigma'}}$ with the corresponding variation of $D$. This
implies that;\\

$I_{italian}(p',C_{1},D)\geq 2$\\

By results of the paper \cite{depiro2}, $D$ must be algebraically
tangent to the curve $C_{1}$ at $p'$. Hence, by the chain rule,
and the characterisation of algebraic tangency given above,
$\phi_{\lambda}$ is algebraically tangent to the curve $C$ at $p$,
in the sense that its differential $D\phi_{\lambda}$ at $p$,
contains the tangent line $l_{p}$ of $C$ $(*)$. The reader should
also look at the proof of Theorem 2.3, where a similar calculation
was carried out. Finally, we need to justify $(4)$. This is clear
from the calculation done above. The final step $(5)$ is also
clear from the corresponding calculation and $(*)$.

\end{proof}

\begin{rmk}
The lemma fails for \emph{non-linear} systems. Let $C$ be defined
in affine coordinates $(x,y)$ by $y=0$ and let $\{\phi_{t}\}$ be
the pencil of curves, defined in characteristic $0$, by
$y=(x-t)^{2}=x^{2}-2tx+t^{2}$. By construction, each $\phi_{t}$ is
tangent to $y=0$ at $(t,0)$. It follows that each $(t,0)\in C$ is
a coincident mobile point for $\phi_{t}$.
\end{rmk}

\begin{lemma}
Suppose that $p\in
{C\setminus Base(\Sigma)}\cap\phi_{\lambda}$ is a singular point, then;\\

$I_{italian}(p,C,\phi_{\lambda})=I_{italian}^{\Sigma}(p,C,\phi_{\lambda})$.

\end{lemma}

\begin{proof}
Suppose that $I_{italian}^{\Sigma}(p,C,\phi_{\lambda})=m$. As
$p\notin Base(\Sigma)$, the condition that $\phi_{\lambda}$ passes
through $p$ defines a proper closed subset of the parameter space
$Par_{\Sigma}$. Hence, we can find $\lambda'\in{\mathcal
V}_{\lambda}$ generic in $Par_{\Sigma}$ and
$\{p_{1},\ldots,p_{m}\}=C\cap\phi_{\lambda'}\cap{\mathcal V}_{p}$,
distinct from $p$, witnessing this multiplicity. As both
$NonSing(C)$ and ${C\setminus Base(\Sigma)}$ are open and defined
over $L$, we have that $\{p_{1},\ldots,p_{m}\}$ must lie in the
intersection of these sets. Applying the result of the previous
lemma, we must have that;\\

$I_{italian}(p_{j},C,\phi_{\lambda'})=1$ for $1\leq j\leq m$\\

Hence, by summability of specialisation, again see the paper
\cite{depiro2}, we must have that
$I_{italian}(p,C,\phi_{\lambda})=m$ as required.
\end{proof}

In the following Lemma 2.13, Lemma 2.16 and Lemma 2.17, by a
canonical set on $C$, we will mean either a set of the form
$V_{\phi_{\Sigma_{1}}}$, for the domain of a birational map
$\phi_{\Sigma_{1}}$, or a set of the form $W_{\psi_{\Sigma_{2}}}$,
for the image of a birational map
 $\psi_{\Sigma_{2}}$, see also Definition 1.30. For ease of
 notation, we will abbreviate either of these sets by $W$. In
 particular, $W$ may include the canonical set $V_{\Phi_{\Sigma}}$
 defined by (any) choice of basis for the linear system $\Sigma$.

\begin{lemma}{Multiplicity at non-base points witnessed by
transverse intersections in the canonical sets.}\\

Let $p\in C\setminus Base(\Sigma)$, then, if
$m=I_{italian}(p,C,\phi_{\lambda})$, we can find
$\lambda'\in{\mathcal V}_{\lambda}$, generic in $Par_{\Sigma}$,
and distinct
$\{p_{1},\ldots,p_{m}\}=C\cap\phi_{\lambda'}\cap{\mathcal V}_{p}$
such that $\{p_{1},\ldots,p_{m}\}$ lies in the canonical set $W$
,with the point $p$ removed, ${W\setminus \{p\}}$, and the
intersection of $C$ with $\phi_{\lambda'}$ at each $p_{j}$ is
transverse for $1\leq j\leq m$.
\end{lemma}

\begin{proof}
 As $p\notin Base(\Sigma)$, the condition that $\phi_{\lambda}$
does not pass through $p$ defines an open subset of
$Par_{\Sigma}$. By the previous lemma, taking generic (over $L$),
${\lambda'\in{\mathcal V}_{\lambda}}$, we can find
$\{p_{1},\ldots,p_{m}\}=C\cap\phi_{\lambda'}\cap{\mathcal V}_{p}$,
distinct from $p$. Finally, ${C\setminus W}$ defines a finite
subset of $C$ (over $L$). Clearly, $\{p_{1},\ldots,p_{m}\}$ avoid
this set, otherwise, by properties of specialisations, some
$p_{j}$ would equal $p$ for $1\leq j\leq m$. Finally, the
transversality result follows from the fact that
$I_{italian}(p_{j},C,\phi_{\lambda'})=1$ for $1\leq j\leq m$,
(using Lemmas 2.10 and 2.12 again).
\end{proof}

We have analogous results to Lemmas 2.10, 2.12 and 2.13 for points
in $Base(\Sigma)$;\\

We first require the following;\\

\begin{lemma}

Let $p\in C\cap Base(\Sigma)$, then there exists an open subset
$U_{p}\subset Par(\Sigma)$ and an integer $I_{p}\geq 1$
such that;\\

$I_{italian}(p,C,\phi_{\lambda})=I_{p}$ for $\lambda\in
U_{p}$.\\

and\\

 $I_{italian}(p,C,\phi_{\lambda})\geq I_{p}$ for
$\lambda\in Par_{\Sigma}$.\\

\end{lemma}

\begin{proof}

By properties of Zariski structures, we have that;\\

$W_{k}=\{\lambda\in Par(\Sigma):
I_{italian}(p,C,\phi_{\lambda})\geq k\}$\\

are definable and Zariski closed in $Par_{\Sigma}$. The result
then follows by taking $I_{p}=min_{\lambda\in
Par_{\Sigma}}I_{italian}(p,C,\phi_{\lambda})$ and the fact that
$Par_{\Sigma}$ is irreducible.

\end{proof}

We can now formulate the corresponding version of Lemmas 2.12 and 2.13 for base points;\\

\begin{lemma}
Let $p\in C\cap Base(\Sigma)\cap\phi_{\lambda}$. Then;\\

$I_{italian}(p,C,\phi_{\lambda})=I_{p}+I_{italian}^{\Sigma}(p,C,\phi_{\lambda})-1$

\end{lemma}

\begin{proof}
Let $m=I_{italian}^{\Sigma}(p,C,\phi_{\lambda})$. Choosing
$\lambda'\in {\mathcal V}_{\lambda}$ generic in $Par_{\Sigma}$, we
can find $\{p_{1},\ldots,p_{m-1}\}=C\cap{\mathcal
V}_{p}\cap\phi_{\lambda'}$, distinct from $p$, witnessing this
multiplicity. Therefore, for $1\leq j\leq m-1$, $p_{j}\notin
Base(\Sigma)$, by properties of specialisations and the fact that
$Base(\Sigma)$ is finite and defined over $L$. Hence, we can apply
the results of Lemmas 2.10 and 2.12 to conclude that
$I_{italian}(p_{j},C,\phi_{\lambda'})=1$ for $1\leq j\leq m-1$. As
$\lambda'$ was generic in $Par_{\Sigma}$, using the previous Lemma
2.14, we have that $I_{italian}(p,C,\phi_{\lambda'})=I_{p}$. Now,
it follows easily, using summability of specialisation (see the
paper \cite{depiro2}), that
$I_{italian}(p,C,\phi_{\lambda})=I_{p}+(m-1)$. The lemma is
proved.
\end{proof}

\begin{lemma}
Let $p\in C\cap Base(\Sigma)\cap\phi_{\lambda}$, then, if
$m=I_{italian}^{\Sigma}(p,C,\phi_{\lambda})$, we can find
$\lambda'\in{\mathcal V}_{\lambda}$, generic in $Par_{\Sigma}$,
and $\{p,p_{1},\ldots,p_{m-1}\}=C\cap\phi_{\lambda'}\cap {\mathcal
V}_{p}$ witnessing this multiplicity such that
$\{p_{1},\ldots,p_{m-1}\}$ lie in the canonical set $W$, see the
explanation before Lemma 2.13, and the intersections
$C\cap\phi_{\lambda'}$ at $p_{j}$ are transverse for $1\leq j\leq
m-1$.
\end{lemma}

\begin{proof}
Use the proof of Lemma 2.13, basic properties of infinitesimals
and the fact that $W$ is open in $C$ and definable over $L$.
\end{proof}

\begin{lemma}{Generic Intersections}\\

Fix a canonical set $W$ and let $\phi_{\lambda}$ be generic in
$\Sigma$, then each point of intersection of $C$ with
$\phi_{\lambda}$ outside $Base(\Sigma)$ lies inside the canonical
set $W$ and is transverse.

\end{lemma}

\begin{proof}

The finitely many points of $({C\setminus W})$ are defined over
the data of $\{W,C\}$. Hence, the condition on $Par_{\Sigma}$ that
$\phi_{\lambda}$ intersects a point of $({C\setminus W})$ outside
$Base(\Sigma)$ consists of a finite union of proper hyperplanes
defined over the data of $\{W,C\}$. Therefore, for generic
$\phi_{\lambda}$, each point of intersection of $\phi_{\lambda}$
with $C$, outside $Base(\Sigma)$, lies inside $W$. Now observe
that the condition of transversality between $C$ and
$\phi_{\lambda}$, inside $W$, defines a constructible
condition on $Par_{\Sigma}$, over the data of $\{C,W\}$. Namely;\\

$\theta(\lambda)\equiv\forall y[(y\in\phi_{\lambda}\cap
W)\rightarrow NonSing(y)\wedge
RightMult_{y}(C,\phi_{\lambda})=1]$\\

By Lemmas 2.13 and 2.16, the condition is Zariski dense in
$Par_{\Sigma}$. Hence, the result
follows.\\

\end{proof}

We now make the following definitions;\\

\begin{defn}

Let $\Sigma$ be a linear system defining a $g_{n}^{r}(\Sigma)$ on
a projective algebraic curve $C\subset P^{w}$. Let
$\{W_{\lambda}\}=Series(\Sigma)$. If $p\in W_{\lambda}$, we say that;\\

$p$ is $s$-fold ($s$-plo in Italian) for the $g_{n}^{r}(\Sigma)$ if $I_{italian}(p,C,\phi_{\lambda})\geq s$.\\

$p$ is counted (contato) $s$-times for the $g_{n}^{r}(\Sigma)$ if
$p$ has multiplicity $s$ in $W_{\lambda}$, equivalently
$I_{italian}(p,C,\phi_{\lambda})=s$.\\

$p$ has $s$-fold contact (contatto) with $\phi_{\lambda}$ if
$I_{italian}^{\Sigma}(p,C,\phi_{\lambda})=s$.\\

\end{defn}

\begin{rmk}
The Italian terminology is generally quite confusing. The above
lemmas show that the discrepancy between contatto and contato
occurs only at fixed points of the system $\Sigma$. The philosophy
behind their approach is that algebraic calculations may be
reduced to visual arguments using the ideas that a $s$-fold
contact at $p$ is a limit of $s$-points converging along the curve
from intersections with forms in the system $\Sigma$ and that
these points are preserved by birationality. In the case of a
fixed point, $p$ may be counted \emph{more} times than its actual
contact with $\phi_{\lambda}$ and this excess intersection is
never actually visually manifested by a variation. The Italian
approach is to ignore this excess intersection. This motivates the
following definition;
\end{rmk}

\begin{defn}
Let a $g_{n}^{r}(\Sigma)$ be given on $C$. For $p\in
C\cap\phi_{\lambda}$, we define;\\

$I_{italian}^{mobile}(p,C,\phi_{\lambda})=Card(C\cap\phi_{\lambda'}\cap{\{\mathcal
V_{p}\setminus p}\})$, $\lambda'\in {\mathcal V}_{\lambda}$
generic in $Par_{\Sigma}$.\\

\end{defn}

\begin{rmk}
One needs to check, as usual, that this is a good definition. This
follows, for example, by Lemma 2.14, Remarks 2.7 and Lemma 2.15.
\end{rmk}

\begin{lemma}

For $p\notin Base(\Sigma)$,
$I_{italian}^{mobile}(p,C,\phi_{\lambda})=I_{italian}(p,C,\phi_{\lambda})$\\

For $p\in Base(\Sigma)$, we have that;\\

$I_{p}+I_{italian}^{mobile}(p,C,\phi_{\lambda})=I_{italian}(p,C,\phi_{\lambda})$\\

\end{lemma}

\begin{proof}
The proof follows immediately from the same results cited in the
previous remark.
\end{proof}

We now make the following definition;\\

\begin{defn}
By a $g_{n}^{r}$, we mean the series obtained from a given
$g_{n'}^{r}(\Sigma)$ by removing some (possibly all) of the fixed
point contributions $I_{x}$ in $Base(\Sigma)$. That is, we
subtract some part of $I_{x}$ from each weighted set
$W_{\lambda}$. We define $n$ to be the total multiplicity of each
$W_{\lambda}$ after subtracting some of the fixed point
contribution, so $n\leq n'$. We say that the $g_{n}^{r}$ has no
fixed points if all the fixed point contributions are removed. We
refine the Italian terminology for a $g_{n}^{r}$ by saying that
$x$ is $s$-fold for $W_{\lambda}$ if it appears in the weighted
set with multiplicity at least $s$ and $x$ is counted $s$-times
for $W_{\lambda}$ if it appears with multiplicity exactly $s$. We
define $Base(g_{n}^{r})$ to be $\{x\in C:\forall\lambda(x\in
W_{\lambda})\}$.

\end{defn}
We now have the following;\\

\begin{lemma}
For a given $g_{n}^{r}$, we always have that $r\leq n$.
\end{lemma}
\begin{proof}
Let the $g_{n}^{r}$ be defined by $\Sigma$, a linear system of
dimension $r$, having finite intersection with $C$. Pick
$\{p_{1},\ldots,p_{r}\}$ independent generic points of $C$, not
contained in $Base(\Sigma)$. The condition that $\phi_{\lambda}$
passes through $p_{j}$ defines a proper hyperplane condition
$H_{p_{j}}$ on $Par_{\Sigma}$. The base points of the subsystem
defined by $H_{p_{1}}\cap\ldots\cap H_{p_{j}}$ are defined over
${p_{1},\ldots,p_{j}}$ and finite, for $1\leq j\leq r-1$, as, by
assumption, no form in $\Sigma$ contains $C$. We must, therefore,
have that $dim(H_{p_{1}}\cap\ldots \cap H_{p_{r}})=0$. That is
there exists a unique $\phi_{\lambda}$ in $\Sigma$ passing through
$\{p_{1},\ldots,p_{r}\}$. We must have that the total intersection
multiplicity of $C$ with $\phi_{\lambda}$ outside the fixed
contribution is at least $r$, by construction, hence $r\leq n$ as
required.

\end{proof}

\begin{defn}{Subordinate Systems}\\

We say that\\

 $g_{n_{1}}^{r_{1}}\subseteq g_{n_{2}}^{r_{2}}$\\

if there exist linear systems $\Sigma'\subseteq\Sigma$ (having
finite intersection with $C$) of dimension $r_{1}$ and dimension
$r_{2}$ respectively such that $g_{n_{1}}^{r_{1}}$ is obtained
from $\Sigma'$, $g_{n_{2}}^{r_{2}}$ is obtained from $\Sigma$ and,
for each $\lambda\in Par_{\Sigma'}$, $W_{\lambda}\subseteq
V_{\lambda}$. Here, by $\{W_{\lambda},V_{\lambda}\}$, we mean the
weighted sets parametrised by
$\{g_{n_{1}}^{r_{1}},g_{n_{2}}^{r_{2}}\}$ and, by
$W_{\lambda}\subseteq V_{\lambda}$, we mean that $n_{p_{i}}\leq
m_{p_{i}}$ where $W_{\lambda}=\{n_{p_{1}},\ldots,n_{p_{r}}\}$ and
$V_{\lambda}=\{m_{p_{1}},\ldots,m_{p_{r}}\}$.
\end{defn}
\begin{rmk}
Note that the relationship of subordination is clearly transitive.
That is, if $g_{n_{1}}^{r_{1}}\subseteq g_{n_{2}}^{r_{2}}\subseteq
g_{n_{3}}^{r_{3}}$, then $g_{n_{1}}^{r_{1}}\subseteq
g_{n_{3}}^{r_{3}}$.
\end{rmk}
\begin{defn}{Composite Systems}\\

We say that a $g_{n}^{r}$, defined by $\Sigma$, is composite, if,
for generic $p\in C$, every weighted set $W_{\lambda}$ containing
$p$ also contains a distinct $p'(p)$ with $p'\notin Base(\Sigma)$.
\end{defn}

\begin{rmk}
Note that the definition of composite is well defined, for, if the
given $g_{n}^{r}$ is defined by $\Sigma$, the statement;\\

$\theta(x)\equiv \exists x'\notin Base(\Sigma)\forall\lambda\in
Par_{\Sigma}(x\in C\cap\phi_{\lambda}\rightarrow x'\in
C\cap\phi_{\lambda})$\\

defines a constructible subset of $C$. Hence, if it holds for some
generic $p$, it holds for \emph{any} generic $p$ in $C$. In modern
terminology, we would say that the given $g_{n}^{r}$ seperates
points (generically), see Proposition 7.3 of \cite{H}.
\end{rmk}
\begin{defn}{Simple Systems}\\

We say that a $g_{n}^{r}$ is simple if it is not composite.
\end{defn}

The importance of simple $g_{n}^{r}$ is due to the following;\\

\begin{lemma}{Construction of a Birational Model}\\

A simple $g_{n}^{r}$ on $C$ defines a projective image $C'\subset
P^{r}$, birational to $C$.
\end{lemma}

\begin{proof}
Let the $g_{n}^{r}$ be defined by a linear system $\Sigma$, with a
choice of basis $\{\phi_{0},\ldots,\phi_{r}\}$, (having finite
intersection with $C$), possibly after removing some fixed point
contribution. Let $\Phi_{\Sigma}$ be the morphism defined as in
Lemma 1.16. This morphism is defined on an open subset
$U={C\setminus Base(\Sigma)}$ of $C$. By continuity, the image of
$\Phi_{\Sigma}$ on $U$ is irreducible, hence either defines a
constructible $V\subset P^{r}$ of dimension $1$ or the image is a
point. We can clearly exclude the second case, otherwise we can
find $\{\phi_{\lambda},\phi_{\mu}\}$ differing by a constant of
proportionality $\rho$ on $U$, therefore $\phi_{\lambda}-\rho
\phi_{\mu}$ contains $C$. Let $C'=\overline{V}$, then $C'$ is an
irreducible projective algebraic curve. We claim that $C'$ is
birational to $C$. If not, then, using Lemma 1.21 or just the
definition of birationality, in characteristic $0$, for generic
$x\in C$, we can find a distinct $x'\in U$ such that
$\Phi_{\Sigma}(x)=\Phi_{\Sigma}(x')$. The choice of basis for
$\Sigma$ determines an isomorphism of $Par_{\Sigma}$ with $P^{r}$.
Using the parametrisation of $Par_{\Sigma}$ given by this
isomorphism, if $\phi_{\lambda}$ passes through $x$, the
corresponding hyperplane $H_{\lambda}\subset P^{r}$ would pass
through $\Phi_{\Sigma}(x)$ and, therefore, $\phi_{\lambda}$ would
pass through $x'$ as well. This contradicts simplicity. The lemma
may, of course, fail in non-zero characteristic. We refer the
reader to the final section for the problems associated to
Frobenius.
\end{proof}

We also have the following transfer results;\\

\begin{lemma}
Let a simple $g_{n}^{r}$ on $C$ be given, as in the previous
lemma, defining a birational projective image $C'\subset P^{r}$.
Let $V_{\phi_{\Sigma}}$ and $W_{\phi_{\Sigma}}$ denote the
canonical sets associated to this birational map. Then, given a
$g_{m}^{d}$ on $C'$, without fixed points, there exists a
corresponding $g_{m}^{d}$ on $C$, without fixed points, and, for
any corresponding pair $\{x,x'\}$ in
$\{V_{\phi_{\Sigma}},W_{\phi_{\Sigma}}\}$, $x$ is counted $s$
times in $V_{\lambda}$ iff $x'$ is counted $s$ times in
$W_{\lambda}$, where $\{V_{\lambda},W_{\lambda}\}$ are the
weighted sets parametrised by the $g_{m}^{d}$, on $\{C,C'\}$
respectively.
\end{lemma}

\begin{proof}
Let the $g_{m}^{d}$ on $C'$ be defined by a linear system
$\Sigma'$ of dimension $d$, with basis
$\{\psi_{0},\ldots,\psi_{d}\}$, after removing all the fixed point
contribution. We then obtain a corresponding linear system $\Sigma''$ defined by;\\

$\{\theta_{\lambda}=\lambda_{0}\psi_{0}(\phi_{0},\ldots,\phi_{r})+\ldots+\lambda_{d}\psi_{d}(\phi_{0},\ldots,\phi_{r})=0\}$\\

Clearly, $\Sigma''$ has finite intersection with $C$, hence it
defines a $g_{m'}^{d}$ after removing all fixed point
contributions. Now, let $\{x,x'\}$ be a corresponding pair in
$\{V_{\phi_{\Sigma}},W_{\phi_{\Sigma}}\}$. Suppose that $x'$ is
counted $s$ times in $W_{\lambda}$, then, as the $g_{m}^{d}$ on
$C'$ has no fixed points, we have that;\\

$I_{italian}^{mobile}(x',C',\psi_{\lambda})=s$\\

Therefore, we can find $\lambda'\in{\mathcal V}_{\lambda}$ generic
in $Par_{\Sigma'}$ and\\

$\{x_{1}',\ldots,x_{s}'\}=C'\cap\psi_{\lambda'}\cap{\mathcal
V}_{x'}\setminus \{x'\}$\\

By properties of infinitesimals, $\{x_{1}',\ldots,x_{s}'\}\subset
W_{\phi_{\Sigma}}$, hence, we can find corresponding
$\{x_{1},\ldots,x_{s}\}\subset V_{\phi_{\Sigma}}\cap
\theta_{\lambda'}\cap{{\mathcal V}_{x}\setminus \{x\}}$. It
follows that $I_{italian}^{mobile}(x,C,\theta_{\lambda})\geq s$,
and equality follows from the converse argument. Therefore, as the
corresponding $g_{m'}^{d}$ has no fixed points, $x$ is counted $s$
times in $V_{\lambda}$. The converse uses the same argument. It
remains to show that $m=m'$. By Lemmas 2.16 and 2.17 and the fact
that the $g_{m'}^{d}$ has no fixed points, a generic $V_{\lambda}$
consists of $m'$ points each counted once inside
$V_{\phi_{\Sigma}}$. Therefore, by the above argument, these
points are each counted once inside the corresponding
$W_{\lambda}$. Hence, $m'\leq m$. We obtain $m\leq m'$ by the
reversal of this argument.

\end{proof}

\begin{lemma}
Let a simple $g_{n}^{r}$ on $C$ be given, defining a birational
map $\Phi_{\Sigma}:C\leftrightsquigarrow C'\subset P^{r}$. Then,
if $d$ is the degree of $C'$, we have that $d'=d+I$ for the
$g_{d'}^{r}(\Sigma)$ defined by $\Sigma$, where $I$ is the total
fixed point contribution from $Base(\Sigma)$.

\end{lemma}

\begin{proof}

Let $k$ be the degree of $\phi_{\Sigma}$ and $v$ the degree of
$C$,
then we claim that;\\

$d+I=kv$\\

By Lemma 2.4, if $H_{\lambda}$ is a generic hyperplane, it cuts
$C'$ transversely in $d$ distinct points. We may also assume that
these points lie inside the canonical set $W_{\Phi_{\Sigma}}$ as
this is defined over the data of $\Phi_{\Sigma}$. Let
$\{p_{1},\ldots,p_{d}\}$ be the corresponding points of
$V_{\Phi_{\Sigma}}$ and $\phi_{\lambda}$ the corresponding form in
$\Sigma$. By, for example, Lemma 2.17;\\

$I_{italian}(p_{j},C,\phi_{\lambda})=1$ for $1\leq j\leq d$\\

There can be no more intersections of $\phi_{\lambda}$ with $C$
outside $Base(\Sigma)$, otherwise one would obtain a corresponding
intersection of $H_{\lambda}$ with $C'$ outside
$W_{\Phi_{\Sigma}}$. Hence, the total multiplicity of intersection
$I_{italian}(C,\phi_{\lambda})$ between $C$ and $\phi_{\lambda}$
is exactly $I+d$, using Lemma 2.14. By Theorem 2.3 and the fact
that $\phi_{\lambda}$ is a form of degree $k$, we also have that
the total multiplicity $I_{italian}(C,\phi_{\lambda})$ is $kv$.
As, by definition, $d'=kv$, The result follows.

\end{proof}

\end{section}

\begin{section}{The Construction of a Birational Model of a
Plane Projective Algebraic Curve without Multiple Points}

We first make the following definition.\\

\begin{defn}

 Let $C\subset P^{w}$ be a projective algebraic curve, not contained in any hyperplane of $P^{w}$.
  We define a point $p\in C$ to be $s$-fold on $C$ if, for every
hyperplane $H$ passing through $p$;\\

$I_{italian}(p,C,H)\geq s$.\\

and equality is attained.We define $p$ to be a multiple point if
it is $s$-fold for some $s\geq 2$. We define $p$ to be simple if
it is not multiple.

\end{defn}

We have the following lemma;\\

\begin{lemma}
Let $C\subset P^{2}$ be a projective algebraic curve and $p$ a
point on $C$. Let $F(X,Y)=0$ be an affine representation of $C$
such that the point $p$ corresponds to $(0,0)$. Then $p$ is
$s$-fold on $C$ iff we can write $F$ in the form;\\

$F(X,Y)=\Sigma_{(i+j)\geq s}a_{ij}X^{i}Y^{j}$\\

In particular, $p$ is non-singular iff it is not multiple.

\end{lemma}

\begin{proof}
Suppose that $p$ is $s$-fold on $C$, then, for a generic line $l$
defined by $aX-bY=0$, we have that;\\

$I_{italian}(p,C,l)\geq s$\\

It follows, by the results of \cite{depiro2}, that we also have;\\

$I_{italian}(p,l,C)=I_{italian}(p,C,l)\geq s$ $(*)$\\

Now, suppose that $F(X,Y)$ has the expansion;\\

$F(X,Y)=\Sigma_{i+j\leq deg(F)} a_{ij}X^{i}Y^{j}$.\\

and that $a_{ij}\neq 0$ for some $(i,j)$ with $i+j<s$ (**). We can
parametrise the branch at $(0,0)$ of $l$ algebraically by;\\

$(x(t),y(t))=(bt,at)$\\

We make the substitution of this parametrisation into $F(X,Y)$ and
obtain;\\

$F(x(t),y(t))=\Sigma a_{ij}x(t)^{i}y(t)^{j}=\Sigma a_{ij}a^{j}b^{i}t^{i+j}$\\

By $(**)$ and generic choice of $\{a,b\}$, this expansion has
order $s_{1}<s$. In characteristic $0$, it then follows, by the
method of \cite{depiro2}, see also Theorem 6.1 of this paper,
that, for the pencil $\Sigma_{1}$ defined by $\{F_{t}\equiv F(X,Y)+t=0\}$;\\

$I_{italian}^{\Sigma_{1}}(p,l,F)=s_{1}$.\\

In particular, as $p\notin Base(\Sigma_{1})$, we have, by Lemma
2.10, that;\\

$I_{italian}(p,l,F)=s_{1}$\\

and therefore, by $(*)$, that;\\

$I_{italian}(p,C,l)=s_{1}$\\

contradicting the fact that $p$ was $s$-fold on $C$. It follows
that $(**)$ doesn't hold, as required. See, however, the final
section for the problem in non-zero characteristic. For the
converse direction, suppose that $F(X,Y)$ has the expansion;\\

$F(X,Y)=\Sigma_{i+j\geq s}a_{ij}X^{i}Y^{j}$\\

Then, by a direct calculation and using the argument above, one
has that, for $\emph{any}$
line;\\

$I_{italian}(p,l,F)\geq s$.\\

Hence, by $(*)$ again, $p$ must be $s$-fold on $C$.\\

Using this result, it follows immediately, by considering the
vector $({\partial F\over \partial X},{\partial F\over \partial
Y})$ evaluated at $(0,0)$, that $p$ is non-singular on $C$ iff $p$
is not multiple.\\

\end{proof}

Our main result in this section will be the following;\\

\begin{theorem}
Let $C\subset P^{2}$ be a projective algebraic curve, then there
exists $C_{1}\subset P^{w}$ such that $C$ and $C_{1}$ are
birational and $C_{1}$ has no multiple points.
\end{theorem}

The proof will proceed using the basic theory of $g_{n}^{r}$ that
we developed in the previous section. We require the following definition;\\

\begin{defn}
Let a $g_{n}^{r}$ without fixed points be given on $C$. We will
call the $g_{n}^{r}$ transverse if the following property holds;\\

There does not exist a subordinate $g_{n'}^{r-1}\subset
g_{n}^{r}$ such that $n'\leq n-2$.\\

\end{defn}
We then have;\\

\begin{lemma}
Let a $g_{n}^{r}$ on $C$ be given without fixed points. Then, if
the $g_{n}^{r}$ is transverse, it must be simple.
\end{lemma}

\begin{proof}
Suppose that the $g_{n}^{r}$, defined by $\Sigma$, is not simple,
then it must be composite. Therefore, for generic $x\in C$, there
exists an $x'(x)\notin Base(\Sigma)$ such that every weighted set
$W_{\lambda}$ containing $x$ also contains $x'(x)$. As $x$ is
generic, $x\notin Base(\Sigma)$, hence the subsystem
$\Sigma_{x}\subset\Sigma$, consisting of forms in $\Sigma$ passing
through $x$, has dimension $r-1$. Define a $g_{n'}^{r-1}$ from
$\Sigma_{x}$ by removing the fixed point contribution of
$g_{n''}^{r-1}(\Sigma_{x})$. We claim that $g_{n'}^{r-1}\subseteq
g_{n}^{r}$ $(*)$. We clearly have that $\Sigma_{x}\subset\Sigma$.
Suppose that $p$ appears in a $W_{\lambda}$ defined by
$g_{n'}^{r-1}$ with multiplicity $s$. Then, by definition,
$I_{italian}^{mobile}(p,C,\phi_{\lambda})=s$, calculated with
respect to $\Sigma_{x}$. It follows easily from the above lemmas
that $I_{italian}^{mobile}(p,C,\phi_{\lambda})=s$, calculated with
respect to $\Sigma$, as well. Hence, $p$ also appears in the
corresponding $V_{\lambda}$ with multiplicity $s$. This gives the
claim $(*)$. We show that $n'\leq n-2$. Let
$\{p_{1},\ldots,p_{r}\}$ be the base points of $\Sigma$. Then
$\Sigma_{x}$ has base points containing the set
$\{p_{1},\ldots,p_{r},x,x'(x)\}$. Clearly, the fixed point
contribution of $\Sigma_{x}$ at $\{p_{1},\ldots,p_{r}\}$ can only
increase the fixed point contribution of $\Sigma$ and, moreover,
the fixed point contribution of $\Sigma_{x}$ at $\{x,x'\}$ must be
at least $2$. Hence, using the Hyperspatial Bezout Theorem and the
definition of $\{g_{n}^{r},g_{n'}^{r-1}\}$, we have that $n'\leq
n-2$. This contradicts the fact that the original $g_{n}^{r}$ was
transverse. Hence, the $g_{n}^{r}$ must be simple.
\end{proof}

Following from this, we have the following;\\

\begin{lemma}
Let a $g_{n}^{r}$ be given on $C$ without fixed points such that
$g_{n}^{r}$ is transverse. Then the $g_{n}^{r}$ defines a
birational map $\Phi:C\leftrightsquigarrow C_{1}\subset P^{r}$
and, moreover, $C_{1}$ has no multiple points.
\end{lemma}

\begin{proof}
The first part of the lemma follows from the previous Lemma 3.5
and Lemma 2.30. See the final section, though, for the problem in
non-zero characteristic. It remains to prove that $C_{1}$ has no
multiple points. Consider the $g_{n'}^{r}$ without fixed points on
$C_{1}$, defined by the linear system $\Sigma$ of hyperplane
sections. By Lemma 2.31, there exists a corresponding $g_{n'}^{r}$
on $C$ without fixed points. By its construction in Lemma 2.31, it
equals the original $g_{n}^{r}$. Now suppose that there exists a
multiple point $p$ of $C_{1}$. We consider the subsystem
$\Sigma_{p}$ of hyperplane sections passing through $p$. This
defines a $g_{n}^{r-1}\subset g_{n}^{r}$ and, as $p$ is multiple,
after removing the fixed point contribution, we obtain a
$g_{n'}^{r-1}\subset g_{n}^{r}$ with $n'\leq n-2$. Using Lemma
2.31 again, we obtain a corresponding $g_{n'}^{r-1}$ on $C$. We
claim that $g_{n'}^{r-1}\subset g_{n}^{r}$ on $C$. This follows
easily from the fact that $\Sigma_{p}\subset\Sigma$, the
$g_{n'}^{r-1}$ has no fixed points and the argument of the
previous lemma. This contradicts the fact that the original
$g_{n}^{r}$ on $C$ was transverse. Hence, $C_{1}$ has no multiple
points.

\end{proof}

We now give the proof of Theorem 3.3. We find a transverse
$g_{n}^{r}$ on $C$ using a combinatorial argument;\\

\begin{proof}{(Theorem 3.3)}\\

Suppose that $C$ has order $n\geq 2$ (the case when $C$ has order
$n=1$ is obvious.) Consider the independent system $\Sigma$
consisting of all plane algebraic curves defined by homogeneous
forms of order $n-1$. Clearly, no form in $\Sigma$ may contain
$C$, and the system has no fixed points on $C$, hence the system
defines a $g_{n_{1}}^{r_{1}}$, without fixed points, of dimension
$r_{1}={(n-1)(n+2)\over 2}$ and order $n_{1}=n(n-1)$. If this
$g_{n_{1}}^{r_{1}}$ is transverse, the proof is complete. Hence,
we may suppose that there exists a $g_{n_{2}}^{r_{2}}\subset
g_{n_{1}}^{r_{1}}$, without fixed points, such that
$r_{2}=r_{1}-1$ and $n_{2}\leq n_{1}-2$. Again, if this
$g_{n_{2}}^{r_{2}}$ is transverse, the proof is complete.
Hence, we may suppose that there exists a sequence;\\

$g_{n_{i}}^{r_{i}}\subset
g_{n_{i-1}}^{r_{i-1}}\subset\ldots\subset g_{n_{1}}^{r_{1}}$\\

with $r_{i}=r_{1}-(i-1)$ and $n_{i}\leq n_{1}-2(i-1)$. $(*)$\\

We need to show that this sequence terminates. By Lemma 2.24, we
always have that $r_{i}\leq n_{i}$. Combining this with $(*)$, we
have that;\\

$r_{1}-(i-1)\leq n_{1}-2(i-1)$\\

Therefore;\\

$i\leq n_{1}-r_{1}+1=n(n-1)-{(n-1)(n+2)\over
2}+1={(n-1)(n-2)\over 2}+1$ $(**)$\\

Now suppose that equality is attained in $(**)$, then we would
have that;\\

$r_{i}=r_{1}-(i-1)={(n-1)(n+2)\over
2}-{(n-1)(n-2)\over 2}=2(n-1)\geq 2$\\

$n_{i}\leq n_{1}-2(i-1)=n(n-1)-(n-1)(n-2)=2(n-1)$ $(***)$\\

This implies that $n_{i}\leq r_{i}$ and hence $n_{i}=r_{i}$.
Therefore, we must have that there exists a $g_{m}^{m}\subset
g_{n_{i}}^{r_{i}}$ for each $i$ with $m\geq 2$ and the sequence
terminates in fewer than ${(n-1)(n-2)\over 2}+1$ steps. The final
$g_{n_{i}}^{r_{i}}$ in the sequence then defines a birational map
from $C$ to $C_{1}\subset P^{w}$ without multiple points ($w\geq
2$), as required.

\end{proof}

\begin{rmk}
The terminology of transverse $g_{n}^{r}$ is partly motivated by
the following fact. Suppose that a transverse $g_{n}^{r}$ on $C$
is defined by a linear system $\Sigma$, possibly after removing
some fixed point contribution, then, if $x\in C\setminus
Base(\Sigma)$ is non-singular, there exists an algebraic form
$\phi_{\lambda}$ from $\Sigma$ such that $\phi_{\lambda}(x)=0$ but
$\phi_{\lambda}$ is not algebraically tangent to $C$ at $x$. The
proof of this follows straightforwardly from the definition of
transversality and the fact that, if $\phi_{\lambda}$ is
algebraically tangent to $C$ at $x$, then
$I_{italian}(x,C,\phi_{\lambda})\geq 2$ (see the proof of Lemma
4.2). In modern terminology, one calls this property separating
tangent vectors. See, for example Proposition 7.3 of \cite{H}. The
full motivation, however, comes from Theorem 6.10.
\end{rmk}

\end{section}
\begin{section}{The Method of Conic Projections}

The purpose of this section is to explore the Italian technique of
projecting a curve onto a plane. The method of conic projections
is extremely old and can be found in \cite{Piero}, where
projective notions are explicitly incorporated in the discussion
of perspective. Severi himself also wrote an article on the
subject of \cite{Piero} in \cite{Sev1}. We assume that we are
given a projective algebraic curve $C\subset P^{w}$ for some $w>2$
and that $C$ is not contained in any
hyperplane section (otherwise reduce to this lower dimension).\\

The construction;\\

Let $\Omega\subset P^{w}$ be a plane of dimension $w-k-1$ and
$\omega\subset P^{w}$ a plane of dimension $2\leq k<w$ such that
$\Omega\cap\omega=\emptyset$. We define the projection of $C$
from $\Omega$ to $\omega$ as follows;\\

Let $P\in C$. We may assume that $P$ does not lie on $\Omega$. Let
$<\Omega,P>$ be the intersection of all hyperplanes containing
$\Omega$ and $P$. It is a plane of dimension $w-k$. Now, by
elementary dimension theory, we must have that;\\

$dim(<\Omega,P>\cap \omega)\geq k+(w-k)-w=0$\\

We may exclude the case that $<\Omega,P>$ and $\omega$ intersect
in a line $l$, as then $\Omega$ and $l$ would intersect in a point
$Q$, contradicting the fact that $\Omega\cap\omega=\emptyset$.
Hence, $<\Omega,P>\cap\omega$ defines a point $pr(P)$. We may
repeat this construction for the cofinitely many points $U\subset
C$
which do not lie on $\Omega$. Now, consider the statement;\\

$\phi(y)\equiv [y\in\omega\wedge\exists x\exists
w(x\in\Omega\wedge w\in U\wedge y\in l_{xw})]$\\

By elimination of quantifiers for algebraically closed fields,
this clearly defines an algebraic set consisting of $\{pr(w):w\in
U\}$. We call this the projection $pr(U)$ of $U$. As $U$ is
irreducible, it follows that $pr(U)$ is irreducible. Moreover,
$pr(U)$ has dimension $1$, otherwise $pr(U)$ would consist of a
single point $Q$, in which case $U$ and therefore $C$ would be
contained in the plane $<\Omega,Q>$, contradicting the assumption.
We define the projection $proj(C)$ of $C$ from $\Omega$ to
$\omega$ to be the closure $\overline{pr(U)}$ of $pr(U)$ in
$\omega$. We can define a correspondence
$\Gamma\subset U\times pr(U)$ by;\\

$\Gamma(w,y)\equiv [y\in pr(U)\wedge w\in U\wedge\exists
x(x\in\Omega\wedge y\in l_{xw})]$\\

We define the associated correspondence $\Gamma_{pr}\subset
C\times proj(C)$ to be the Zariski closure $\overline{\Gamma}$.
Note that, in the case when $\Omega\cap C=\emptyset$, the
correspondence $\Gamma$ defines an algebraic \emph{function} (in
the sense of model theory) $pr$ from $C$ to $pr(C)$. By the model
theoretic description of definable closure in algebraically closed
fields of characteristic $0$, $pr$ defines a morphism from $C$ to
$pr(C)$. See the final section for the problem in non-zero
characteristic.

\begin{rmk}
Note that when $\Omega$ is a point $P$ and $\omega$ is a
hyperplane, the construction is equivalent to forming the cone;\\

$Cone(C)=\bigcup_{x\in C}l_{xP}$\\

and taking the intersection $\omega\cap Cone(C)$. This is the
reason for the terminology of "conic projections". The case when
$w=3$ is explicitly discussed in \cite{Piero}, $P$ represents the
eye of the observer wishing to obtain a representation of a curve
on a plane.

\end{rmk}

We now prove the following general lemma on conic projections;\\

\begin{lemma}
Let $w\geq 3$ and suppose that $\{A,B\}$ are independent generic
points of $C$. Then the line $l_{AB}$ does not otherwise meet the
curve.
\end{lemma}

\begin{proof}
Suppose, for contradiction, that we can find independent generic
points $\{A,B\}$ on $C$ such that $l_{AB}$ intersects $C$ in a new
point $P$. As $A$ and $B$ are generic, they are non-singular
points on the curve $C$, hence we can define the tangent lines
$l_{A}$ and $l_{B}$ at these points. We claim the
following;\\

For any hyperplane $H_{\lambda}$ containing $l_{A}$ and passing through $A$;\\

$I_{italian}(A,C,H_{\lambda})\geq 2$ (*)\\

This is the converse to a result already proved in Lemma 2.10
above. It can be proved in a similar way, namely fix a birational
map $\Phi_{\Sigma}$ between $C$ and $C_{1}\subset P^{2}$ such that
$A$ and its corresponding non-singular point $A'$ on $C_{1}$ lie
in the fundamental sets $V_{\Phi_{\Sigma}}$ and
$W_{\Phi_{\Sigma}}$ (use the fact that $A$ is generic). By the
chain rule, the corresponding $\phi_{\lambda}$ passes through $A'$
and is tangent to the curve $C_{1}$. Hence, by results of
\cite{depiro2}, we have that;\\

$I_{italian}(A',C_{1},\phi_{\lambda})\geq 2$\\

Using the technique of Section 2, it follows easily that $(*)$
holds as required.\\

Now choose a generic hyperplane $H$ in $P^{w}$. Let $pr$ be the
projection of $C$ from $P$ to $H$. Let $D=pr(A)=pr(B)$. We have
that $pr(C)$ is defined over $P$ and, moreover, that $D$ is
generic in $pr(C)$. This follows as, otherwise, $dim(D/P)=0$,
therefore, as $\{A,B\}\subset l_{PD}$, $dim(A,B/P)=0$, which
implies $dim(A,B)=1$, contradicting the fact that $\{A,B\}$ were
independent generic. Hence, $D$ defines a non-singular point on
the curve $pr(C)$. Now, let $l_{1}$ and $l_{2}$ be the projections
of the lines $l_{A}$ and $l_{B}$. (We will deal with the
degenerate case when one of these projections is a point below
$(\dag)$). Clearly $l_{1}$ and $l_{2}$ pass through $D$. We claim
that they have the property $(*)$. Let $H_{\lambda}$ be a
hyperplane of $H$ passing through $l_{1}$. Then the hyperplane
$<H_{\lambda},P>$ of $P^{w}$ passes through $l_{A}$. Let
$\lambda'\in{\mathcal V}_{\lambda}$ be generic in $Par_{H}$, then,
we may assume that $D\notin pr(C)\cap H_{\lambda'}$. Therefore,
the corresponding hyperplane $<H_{\lambda'},P>$ does not pass
through $A$. Now, using the property $(*)$ of $l_{A}$, Lemma 2.10
and the fact that $<H_{\lambda'},P>$ is an infinitesimal variation
of $<H_{\lambda},P>$, we can find distinct
$\{A_{1},A_{2}\}\subset{\mathcal V}_{A}\cap C\cap
<H_{\lambda'},P>$. It follows that
$\{pr(A_{1}),pr(A_{2})\}\subset{\mathcal V}_{D}\cap pr(C)\cap
H_{\lambda'}$. We claim that $pr(A_{1})$ and $pr(A_{2})$ are
distinct. Suppose not, then $pr(A_{1})=pr(A_{2})=D'\in{\mathcal
V}_{D}\cap pr(C)$. Consider the following finite cover $F\subset pr(C)\times P^{w}$;\\

$F(y,x)\equiv y\in pr(C)\wedge x\in C\cap l_{yP}$\\

As we have shown, $(DA)$ is generic for this cover, hence, by
properties of Zariski structures, $mult_{(DA)}(F/pr(C))=1$.
However, this contradicts the fact that we can find
$D'\in{\mathcal V}_{D}\cap pr(C)$ and distinct
$\{A_{1},A_{2}\}\subset {\mathcal V}_{A}$ such that $F(D'A_{1})$
and $F(D'A_{2})$. Therefore, we have shown that
$I_{italian}(D,pr(C),H_{\lambda})\geq 2$. Hence, $l_{1}$ and
$l_{2}$ have the property $(*)$. We now claim that $l_{1}=l_{2}$
(**). Suppose not, then, if $l_{D}$ is \emph{the} tangent line to
$pr(C)$ at $D$, it must be distinct from say $l_{1}$. Choose a
hyperplane $H_{\lambda}$ of $\omega$ passing passing through
$l_{1}$ but not through $l_{D}$. We have that
$I_{italian}(D,pr(C),H_{\lambda})\geq 2$, but, using part of the
proof of Lemma 2.10, $H_{\lambda}$ must then pass through $l_{D}$.
Hence,
$(**)$ is shown.\\

Now, by $(**)$, we have that the tangent lines $l_{A}$ and $l_{B}$
both lie on the plane spanned by $l_{B}$ and $l_{AB}$. Consider
the statement;\\

$x\in NonSing(C)$ and the tangent line $l_{x}$ lies on the plane
spanned by $l_{B}$ and $l_{xB}$.\\

It defines an algebraic subset of $NonSing(C)$ over $B$, and,
moreover, as it holds for $A$, which was assumed to be generic in
$C$ and independent from $B$, it defines an open subset $V$ of $C$. (***)\\

Now choose a generic hyperplane $H$ and let $pr_{B}(C)$ be the
projection of $C$ from $B$ to $H$. By the proof (below) that the
degenerate case $(\dag)$ cannot occur, and the argument above, we
can find an open $W\subset V$, such that, for each $x\in W$,
$l_{x}$ is projected to $l_{pr_{B}(x)}$. We claim that;\\

For $y\in pr_{B}(W)$, the $l_{y}$ intersect in a point $Q$. $(****)$\\

Let $y_{1}$ and $y_{2}$ be independent generic points in $pr_{B}(W)$, then we can find
corresponding $x_{1}$ and $x_{2}$ in $W$ such that
$<l_{By_{1}},l_{y_{1}}>=<l_{Bx_{1}},l_{x_{1}}>$ and
$<l_{By_{2}},l_{y_{2}}>=<l_{Bx_{2}},l_{x_{2}}>$, $(*****)$. As $x_{1}$ and
$x_{2}$ lie in $V$, by $(***)$, we have that;\\

$<l_{Bx_{1}},l_{x_{1}}>\cap <l_{Bx_{2}},l_{x_{2}}>=l_{B}$ $(******)$\\

We can assume that $l_{y_{1}}\neq l_{y_{2}}$, $(\dag\dag)$, as if $l_{y_{1}}=l_{y_{2}}$, then, considering $\{y\in pr_{B}(C):l_{y}=l_{y_{1}}\}$, we have all the tangent lines $l_{y}$ of $C$ equal $l_{y_{1}}$. Now, if $pr_{B}(C)$ is not a line $l$, by Lemma 2.17, we can find a hyperplane $H$, with $I_{italian}(p,pr_{B}(C),H)=1$, for $p\in (pr_{B}(C)\cap H)$, and $Card(pr_{B}(C)\cap H)\geq 2$. Picking $p,q\in (pr_{B}(C)\cap H)$,  we have that $l_{p}=l_{q}$, therefore, $l_{pq}=l_{p}\subset H$. Then, by $(*)$, we have that $I_{italian}(p,C,H)\geq 2$, a contradiction. Then $pr_{B}(C)$ is a line, which, we may assume, is not the case. Now, by $(*****),(******)$, we have that $<l_{y_{1}},l_{By_{1}}>\cap <l_{y_{2}},l_{By_{2}}>=l_{B}$, $(*******)$. Then, using $(\dag\dag)$, and applying $pr_{B}$ to both sides of $(*******)$, we have that $l_{y_{1}}\cap l_{y_{2}}=pr_{B}(l_{B})=Q$, as required. Now, to obtain the final contradiction, we need to show that $(****)$ cannot occur. This also covers the degenerate case $(\dag)$. We show;\\

If $C$ is a projective algebraic curve in $P^{w}$, with the
property that there exists an open $W\subset Nonsing(C)$ such that
each $l_{x}$ for $x\in W$ intersects in a point $Q$, then $C$ is a
line $l$. $(\dag)$\\

Choose a generic hyperplane $H$ and consider the projection
$pr_{Q}(C)$ of $C$ from $Q$ to $H$. If this projection is a point,
then $C$ is a line $l$. Hence, we may assume that $pr_{Q}(C)$ is a
projective algebraic curve in  $H$. Suppose that $x$ is generic in
$pr_{Q}(C)$, with $pr_{Q}(y)=x$ and $y$ generic in $C$. Then $x$
and $y$ are nonsingular and, moreover, if $H_{\lambda}\subset H$
is \emph{any} hyperplane passing through $x$, the corresponding
hyperplane $<H_{\lambda},Q>$ passes through $y$. By the
assumption, $<H_{\lambda},Q>$ contains the tangent line $l_{y}$.
Hence, by the proof given above,\\

$min\{I_{italian}(y,C,<H_{\lambda},Q>),I_{italian}(x,C,H_{\lambda})\}\geq
2$.\\

Now, as $x$ is nonsingular, we can find a hyperplane $H_{\lambda}$
passing through $x$, not containing $l_{x}$. By the usual
argument, given above, we obtain a contradiction. Therefore,
$(\dag)$ is shown.\\

As we assumed that our original $C$ was not contained in a
$P^{2}$, the proof of the lemma is shown. The lemma also holds in
non-zero characteristic even though, surprisingly, the result
$(\dag)$ turns out to be false. We will deal with these problems
in the final section.

\end{proof}

\begin{rmk}
The proof of the above lemma is attributed to Castelnuovo in
\cite{Sev}, but I have been unable to find a convenient reference.
\end{rmk}

\begin{defn}
We will define a correspondence $\Gamma\subset C_{1}\times C_{2}$,
where $C_{1}$ and $C_{2}$ are projective algebraic curves, to be
generally biunivocal, if, for generic $x\in C_{1}$ there exists a
unique generic $y\in C_{2}$ such that $\Gamma(x,y)$ and
$\Gamma(x',y)$ implies that $x=x'$ (*). We will say that $\Gamma$
is biunivocal at $x$ if $(*)$ holds.

\end{defn}

We then have as an immediate consequence of the above that;\\

\begin{lemma}
Let $C\subset P^{w}$, $(w\geq 3)$, be a projective algebraic
curve, not contained in any hyperplane section. Fix some
hyperplane $H$. Then, if $P$ is a generic point of the curve, the
projection $pr_{P}$ from $P$ to $H$ is generally biunivocal on
$C$.
\end{lemma}

We now note the following;\\

\begin{lemma}
Let $C\subset P^{w}$ be as in the above lemma. Fix a plane
$\omega$ of dimension $w-k\geq 2$. Then, if
$\{P_{1},\ldots,P_{k}\}$ are independent generic points of $C$,
the projection from $\Omega=<P_{1},\ldots,P_{k}>$ to $\omega$ is
generally biunivocal.
\end{lemma}

\begin{proof}
Choose a sequence of planes
$\omega=\omega_{0}\subset\omega_{1}\subset\ldots\subset\omega_{k-1}$
such that $dim(\omega_{i})=dim(\omega)+i$ for $i\leq k-1$, with
field of definition equal to that of $\omega$. The projection
$pr_{P_{1}}$ from $P_{1}$ to $\omega_{k-1}$ is generally
biunivocal, and $\{pr_{P_{1}}(P_{2}),\ldots,pr_{P_{1}}(P_{k})\}$
define independent generic points on $pr_{P_{1}}(C)$. Repeating
the argument $k$ times, we obtain that the composition
$pr_{P_{k}}\circ\ldots\circ pr_{P_{1}}$ is biunivocal as a
projection to $\omega$. We claim that
$pr_{<P_{1},\ldots,P_{k}>}=pr_{P_{k}}\circ\ldots\circ pr_{P_{1}}$
(*). This can be checked for a generic point $x$ of $C$. Suppose,
inductively, that
$pr_{<P_{1},\ldots,P_{i}>}(x)=pr_{P_{i}}\circ\ldots\circ
pr_{P_{1}}(x)=z$. Let $u=pr_{P_{i+1}}(z)$, then there exists
$z'=<P_{1},\ldots,P_{i+1}>\cap\omega_{k-i}$ such that
$l_{zz'}\cap\omega_{k-i-1}=u$. We have that
$<P_{1},\ldots,P_{i},x>\cap\omega_{k-i}=z$, hence
$<P_{1},\ldots,P_{i},P_{i+1},x>\cap \omega_{k-i}=l_{zz'}$,
therefore $<P_{1},\ldots,P_{i+1},x>\cap\omega_{k-i-1}=u$. This
shows that $pr_{<P_{1},\ldots,P_{i+1}>}(x)=u$ and, therefore, by
induction, that
$pr_{<P_{1},\ldots,P_{k}>}(x)=pr_{P_{k}}\circ\ldots\circ
pr_{P_{1}}(x)$. Hence $(*)$ is shown. It follows immediately that
the projection from $\Omega=<P_{1},\ldots,P_{k}>$ to $\omega$ is
biunivocal on $C$ as required.
\end{proof}

As a consequence, we note the following, which is of independent
interest;\\

\begin{lemma}
Let $C\subset P^{w}$ be as in the previous lemma. Let
$\{P_{1},\ldots,P_{k}\}$ be independent generic points on $C$ for
$k\leq w-1$. Then the hyperplane $<P_{1},\ldots,P_{k}>$ does not
otherwise encounter $C$.

\end{lemma}

\begin{proof}
By the previous lemma, the projection
$pr_{<P_{1},\ldots,P_{k-1}>}$ is generally biunivocal. Suppose
there existed another intersection $Q$ of $<P_{1},\ldots,P_{k}>$
with $C$. Then
$pr_{<P_{1},\ldots,P_{k-1}>}(Q)=pr_{<P_{1},\ldots,P_{k-1}>}(P_{k})$,
contradicting the definition of generally biunivocal.
\end{proof}

Using Lemma 4.6, we obtain an alternative proof of Theorem 1.33;\\

\begin{theorem}
Let $C\subset P^{w}$ be a projective algebraic curve, then $C$ is
birational to a plane projective curve.
\end{theorem}

\begin{proof}
By Lemma 4.6, we can find a generally biunivocal correspondence
between $C$ and $C_{1}\subset P^{2}$. In characteristic $0$, we
can therefore find $U\subset C$ and $V\subset C_{1}$ such that
$U\cong V$. This defines a birational map
$\Phi_{\Sigma}:C\leftrightsquigarrow C_{1}$.
\end{proof}

We now use the techniques of this section to prove some further
results which will be required later. We have first;\\

\begin{lemma}
Let $C\subset P^{w}$ be a projective algebraic curve, not
contained in any hyperplane section, and $x\in C$. Then there
exists a plane $\Omega$ of dimension $w-3$ and a plane $\omega$ of
dimension $2$ such that the projection $pr$ from $\Omega$ to
$\omega$ is generally biunivocal and, moreover, biunivocal at $x$.
\end{lemma}
\begin{proof}
Let $H$ be a hyperplane containing $x$, then, by the assumption on
$C$, it intersects $C$ in a finite number $r$ of points. It
follows that we can find $P$ generic in $H$ such that $l_{Px}$
does not otherwise encounter the curve $C$. Now choose a further
hyperplane $H'$ not containing $P$. Let $Q$ be generic on $C$ with
$Q$ independent from $P$. Suppose that $l_{PQ}$ intersects $C$ in
a new point $R$, $(*)$. Then $\{Q,R\}$ must form a generic
independent pair. Otherwise, as $P\in l_{QR}$, we would have that
$dim(P/Q)=0$ and therefore $dim(P)=0$ which is a contradiction.
Now, we can imitate the proof of Lemma 4.2 for the independent
pair $\{Q,R\}$ and the projection $pr_{P}$ to obtain a
contradiction. It follows that $(*)$ cannot occur, hence the
projection $pr_{P}$ is generally biunivocal and, moreover, by
construction, biunivocal at $x$. Now, repeat this argument $w-2$
times, to find a sequence of points $\{P_{1},\ldots,P_{w-2}\}$ and
planes
$\{\omega=\omega_{1}\subset\omega_{2}\subset\ldots\subset\omega_{w-1}=P^{w}\}$
such that the projection $pr_{P_{j}}$ from $\omega_{j+1}$ to
$\omega_{j}$ is generally biunivocal on $C$ and, moreover,
biunivocal at $x$, for $1\leq j\leq w-2$. It follows that the
projection $pr_{<P_{1},\ldots,P_{w-2}>}$ from
$<P_{1},\ldots,P_{w-2}>$ to $\omega$ has the required property.
\end{proof}

\begin{defn}
Let $C\subset P^{w}$ be a projective algebraic curve. Let $pr$ be
a projection from $\Omega$ to $\omega$ such that $pr$ is
biunivocal at $x$. We call $<\Omega,x>$ the axis of the projection
at $x$.
\end{defn}

We now refine Lemma 4.9 to ensure that the axis of projection is
not in "special position" with respect to $x$.\\

\begin{defn}{Special Position}\\

Let $C\subset P^{w}$ be a projective algebraic curve and $x$ an
$s$-fold point on $C$, (see Definition 3.1). Let $\Omega$ be a
plane passing through $x$. We say that $\Omega$ is in special
position if for every hyperplane $H$ containing $\Omega$;\\

$I_{italian}(x,C,H)\geq s+1$.\\

\end{defn}

\begin{lemma}
Let hypotheses be as in Lemma 4.9. Let $x$ be an $s$-fold point of
$C$. Then we can obtain the conclusion of Lemma 4.9, with the
extra requirement that the axis of projection $<\Omega,x>$ is
\emph{not} in special position.

\end{lemma}

\begin{proof}
As $x$ is $s$-fold on $C$, we can find a hyperplane $H$ passing
through $x$ such that $I_{italian}(x,C,H)=s$. Now imitate the
proof of Lemma 4.9 by choosing a point $P_{1}$, generic on $H$
such that $l_{P_{1}x}$ does not otherwise encounter the curve $C$.
Repeating the argument for the projected hyperplane
$pr_{P_{1}}(H)$ in $P^{w-1}$, we obtain a series
$\{P_{1},\ldots,P_{w-2}\}$ as in the conclusion of Lemma 4.9 and,
by construction, the axis of projection
$<P_{1},\ldots,P_{w-2},x>\subset H$. Hence, $<\Omega,x>$ is not in
special position.

\end{proof}

We can now prove;\\

\begin{lemma}
Let $C\subset P^{w}$ be a projective algebraic curve. Suppose that
$x\in C$ and $pr$ is a projection from $\Omega$ to $\omega=P^{2}$
satisfying the conclusion of Lemma 4.12. Then $x$ is $s$-fold on
$C$ iff $pr(x)$ is $s$-fold on $pr(C)$.
\end{lemma}

\begin{proof}
Let $V=\{x\in C:pr\ is\ biunivocal\ at\ x\}$ and $W=pr(V)$. $V$
and $W$ are open subsets of $C$ and $pr(C)$ respectively and are
in bijective correspondence. Let $\{H_{\lambda}:\lambda\in
Par_{H}\}$ be the independent system $\Sigma$ of lines in $P^{2}$.
We then obtain a corresponding independent system $\Sigma'$ on
$P^{w}$ defined by $\{<\Omega,H_{\lambda}>:\lambda\in Par_{H}\}$.
Clearly, $\Sigma$ has no base points on $pr(C)$. We claim that
$\Sigma'$ has no base points on $C$. Suppose that there existed a
base point $y\in C$ for $\Sigma'$, then clearly $pr(y)$ would be a
base point for $\Sigma$ on $pr(C)$, which is a contradiction. We
now claim that for $y\in V$
and corresponding $pr(y)\in W$, that;\\

$I_{italian}(y,C,<\Omega,H_{\lambda}>)=I_{italian}(pr(y),pr(C),H_{\lambda})$ $(*)$\\

By the fact that $\Sigma'$ has no base points, and results of
Section 2, we have that;\\

$I_{italian}(y,C,<\Omega,H_{\lambda}>)=I_{italian}^{\Sigma'}(y,C,<\Omega,H_{\lambda}>)$\\

Therefore, the result follows immediately from the definition of
$I_{italian}^{\Sigma'}$ and the fact that $V$ and $W$ are in
bijective correspondence.\\

We now claim that $deg(C)=deg(pr(C))$. Suppose that
$deg(pr(C))=d$, then, by the proof of Lemma 2.4, we can find a
generic plane $H_{\lambda}$ intersecting $pr(C)$ transversely at
$d$ distinct points inside $W$. By $(*)$, the corresponding
hyperplane $<\Omega,H_{\lambda}>$ intersects $C$ transversely at
$d$ distinct points inside $V$. Hence, by general properties of
infinitesimals, $deg(C)=d$ as well.\\

Now, let $\Sigma''$ be the independent system defined by the lines
in $P^{2}$ passing through $pr(x)$. By Bezout, it defines a
$g_{d}^{r}(\Sigma'')$ on $pr(C)$. Suppose that $pr(x)$ is
$s_{1}$-fold on $pr(C)$, then, after removing the fixed point
contribution at $pr(x)$, we obtain a $g_{d-s_{1}}^{r}$ on $pr(C)$
without fixed points. Let $\Sigma'''$ be the independent system
defined by $<\Omega,\Sigma''>$. It consists exactly of the
hyperplanes passing through the axis of projection $<\Omega,x>$.
Again, by hyperspatial Bezout, it defines a $g_{d}^{r}(\Sigma''')$
on $C$. As $pr$ was assumed to be biunivocal at $x$, this
$g_{d}^{r}(\Sigma''')$ has a unique fixed point at $x$. Moreover,
by the assumption on $pr$ that the axis of projection is not in
special position, if $x$ is $s_{2}$-fold on $C$, then its fixed
point contribution at $x$ is $s_{2}$. Hence, after removing this
contribution, we obtain a $g_{d-s_{2}}^{r}$ on $C$ without fixed
points. Now, using Lemma 2.17, for generic $\lambda$, the weighted
set $W_{\lambda}$ for $g_{d-s_{2}}^{r}$ consists of $d-s_{2}$
points, each counted once, lying inside $V$. Using $(*)$, the
corresponding $V_{\lambda}$ for $g_{d-s_{1}}^{r}$ consists of
$d-s_{2}$ points, each counted once, lying inside $W$. By Lemma
2.17 again, we must have that $d-s_{1}=d-s_{2}$, hence
$s_{1}=s_{2}$ as required.
\end{proof}

As an easy consequence of the above lemma, we have;\\

\begin{lemma}
Let $C\subset P^{w}$ be a projective algebraic curve and suppose
that $x\in C$, then $x$ is non-singular iff $x$ is not multiple.
\end{lemma}

\begin{proof}
Choose a projection $pr$ as in Lemma 4.13. Then, $x$ is not
multiple iff $pr(x)$ is not multiple. By Lemma 3.2, $pr(x)$ is not
multiple iff $pr(x)$ is non-singular. In characteristic $0$, using
the fact that $pr$ is generally biunivocal and an elementary model
theoretic argument, we can find an inverse morphism
$\phi:W'\rightarrow V'$, where $\{W',V'\}$ are open subsets of
$\{W,V\}$ given in the previous lemma. It follows that $V'\cong
W'$. As $pr(x)$ is non singular, we can extend the morphism $\phi$
to include $pr(x)$ in its domain, by biunivocity of $pr$ at $x$,
we must have that $\phi(pr(x))=x$. Hence, we may assume that $x$
lies in $V'$ and $pr(x)$ lies in $W'$. Therefore, $pr(x)$ is
non-singular iff $x$ is non-singular. Hence, the result is shown.

\end{proof}

Combining this result with Theorem 3.3, we then have the
following;\\

\begin{theorem}
Let $C\subset P^{w}$ be a projective algebraic curve, then $C$ is
birational to a non-singular projective algebraic curve
$C_{1}\subset P^{w'}$.
\end{theorem}

We finish this section with the following application of the
method of Conic Projections, the result will not be required later
in the paper. Some part of the proof will require methods
developed in Sections 5 and 6, we refer the reader to Definition
6.3 for the terminology "node". I have not seen a reasonable
algebraic proof of this result.\\

\begin{theorem}
Let $C\subset P^{w}$ be a projective algebraic curve, then $C$ is
birational to a plane projective algebraic with at most nodes as
singularities.
\end{theorem}

\begin{proof}
We may, by the previous theorem, assume that $C$ is non-singular.
We first reduce to the case $w=3$. That is, we claim that $C$ is
birational to a non-singular projective algebraic curve $C'\subset
P^{3}$. Let $V_{3}$ be the variety of chords on $C$.
That is;\\

$V_{3}=\overline{\{\bigcup l_{ab}:(a,b)\in
{C^{2}\setminus\Delta}\}}$\\

We may assume that $C$ is not contained in any plane of dimension
$2$ ($\dag$) and that $w\geq 4$. We then claim that $V_{3}$
has dimension $3$. Let $a\in C$ and define;\\

$Cone_{a}=\overline{\{\bigcup l_{ab}:b\in({C\setminus a})\}}$\\

As $C$ is irreducible, so is $Cone_{a}$. (use the fact that for
any component $W$ of $Cone_{a}$, $\{b\in ({C\setminus
a}):l_{ab}\subset W\}$ is a closed subset of ${C\setminus a}$).
Suppose that $Cone_{a}$ has dimension $1$. Then there exists
$b$ such that;\\

$Cone_{a}=l_{ab}$\\

It follows immediately that $C$ is contained in $l_{ab}$,
contradicting the hypothesis ($\dag$). Hence, $Cone_{a}$ has
dimension $2$. Now, suppose that $V_{3}$ has dimension $2$. Then
there exist finitely many points
$\{a_{1},\ldots,a_{m}\}$ such that;\\

$V_{3}=Cone_{a_{1}}\cup\ldots\cup Cone_{a_{m}}$\\

In particular, we can find distinct $\{a,b\}\subset C$ such that
$Cone_{a}=Cone_{b}$. Choose $a'\in C$ distinct from $\{a,b\}$.
Then we may assume that $l_{aa'}$ has infinite intersection with
${Cone_{b}\setminus\delta(Cone_{b})}$ and therefore $C\subset
H_{aa'b}$, contradicting the hypothesis ($\dag$). Hence, $V_{3}$
has dimension $3$ as required. Now define the tangent variety on $C$;\\

$Tang(C)=\{\bigcup l_{a}:a\in C, l_{a}\ is\ the\ tangent\ line\
to\ C\ at\ a\}$\\

We claim that $Tang(C)$ is a closed subvariety of $V_{3}$ of
dimension $2$. In order to see that $Tang(C)\subset V_{3}$, it is
sufficient to prove that $l_{a}\subset Cone_{a}$. Consider the
following covers $F,F^{*}\subset (C\setminus \{a\})\times P^{w}$;\\

$F(x,y)\equiv x\in(C\setminus \{a\})\wedge y\in l_{ax}$\\

$F^{*}(x,\lambda)\equiv x\in (C\setminus \{a\})\wedge H_{\lambda}\
contains\ l_{ax}$\\

Let $\bar F$ and $\bar F^{*}$ be the closures of these covers
inside $C\times P^{w}$. We have the incidence relation $I\subset
P^{w}\times P^{w^{*}}$ given by;\\

$I(x,\lambda)\equiv x\in H_{\lambda}$\\

By definition, for $x\in (C\setminus \{a\})$, we have that;\\

$I({\bar F}(x),{\bar F^{*}}(x))$;\\

Hence, this relation holds at $\{a\}$ as well $(*)$. Now, by the
proof of Theorem 6.7, the fibre ${\bar F^{*}}(a)$ consists exactly
of the hyperplanes $H_{\lambda}$ containing $l_{a}$. Hence, by
$(*)$, we must have that the fibre ${\bar F}(a)$ defines $l_{a}$.
In order to complete the proof, observe that the projection
$pr:C\times P^{w}\rightarrow P^{w}$ defines a morphism
from $\bar F$ to $Cone_{a}$. Let;\\

$\Gamma_{pr}(x,y,z)\subset\bar F\times Cone_{a}\subset C\times
P^{w}\times P^{w}$\\

be the graph of this projection. Then, for $x\in
(C\setminus\{a\})$, the fibre $\Gamma_{pr}(x)=l_{ax}\times
l_{ax}$. Hence, the fibre $\Gamma_{pr}(a)=l_{a}\times l_{a}$. This
proves that $l_{a}\subset Cone_{a}$ as required. We clearly have
that $Tang(C)$ has dimension $2$ if $C$ is not contained in a line
$l$, contradicting the hypothesis $(\dag)$. The fact that
$Tang(C)$ is a closed subvariety of $V_{3}$ follows immediately
from the fact that $Tang(C)$ is a closed subvariety of $P^{w}$.\\

Now choose a plane $\Omega$ of dimension $w-4$ such that
$\Omega\cap V_{3}=\emptyset$. Let $pr$ be the projection from
$\Omega$ to $\omega$ where $\omega$ is a plane of dimension $3$
and $\Omega\cap\omega=\emptyset$. Let $pr(C)\subset\omega$ be the
projection of $C$. We claim that $pr(C)$ is birational to $C$ and
non-singular. First, observe that $pr$ is defined everywhere on
$C$ as $C\subset V_{3}$ and $\Omega\cap V_{3}=\emptyset$.
Secondly, we have that $pr$ is biunivocal everywhere on $C$. For
suppose that $pr(x)=pr(x')$, then
$<\Omega,x>=<\Omega,x'>=<\Omega,pr(x)>$, hence $l_{xx'}$
intersects $\Omega$, contradicting the fact that $\Omega\cap
V_{3}=\emptyset$. Finally, we have that for every $x\in C$, the
axis of projection $<\Omega,x>$ is \emph{not} in special position,
see Definition 4.11. This follows from the fact that, as $x$ is
non-singular, $<\Omega,x>$ would only be in special position if it
contained the tangent line $l_{x}$. As $Tang(C)\subset V_{3}$,
this would contradict the fact that $\Omega\cap V_{3}=\emptyset$.
Now, by Lemmas 4.13 and 4.14, we must have that $pr(x)$ is
non-singular for every $x\in C$, that is $pr(C)$ is non-singular.
We may, therefore, using the argument of Lemma 4.14 invert the
morphism $pr$ to obtain an isomorphism between $C$ and $pr(C)$. In
particular, $C$ and
$pr(C)$ are birational.\\

We now consider the curve $C'=pr(C)\subset P^{3}$. In order to
prove the theorem, it will be sufficient to show that $C'$ is
birational to a plane projective curve with at most nodes as
singularities. We now define the following $5$ varieties, we use
the notation $l$
to denote a line and $P$ to define a $2$-dimensional plane;\\

(i). $Tangent(C')=\{\bigcup l_{a}:a\in C,l_{a}\ the\ tangent\
line\
at\ C\}$\\

(ii). $Trisecant(C')=\overline{\{\bigcup l_{abc}:(a,b,c)\in
C^{3}\ distinct,l_{abc}\supset \{a,b,c\}\}}$\\

By a bitangent plane $P_{ab}$, we mean a hyperplane passing
through the tangent lines $l_{a}$ and $l_{b}$ for distinct
$\{a,b\}$ in $C$. By an osculatory plane $P_{a}$, we refer the
reader to Definition 6.3.

We define $Supp(P_{ab})=\{a,b\}$ and $Supp(P_{a})=C'\cap P_{a}$. We then consider;\\

(iii). $Bitangent\ Chord(C')=\overline{\{\bigcup l_{ab}:(a,b)\in
(C^{2}\setminus\Delta), l_{ab}\supset
Supp(P_{ab})\}}$\\

(iv). $Osculatory\ Chord(C')=\overline{\{\bigcup_{a\in C}
l_{ab}:(a,b)\in ({Supp(P_{a})^{2}\setminus\Delta})\}}$.\\

 By Remark 6.6, there exists a finite set $W\subset C'$, consisting
of points which are the origins of non-ordinary branches. We
let;\\

(v). $Singular\ Cone(C')=\bigcup_{y\in W}Cone_{y}$\\

As we have already seen, $Tangent(C')$ defines a $2$-dimensional
algebraic variety, unless $C'$ is contained in a line, which we
may assume is not the case. By Lemma 4.2, there are no trisecant
lines passing through an independent generic pair $(a,b)$ of
$C'^{2}$, unless $C'$ is contained in a plane, which again we may
exclude. The statement $D(x,y)\subset {C'^{2}\setminus\Delta}$, given by;\\

$D(x,y)\equiv \{(x,y)\in C'^{2}:(x\neq y)\wedge \not\exists
w(w\neq
x\wedge w\neq y\wedge w\in C'\cap l_{xy})\}$\\

is clearly constructible and has the property that
$\overline{D}=C'^{2}$. This clearly implies that
$dim(Trisecant(C'))\leq 2$. We now claim that, if $(a,b)$ is an
independent generic pair in $C^{'2}$, there is \emph{no} bitangent
plane $P_{a,b}$. Let $\Sigma=\{H_{\lambda}:\lambda\in P^{3*}\}$ be
the system of hyperplanes in $P^{3}$. $\Sigma$ has no fixed points
on $C'$, hence the linear condition that $H_{\lambda}$ passes
through $l_{a}$ has codimension $2$ and defines a $1$-dimension
family $\Sigma_{1}\subset\Sigma$. As we may assume that $C'$ is
not contained in any hyperplane section,$\Sigma_{1}$ has finite
intersection with $C'$. As $b$ is independent generic from $a$, it
cannot be a base point for the new system $\Sigma_{1}$. Hence,
using the results of Theorems 6.2 and 6.5, the condition that
$H_{\lambda}\in\Sigma_{1}$ passes through $l_{b}$ is also a
codimension $2$ linear condition. As $\Sigma_{1}$ is
$1$-dimensional, this implies the claim. Now consider the
statement $D'(x,y)\subset {C'^{2}\setminus\Delta}$,
given by;\\

$D'(x,y)\equiv \{(x,y)\in C'^{2}:(x\neq
y)\wedge\not\exists\lambda(H_{\lambda}\supset l_{a}\cup
l_{b})\}$\\

Again, this is constructible and has the property that
$\overline{D'}=C'^{2}$. This clearly implies that $dim(Bitangent\
Chord(C'))\leq 2$. For any given $a\in C$, there exists a unique
osculatory plane $P_{a}$ passing through $a$. Hence, there exist
finitely many lines of the form $l_{ab}$ for $b\in
(C^{2}\setminus\Delta)\cap Supp(P_{a})$. This clearly implies that
$dim(Osculatory\ Chord(C'))\leq 2$. Finally, by the above
consideration, we have that $dim(Singular\ Cone(C'))=2$\\

We now choose a point $P$ in $P^{3}$ such that $P$ lies outside
the above defined varieties. This is clearly possible as they all
have dimension at most $2$. It is easily shown that a generic $P$,
lies on finitely many bisecant lines of $C'$, $(\dag)$, see \cite{H}.
Let $\omega$ be a plane of dimension
$2$ not containing $P$ and let $pr_{P}(C')$ be the projection of
$C'$ from $P$ to $\omega$. We claim that $pr_{P}(C')$ is
birational to $C'$ with at most nodes as singularities. Clearly,
the birational property holds by $(\dag)$. Now, first,
suppose that $y\in pr_{P}(C')$ is singular, then, by Lemma 4.14,
it must be multiple. By Lemma 4.13 and the fact that $C'$ is
non-singular, this can only occur if either $|pr_{P}^{-1}(y)|\geq
2$ or there exists a unique $x$ such that $pr_{P}(x)=y$ and the
axis of projection $<P,x>$ is in special position. We may exclude
the second possibility on the grounds that $P$ is disjoint from
$Tang(C')$. We may therefore assume that $|pr_{P}^{-1}(y)|\geq 2$.
Now, we may exclude the possibility that $|pr_{P}^{-1}(y)|\geq 3$,
on the grounds that $P$ is disjoint from $Trisecant(C')$. Hence,
we may assume that $|pr_{P}^{-1}(y)|=2$. Now, as $C'$ is
non-singular, by Definition 5.2, there exist exactly $2$ branches
centred at $y$ on $pr_{P}(C')$. Moreover, we may assume that both
elements of $pr_{P}^{-1}(y)=\{x_{1},x_{2}\}$ are the origins of
ordinary branches, otherwise $P$ would lie inside $Singular\
Cone(C')$. If $P$ is situated on the osculatory plane $P_{x_{1}}$
of say $x_{1}$, then $(x_{1}x_{2})\subset Supp(P_{x_{1}})$ and,
hence, as $P\in l_{x_{1}x_{2}}$, we would have that $P$ lies in
$Osculatory\ Chord(C')$, which is not the case. Hence, we may
apply Theorem 6.4, to obtain that both branches have character
$(1,1)$. Finally, let $l_{a}$ and $l_{b}$ be the tangent lines to
the $2$ branches $\{\gamma_{y}^{1},\gamma_{y}^{2}\}$ at $y$. We
claim that $l_{a}$
and $l_{b}$ are distinct $(*)$. By definition, we have that;\\

$I_{italian}(y,\gamma_{y}^{1},pr_{P}(C'),l_{a})=2$\\

Hence, using Definition 5.15 and Lemma 5.17, we can find an
infinitesimal variation $l_{a}'$ of $l_{a}$ intersecting
$\gamma_{y}^{1}$ in distinct points $\{y',y''\}$. By Definition
5.15, we can find distinct points $\{x_{1}',x_{1}''\}$ in
$C'\cap{\mathcal V_{x_{1}}}$ such that $pr_{P}(x_{1}')=y'$ and
$pr_{P}(x_{1}'')=y''$. We clearly have that that $<l_{a}',P>$ is
an infinitesimal variation of $<l_{a},P>$ which intersects $C'$ in
$\{x_{1}',x_{1}''\}$. Therefore;\\

$I_{italian}(x_{1},C',<l_{a},P>)\geq 2$\\

and hence, by previous arguments, $<l_{a},P>$ contains the tangent
line $l_{x_{1}}$. As the axis of projection $l_{Px_{1}}$ was not
in special position, we must have that $pr_{P}(l_{x_{1}})=l_{a}$.
Similarly, $pr_{P}(l_{x_{2}})=l_{b}$. Now, if $l_{a}=l_{b}$, we
would have that $<l_{a},P>$ is a bitangent plane to $C'$,
$l_{x_{1}x_{2}}$ is a bitangent chord and $P$ would belong to the
variety $Bitangent\ Chord(C')$. As this is not the case, the claim
$(*)$ is proved. This completes the theorem.

\end{proof}

\end{section}
\begin{section}{A Theory of Branches for Algebraic Curves}

We now develop a theory of branches for algebraic curves, using
the techniques of the previous sections.\\

Let $C\subset P^{w}$ be a projective algebraic curve. Then, by
Theorem 4.15, we can find $C^{ns}\subset P^{w_{1}}$ which is
non-singular and birational to $C$. Let
$\Phi^{ns}:C^{ns}\leftrightsquigarrow C$ be the birational map
between $C^{ns}$ and $C$. As $C^{ns}$ is non-singular, $\Phi^{ns}$
extends to a totally defined morphism from $C^{ns}$ to $C$. As
usual, we let $\Gamma_{\Phi^{ns}}$ denote the graph of the
correspondence between $C^{ns}$ and $C$. It has the property that,
given any $x\in C^{ns}$, there exists a unique $y\in C$ such that
$\Gamma_{\Phi^{ns}}(x,y)$. We claim the following;\\

\begin{lemma}
Let $C_{1}\subset P^{w_{1}}$ and $C_{2}\subset P^{w_{2}}$ be
\emph{any} two non-singular birational models of $C$, with
corresponding morphisms $\Phi_{1}^{ns}$ and $\Phi_{2}^{ns}$. Then
there exists a unique \emph{isomorphism}
$\Phi:C_{1}\leftrightarrow C_{2}$ with the property that
$\Phi_{2}^{ns}\circ\Phi=\Phi_{1}^{ns}$ and
$\Phi_{1}^{ns}\circ\Phi^{-1}=\Phi_{2}^{ns}$.
\end{lemma}

\begin{proof}
As $\Phi_{2}^{ns}$ is a birational map, we can invert it to give a
birational map ${\Phi_{2}^{ns}}^{-1}:C\leftrightsquigarrow C_{2}$.
Then the composition
${\Phi_{2}^{ns}}^{-1}\circ\Phi_{1}^{ns}:C_{1}\leftrightsquigarrow
C_{2}$ is birational as well. As $C_{1}$ is non-singular, we can
extend this birational map to a totally defined morphism
$\Phi:C_{1}\rightarrow C_{2}$. By the same argument, we can find a
totally defined morphism $\Psi:C_{2}\rightarrow C_{1}$ with the
property that $\Psi$ inverts $\Phi$ as a birational map. We claim
that $\Psi$ inverts $\Phi$ as a morphism. We have that
$\Psi\circ\Phi:C_{1}\rightarrow C_{1}$ is a morphism with the
property that there exists an open $U\subset C_{1}$ such that
$\Psi\circ\Phi|U=Id_{U}$. Then $Graph(\Psi\circ\Phi)\subset
C_{1}\times C_{1}$ is closed irreducible and intersects the
diagonal $\Delta$ in an open dense subset. Therefore,
$Graph(\Psi\circ\Phi)=\Delta$, hence $\Psi\circ\Phi=Id_{C_{1}}$.
Similarly, one shows that $\Phi\circ\Psi=Id_{C_{2}}$. Therefore,
$\Phi$ is an isomorphism. By construction, $\Phi_{1}^{ns}$ and
$\Phi_{2}^{ns}\circ\Phi$ agree as birational maps and are totally
defined, hence, by a similar argument, they agree as morphisms.
Similarly, $\Phi_{2}^{ns}$ and $\Phi_{1}^{ns}\circ\Phi^{-1}$ agree
as morphisms. For the uniqueness statement, use the fact that any
$2$ isomorphisms $\Phi_{1}$ and $\Phi_{2}$, satisfying the
properties of $\Phi$, would agree on an open subset $U$ of
$C_{1}$, hence must be identical.

\end{proof}

Now fix a point $O$ of $C$. Let $\Gamma_{\Phi^{ns}}\subset
C^{ns}\times C$ be the graph of the correspondence defined above.
We denote by $\{O_{1},\ldots,O_{t}\}$ the fibre
$\Gamma_{\Phi^{ns}}(x,O)$. Note that, by the previous lemma, if we
are given $2$ non-singular models with correspondences
$\Gamma_{\Phi_{1}^{ns}}$ and $\Gamma_{\Phi_{2}^{ns}}$, then we
have a correspondence $O_{1}\sim O_{1}',O_{2}\sim
O_{2}',\ldots,O_{t}\sim O_{t}'$ between the fibres
$\Gamma_{\Phi_{1}^{ns}}(x,O)$ and $\Gamma_{\Phi_{2}^{ns}}(x,O)$,
given by $O_{j}\sim O_{j}'$ iff $\Phi(O_{j})=O_{j}'$, where $\Phi$
is the isomorphism given by the previous lemma. Moreover, as
$\Phi$ is an isomorphism, we also have a correspondence of
infinitesimal neighborhoods ${\mathcal V}_{O_{1}}\sim{\mathcal
V}_{O_{1}'},\ldots,{\mathcal V}_{O_{t}}\sim{\mathcal V}_{O_{t}'}$,
given by ${\mathcal V}_{O_{j}}\sim{\mathcal V}_{O_{j}'}$ iff
$\Phi:{\mathcal V}_{O_{j}}\cong{\mathcal V}_{O_{j}'}$, here we
mean that $\Phi$ is a bijection of sets. We now make
the following definition;\\

\begin{defn}
Let $C\subset P^{w}$ be a projective algebraic curve. Suppose that
$O\in C$. For $1\leq j\leq t$, we define the branches
$\gamma_{O}^{j}$ at $O$ to be the equivalence classes $[{\mathcal
V}_{O_{j}}]$ of the infinitesimal neighborhoods of $O_{j}$ in the
fibre $\Gamma_{\Phi^{ns}}(x,O)$ of \emph{any} non-singular model
$C^{ns}$ of $C$. We define $\gamma_{O}$ to be the union of the
branches at $O$.
\end{defn}

\begin{rmk}
Note that the definition does not depend on the choice of a
non-singular model $C^{ns}$, however, for computational purposes,
it \emph{is} convenient to think of a branch $\gamma$ as the
infinitesimal neighborhood of some $O_{j}$ in a fixed non-singular
model $C^{ns}$.
\end{rmk}

We now have the following lemmas concerning branches;\\

\begin{lemma}
Let $C\subset P^{w}$ and $O\in C$ be an $s$-fold point. Then $O$
is the origin of at most $s$-branches. In particular, a
non-singular point is the origin of a single branch.

\end{lemma}

\begin{proof}
Suppose that $O$ is the origin of $t$ branches. Then there exists
a non-singular model $C^{ns}$ and a birational map
$\Phi_{\Sigma}:C^{ns}\rightarrow C$ such that
$\Phi_{\Sigma}^{-1}(O)=\{O_{1},\ldots,O_{t}\}$. By a slight
extension to Remarks 1.32, and the fact that
$\{O_{1},\ldots,O_{t}\}$ are non-singular, we may assume that
$Base(\Sigma)$ is disjoint from this set (*). Let $\Sigma$ define
the system of hyperplanes in $P^{w}$. It defines a
$g_{d}^{r}(\Sigma)$ without fixed points on $C$, where $d=deg(C)$.
By Lemma 2.32, there is a corresponding $g_{d+I}^{r}(\Sigma)$ on
$C^{ns}$, where $I$ is the total fixed point contribution at
$Base(\Sigma)$. Now let $\Sigma'\subset\Sigma$ define the system
of hyperplanes passing through $O$. It defines a
$g_{d}^{r-1}(\Sigma')$ on $C$ with corresponding
$g_{d+I}^{r-1}(\Sigma')$ on $C^{ns}$. As $O$ is $s$-fold on $C$,
we can find $g_{d-s}^{r-1}\subset g_{d}^{r-1}(\Sigma')$ on $C$
without fixed points. By Lemma 2.31, we can find a corresponding
$g_{d-s}^{r-1}\subset g_{d+I}^{r-1}(\Sigma')$ on $C^{ns}$, without
fixed points (**). Now each $\phi_{\lambda}$ in $\Sigma'$ must
pass through $\{O_{1},\ldots,O_{t}\}$. Moreover, the fixed point
contribution from $\Sigma'$ at $Base(\Sigma)$ must be at least
$I$. Hence, by the assumption (*), the total fixed point
contribution from $Base(\Sigma')$ must be at least $t+I$. Hence,
by $(**)$, $d-s\leq(d+I)-(t+I)$, therefore $t\leq s$. The lemma is
proved.

\end{proof}

We now make the following definition;\\

\begin{defn}
Let $C_{1}\subset P^{w_{1}},C_{2}\subset P^{w_{2}},C_{3}\subset
P^{w_{3}}$ be birational projective algebraic curve with
correspondences $\Gamma_{\Phi_{1}}\subset C_{1}\times C_{2}$ and
$\Gamma_{\Phi_{2}}\subset C_{2}\times C_{3}$. We define the
composition $\Gamma_{\Phi_{2}}\circ\Gamma_{\Phi_{1}}\subset C_{1}\times C_{3}$ to be;\\

$\Gamma_{\Phi_{2}}\circ\Gamma_{\Phi_{1}}(xz)\equiv\exists y(y\in
C_{2}\wedge\Gamma_{\Phi_{1}}(xy)\wedge\Gamma_{\Phi_{2}}(yz))$\\

\end{defn}

\begin{lemma}
Let hypotheses be as in the previous lemma, then if
$\Phi_{1}:C_{1}\rightarrow C_{2}$ and $\Phi_{2}:C_{2}\rightarrow
C_{3}$ are birational maps;\\

$\Gamma_{\Phi_{2}}\circ\Gamma_{\Phi_{1}}=\Gamma_{\Phi_{2}\circ\Phi_{1}}$\\

\end{lemma}

\begin{proof}
The proof is a straightforward consequence of the fact that both
correspondences obviously agree on a Zariski open subset.
\end{proof}

\begin{lemma}{Birational Invariance of Branches}\\

Let $C_{1}\subset P^{w_{1}}$ and $C_{2}\subset P^{w_{2}}$ be
birational projective algebraic curves with correspondence
$\Gamma_{[\Phi]}\subset C_{1}\times C_{2}$. Fix $O\in C_{2}$ and
let $\{P_{1},\ldots,P_{s}\}=\Gamma_{[\Phi]}(x,O)$. Then,
$[\Phi]$ induces an injective map;\\

$[\Phi]^{*}:\gamma_{O}\rightarrow\bigcup_{1\leq k\leq s}\gamma_{P_{k}}$\\

and, moreover;\\

$[\Phi]^{*}:\bigcup_{O\in C_{2}}\gamma_{O}\rightarrow\bigcup_{O\in
C_{1}}\gamma_{O}$\\

is a bijection, with inverse given by $[\Phi^{-1}]^{*}$.\\
\end{lemma}
\begin{proof}
We first claim that there exists a non-singular model $C^{ns}$ of
$C_{1}$ and $C_{2}$ with morphisms $\Phi_{1}:C^{ns}\rightarrow
C_{1}$ and $\Phi_{2}:C^{ns}\rightarrow C_{2}$ such that
$\Phi_{2}=\Phi\circ\Phi_{1}$ and $\Phi_{1}=\Phi^{-1}\circ\Phi_{2}$
as birational maps (*). Choose a non-singular model $C^{ns}$ of
$C_{1}$ with morphism $\Phi_{1}:C^{ns}\rightarrow C_{1}$. Then
$\Phi\circ\Phi_{1}$ defines a birational map from $C^{ns}$ to
$C_{2}$. As $C^{ns}$ is non-singular, it extends(uniquely) to a
birational morphism $\Phi_{2}:C^{ns}\rightarrow C_{2}$. Clearly,
$\{C^{ns},\Phi_{1},\Phi_{2}\}$ have the property $(*)$. In order
to define the map $[\Phi]_{*}$, let $O\in C_{2}$ and suppose that
$\gamma_{O}^{j}$ is a branch corresponding to the equivalence
class $[{\mathcal V}_{O_{j}}]$, where $O_{j}$ lies in the fibre
$\Phi_{2}^{-1}(O)$. We claim that $\Phi_{1}(O_{j})$ lies in the
fibre $\Gamma_{[\Phi]}(x,O)$ $(**)$. As
$\Phi\circ\Phi_{1}=\Phi_{2}$ as birational maps, it follows, by
Lemma 1.21, that $\Gamma_{\Phi\circ\Phi_{1}}=\Gamma_{\Phi_{2}}$.
However, by the previous lemma,
$\Gamma_{\Phi\circ\Phi_{1}}=\Gamma_{\Phi}\circ\Gamma_{\Phi_{1}}$,
which gives $(**)$.  Then, if $\Phi_{1}(O_{j})=P_{k}$, and
$\Phi_{1}^{-1}(P_{k})=\{P'_{k1},\ldots,P'_{kr},\ldots,P'_{kl}\}$,
for $1\leq r\leq l$, we must have that $[{\mathcal
V}_{O_{j}}]=[{\mathcal V}_{P'_{kr}}]$, for some $r$. Hence, we can
set $[\Phi]^{*}(\gamma^{j}_{O})=\gamma^{r}_{P_{k}}$. One needs to
check that this definition does not depend either on the choice of
the non-singular model $C^{ns}$ or the choice of the birational
map $\Phi$ representing the correspondence $\Gamma_{[\Phi]}$. For
the first claim, after fixing a choice of $\Phi$, note that by the
requirement $(*)$, $\Phi_{1}$ is uniquely determined by the choice
of $\{C^{ns},\Phi_{2}\}$. By Lemma 5.1, for any $2$ choices
$\{C_{1}^{ns},\Phi^{1}_{2}\}$ and $\{C_{2}^{ns},\Phi^{2}_{2}\}$,
there exists a connecting isomorphism
$\Theta:(C_{1}^{ns},\Phi^{1}_{2})\leftrightarrow
(C_{2}^{ns},\Phi^{2}_{2})$, which must therefore be a connecting
isomorphism $\Theta:(C_{1}^{ns},\Phi^{1}_{1})\leftrightarrow
(C_{2}^{ns},\Phi^{2}_{1})$. The result then follows immediately by
definition of the branch as an equivalence class of infinitesimal
neighborhoods. For the second claim, after fixing a choice of
$\{C^{ns},\Phi_{1},\Phi_{2}\}$ and replacing $\Phi$ by an
equivalent $\Phi'$, one still obtains $(*)$ for the original
$\{C^{ns},\Phi_{1},\Phi_{2}\}$ and the result follows trivially.\\
The rest of the lemma now follows by proving that
$[\Phi^{-1}]^{*}$ inverts $[\Phi]^{*}$. This is trivial to check
using the particular choice $\{C^{ns},\Phi_{1},\Phi_{2},\Phi\}$.
\end{proof}
\begin{rmk}
Note that points are \emph{not} necessarily preserved by
birational maps, but, by the above, branches \emph{are} always
preserved.
\end{rmk}

We now refine the Italian definition of intersection multiplicity,
in order to take into account the above construction of a branch.
Suppose that $C\subset P^{w}$ is a projective algebraic curve. We
let $Par_{F}$ be the projective parameter space for \emph{all}
hypersurfaces of a given degree $e$. Now, fix a particular form
$F_{\lambda}$ of degree $e$, such that $F_{\lambda}$ has finite
intersection with $C$. The condition that a hypersurface of degree
$e$ contains $C$ is clearly linear on $Par_{F}$, (use Theorem 2.3
to show that the condition is equivalent to the plane condition
$P$ on $Par_{F}$ of passing through $de+1$ distinct points on the
curve $C$.) Now fix a maximal linear system $\Sigma$ containing
$F_{\lambda}$, such that $\Sigma$ has empty intersection with $P$.
It is trivial to check that the corresponding $g_{n}^{r}(\Sigma)$
on $C$ has no fixed points. Now fix a non-singular model $C^{ns}$
of $C$, with corresponding birational morphism $\Phi_{\Sigma'}$.
By the transfer result Lemma 2.31, we obtain a corresponding
$g_{n}^{r}$ without fixed points on $C^{ns}$. Let
$Card(O,V_{\lambda},g_{n}^{r})$ denote the number of times $O\in
C^{ns}$ is counted for this $g_{n}^{r}$ in the weighted set
$V_{\lambda}$. We then define;\\

\begin{defn}

$I_{italian}(p,\gamma_{p}^{j},C,F_{\lambda})=Card(p_{j},V_{\lambda},g_{n}^{r})$ $(*)$\\

where the branch $\gamma_{p}^{j}$ corresponds to $[{\mathcal
V}_{p_{j}}]$ in the fibre $\Gamma_{[\Phi_{\Sigma'}]}(x,p)$.\\

\end{defn}

\begin{rmk}

One may, without loss of generality, assume that $Base(\Sigma')$
is disjoint from the fibre $\Gamma_{[\Phi_{\Sigma'}]}(x,p)$. In
which case, the definition becomes the more familiar;\\

$I_{italian}(p,\gamma_{p}^{j},C,F_{\lambda})=Card(C^{ns}\cap
\overline{F}_{\lambda'}\cap {\mathcal V}_{p_{j}})$,
$\lambda'\in{\mathcal
V}_{\lambda}$ generic in $Par_{\Sigma}$.\\

Here, we have used the notation
$\{\overline{F}_{\lambda}:\lambda\in Par_{\Sigma}\}$ to denote the
family of "lifted" forms on $C^{ns}$ corresponding to the family
$\{F_{\lambda}:\lambda\in Par_{\Sigma}\}$ on $C$, by way of the
birational morphism $\Phi_{\Sigma'}$.\\
\end{rmk}
\begin{lemma}
Given $\Sigma$ containing $F_{\lambda}$. The definition $(*)$ does
not depend on the choice of the non-singular model $C^{ns}$ and
birational morphism $\Phi_{\Sigma'}$.
\end{lemma}
\begin{proof}

We divide the proof into the
following $2$ cases;\\

Case 1. $(C^{ns},\Phi)$ is fixed and we have $2$ presentations
$\Phi_{\Sigma_{1}}$ and $\Phi_{\Sigma_{2}}$ of the birational
morphism $\Phi$, as given by Lemma 1.20, such that
$Base(\Sigma_{1})$ possibly includes some of the points in $\Gamma_{[\Phi]}(x,p)$, while $Base(\Sigma_{2})$ is disjoint from this set.\\

Let $V=V_{\Phi_{\Sigma_{1}}}\cap V_{\Phi_{\Sigma_{2}}}$ with
corresponding $W\subset C$. Denote the weighted sets of the
\emph{2} given $g_{n}^{r}$ on $C^{ns}$,in the definition $(*)$,
corresponding to the presentations $\Phi_{\Sigma_{1}}$ and
$\Phi_{\Sigma_{2}}$, by  $\{V^{1}_{\lambda}\}$ and
$\{V^{2}_{\lambda}\}$. By the proof of Lemma 2.31, if $x\in V$,
then $x$ is counted $s$-times in $V^{1}_{\lambda}$ iff $x$ is
counted $s$-times in $V^{2}_{\lambda}$ $(\dag)$. Now, we claim
that
$Card(p_{j},V^{1}_{\lambda},g_{n}^{r})=Card(p_{j},V^{2}_{\lambda},g_{n}^{r})$.
This follows by $(\dag)$ and an easy application of, say Lemma
2.16, to witness both these cardinalities in the canonical set
$V$.\\

Case 2. We have $2$ non-singular models $(C^{ns}_{1},\Phi_{1})$
and $(C^{ns}_{2},\Phi_{2})$, with presentations
$\Phi_{\Sigma_{1}}$ and $\Phi_{\Sigma_{2}}$, such that
$Base(\Sigma_{1})$
is disjoint from $\Gamma_{[\Phi_{1}]}(x,p)$ and $Base(\Sigma_{2})$ is disjoint from $\Gamma_{[\Phi_{2}]}(x,p)$.\\

Using Lemma 5.1, we can find an isomorphism $\Phi:C_{1}\rightarrow
C_{2}$, such that $\Phi_{2}\circ\Phi=\Phi_{1}$. Let
$\{G_{\lambda}:\lambda\in Par_{\Sigma}\}$ denote the lifted forms
on $C^{ns}_{1}$ corresponding to the morphism $\Phi_{\Sigma_{1}}$
and $\{H_{\lambda}:\lambda\in Par_{\Sigma}\}$ denote the lifted
forms on $C^{ns}_{2}$ corresponding to the morphism
$\Phi_{\Sigma_{2}}$. We need to check that;\\

$Card(C^{ns}_{1}\cap G_{\lambda'}\cap{\mathcal
V}_{p_{j}})=Card(C^{ns}_{2}\cap H_{\lambda'}\cap{\mathcal
V}_{q_{j}})$, $\lambda'\in{\mathcal V}_{\lambda}\cap Par_{\Sigma}$
generic.\\

where $\Phi(p_{j})=q_{j}$ and $\{p_{j},q_{j}\}$ in
$\{C_{1}^{ns},C_{2}^{ns}\}$ respectively, correspond to the branch
$\gamma^{j}_{p}$.\\

Suppose that $Card(C^{ns}_{2}\cap H_{\lambda'}\cap{\mathcal
V}_{q_{j}})=n$. We may, without loss of generality, assume that
$Base(\Phi_{\Sigma_{3}})$ is disjoint from
$\Gamma_{[\Phi_{1}]}(x,p)$ in a particular presentation
$\Phi_{\Sigma_{3}}$ of the morphism $\Phi$ $(\dag)$.  Hence,
$Base(\Phi_{\Sigma_{2}}\circ\Phi_{\Sigma_{3}})$ is disjoint from
$\Gamma_{[\Phi_{1}]}(x,p)$ as well $(\dag\dag)$. Let
$\{\Phi_{\Sigma_{3}}^{*}H_{\lambda}:\lambda\in Par_{\Sigma}\}$
denote the lifted forms on $C_{1}^{ns}$ corresponding to this
presentation. By the fact that $\Phi$ is an isomorphism and
$(\dag)$, we have $Card(C^{ns}_{1}\cap
\Phi^{*}H_{\lambda'}\cap{\mathcal V}_{p_{j}})=n$. Now let
$V=V_{\Phi_{\Sigma_{1}}}\cap
V_{\Phi_{\Sigma_{2}}\circ\Phi_{\Sigma_{3}}}$. By $(\dag\dag)$ and
Lemma 2.13, we may witness $Card(C^{ns}_{1}\cap
\Phi^{*}H_{\lambda'}\cap{\mathcal V}_{p_{j}})=n$ inside the
canonical set $V$. Now, using the fact that
$\Phi_{2}\circ\Phi=\Phi_{1}$ as birational maps, and Lemma 2.31,
it follows that $Card(C_{1}^{ns}\cap G_{\lambda'}\cap {\mathcal
V}_{p_{j}})=n$.

\end{proof}

\begin{lemma}
The definition $(*)$ does not depend on the choice of a maximal
independent system $\Sigma$ containing $F_{\lambda}$, having
finite intersection with $C$.
\end{lemma}

\begin{proof}
Let $\Sigma$ be a maximal independent system containing
$F_{\lambda}$. We claim first that $\Sigma$ has no base points on
$C$, ($\dag$). For suppose that $w\in Base(\Sigma)\cap C$. Let
$F_{\mu}$ be \emph{any} form of degree $e$ having finite
intersection with $C$. Then $<F_{\mu},\Sigma>$ defines a system of
higher dimension, hence must intersect $H$ in a point. That is, we
can find parameters $\{\alpha,\beta\}$ and a fixed $F_{\lambda}$
in $\Sigma$ such that $\alpha F_{\mu}+\beta F_{\lambda}$ contains
$C$. It follows immediately that $w$ must also be a base point for
$F_{\mu}$. Clearly, we can find a form of degree $e$, having
finite intersection with $C$, which doesn't contain $w$. This
gives a contradiction and proves ($\dag$). Now, given a choice of
$\Sigma$ containing $F_{\lambda}$, by the proof of the previous
lemma, we may assume that;\\

$I_{italian}(p,\gamma_{p}^{j},C,F_{\lambda})=Card(C^{ns}\cap
\overline{F_{\lambda'}}\cap{\mathcal V}_{p_{j}})$,
$\lambda'\in{\mathcal
V}_{\lambda}$ generic in $Par_{\Sigma}$.\\

As $p$ is not a base point for $\Sigma$ on $C$, it follows
immediately that $p_{j}$ is not a base point for $\Sigma$ on
$C^{ns}$. Hence, by Lemma 2.12, we have that;\\

$I_{italian}(p,\gamma_{p}^{j},C,F_{\lambda})=I_{italian}(p_{j},C^{ns},\overline{F_{\lambda}})$\\

Clearly, this equality does not depend on the particular choice of
$\Sigma$ containing $F_{\lambda}$ but only on the presentation of
the morphism $\Phi:C^{ns}\rightarrow C$. The result follows.

\end{proof}

Following from the definition of the Italian intersection
multiplicity at a branch, we obtain a more refined version of
Bezout's theorem;\\

\begin{theorem}{Hyperspatial Bezout, Branched Version}\\

Let $C\subset P^{w}$ be a projective algebraic curve of degree $d$
and $F_{\lambda}$ a hypersurface of degree $e$, intersecting $C$
in finitely many points. Let $p$ be a point of intersection with
branches given by $\{\gamma_{p}^{1},\ldots,\gamma_{p}^{n}\}$.
Then;\\

$I_{italian}(p,C,F_{\lambda})=\sum_{1\leq j\leq
n}I_{italian}(p,\gamma_{p}^{j},C,F_{\lambda})$\\

In particular;\\

$\sum_{p\in C\cap F_{\lambda}}\sum_{1\leq j\leq
n_{p}}I_{italian}(p,\gamma^{j}_{p},C,F_{\lambda})=de$\\

\end{theorem}

\begin{proof}
Fix a non-singular model $(C^{ns},\Phi)$ of $C$. Let
$\Phi_{\Sigma'}$ be a particular presentation of the morphism
$\Phi$ such that $Base(\Phi_{\Sigma'})$ is disjoint from the fibre
$\Gamma_{\Phi}(x,p)=\{p_{1},\ldots,p_{n}\}$ $(*)$. Let $\Sigma$ be
a linear system containing $F_{\lambda}$ such that $\Sigma$
has finite intersection with $C$ and;\\

$I_{italian}(p,C,F_{\lambda})=I_{italian}^{\Sigma}(p,C,F_{\lambda})$\\

Let $V_{\Phi_{\Sigma'}}\subset C^{ns}$ and
$W_{\Phi_{\Sigma'}}\subset C$ be the canonical sets associated to
the morphism $\Phi_{\Sigma'}$. By Lemma 2.13, we can witness
$m=I_{italian}(p,C,F_{\lambda})$ by transverse intersections
$\{x_{1},\ldots,x_{m}\}=C\cap F_{\lambda'}\cap{\mathcal V}_{p}$
inside the canonical set $W_{\Phi_{\Sigma'}}$, for ${\lambda}'$
generic in $Par_{\Sigma}$. Again, by Lemma 2.13 and $(*)$, we can
find $\lambda'$ such that for each $p_{j}\in\Gamma_{[\Phi]}(x,p)$,
$m_{j}=I_{italian}(p,\gamma_{p}^{j},C,F_{\lambda})$ is witnessed
by transverse intersections
$\{y_{j}^{1},\ldots,y_{j}^{m_{j}}\}=C^{ns}\cap
\overline{F}_{\lambda'}\cap{\mathcal V}_{p_{j}}$ inside the
canonical set $V_{\Phi_{\Sigma'}}$. We need to show that
$m_{1}+\ldots+m_{j}+\ldots+m_{n}=m$. By properties of
infinitesimals and the fact that $\{p_{1},\ldots,p_{n}\}$ are
distinct, the sets
$\{y_{1}^{1},\ldots,y_{1}^{m_{1}}\},\ldots,\{y_{n}^{1},\ldots,y_{n}^{m_{n}}\}$
are disjoint. If $y_{i}^{j}$ belongs to one of these sets, then
$y_{i}^{j}\in{\mathcal V}_{p_{i}}$, hence
$\Phi_{\Sigma'}(y_{i}^{j})\in{\mathcal V}_{p}$. Moreover,
$\Phi_{\Sigma'}(y_{i}^{j})\in C\cap F_{\lambda'}$. It follows that
$y_{i}^{j}$ corresponds to a unique $x_{k}$ in
$\{x_{1},\ldots,x_{m}\}$. As each $y_{i}^{j}$ lies in
$V_{\Phi_{\Sigma'}}$, this clearly gives an injection from
$\{y_{1}^{1},\ldots,y_{1}^{m_{1}},\ldots,y_{n}^{1},\ldots,y_{n}^{m_{n}}\}$
to $\{x_{1},\ldots,x_{m}\}$. Hence, $m_{1}+\ldots+m_{n}\leq m$. To
prove equality, suppose that $x_{k}$ lies in
$\{x_{1},\ldots,x_{m}\}$. Then there exists a unique $y_{k}$ with
$\Phi_{\Sigma'}(y_{k})=x_{k}$. By a similar argument to the above,
$y_{k}$ must appear in one of the sets
$\{y_{1}^{1},\ldots,y_{1}^{m_{1}}\},\ldots,\{y_{n}^{1},\ldots,y_{n}^{m_{n}}\}$.
This gives the result.

\end{proof}
\begin{rmk}
One may use this  version of Bezout's theorem in order to develop
a more refined theory of $g_{n}^{r}$.
\end{rmk}

We also simplify the branch terminology for later
applications;\\

\begin{defn}
Let $C\subset P^{w}$ be a projective algebraic curve and
$C^{ns}\subset P^{w_{1}}$ some non-singular birational model with
birational morphism $\Phi^{ns}:C^{ns}\rightarrow C$. Let
$\gamma_{p}^{j}$ correspond to the infinitesimal neighborhood
${\mathcal V}_{p_{j}}$ of $p_{j}$ in $\Gamma_{\Phi^{ns}}(x,p)$.
Then we will also define the branch $\gamma_{p}^{j}$ by the formula;\\

$\gamma_{p}^{j}:=\{x\in C:\exists y(\Gamma_{\Phi^{ns}}(y,x)\wedge
y\in{\mathcal V}_{p_{j}})\}$\\

\end{defn}

\begin{rmk}
Note that the definition uses the language ${\mathcal L}_{spec}$,
and, in particular, is not algebraic.

\end{rmk}
We have the following;\\

\begin{lemma}
Definition 5.15 does not depend on the choice of non-singular
model $C^{ns}$ and morphism $\Phi^{ns}$. Lemma 5.7 may be
reformulated replacing the old definition of a branch with
Definition 5.15. Finally, we have, with hypotheses as for the old
definition of intersection multiplicity at a branch, that;\\

$I_{italian}(p,\gamma_{p}^{j},C,F_{\lambda})=Card(C\cap
F_{\lambda'}\cap \gamma_{p}^{j})$, $\lambda'\in{\mathcal
V}_{\lambda}$ generic in $Par_{\Sigma}$ $(*)$\\

\end{lemma}

\begin{proof}
The first part follows immediately from Lemma 5.1 and the fact
that all the data of the lemma may be taken inside a standard
model. The second part is similar, follow through the proof of
Lemma 5.7. The final part may be checked by following carefully
through the proofs of Lemmas 5.11 and 5.12.
\end{proof}

\begin{rmk}
Note that we could not have simplified the above presentation by
taking $(*)$ as our original definition of intersection
multiplicity at a branch. The main reason being that the arguments
on Zariski multiplicities require us to count intersections inside
$C\cap {\mathcal V}_{p}$, rather than the smaller ${\mathcal
L}_{spec}$ definable $C\cap\gamma_{p}^{j}$.

\end{rmk}

We now reformulate the preliminary definitions of Section 2 in
terms of branches.\\

Let $C\subset P^{w}$ be a projective algebraic curve and let
$\Sigma$ be a linear system, having finite intersection with $C$.
Let $g_{n'}^{r}(\Sigma)$ be defined by this linear system $\Sigma$
and let $g_{n}^{r}\subset g_{n'}^{r}$ be obtained by removing its
fixed point contribution. Fix a non-singular model $(C^{ns},\Phi)$
of $C$, with corresponding presentation $\Phi_{\Sigma'}$. By the
transfer result Lemma 2.31, we obtain a corresponding $g_{n}^{r}$
without fixed points on $C^{ns}$. Let
$Card(O,V_{\lambda},g_{n}^{r})$ denote the number of times $O\in
C^{ns}$ is counted for this $g_{n}^{r}$ in the weighted set
$V_{\lambda}$. We then make the following definition;\\

\begin{defn}
$I_{italian}^{\Sigma,mobile}(p,\gamma_{p}^{j},C,\phi_{\lambda})=Card(p_{j},V_{\lambda},g_{n}^{r})\
(*)$\\

where the branch $\gamma_{p}^{j}$ corresponds to $[{\mathcal
V}_{p_{j}}]$ in the fibre $\Gamma_{[\Phi_{\Sigma'}]}(x,p)$.

\end{defn}

As before, one needs the following lemma;\\

\begin{lemma}
The definition $(*)$ does not depend on the choice of non-singular
model $C^{ns}$ and birational morphism $\Phi_{\Sigma'}$.
\end{lemma}

\begin{proof}
The proof is the same as Lemma 5.11, we leave the details to the
reader.
\end{proof}

We now make the following definition;\\

\begin{defn}
Let hypotheses be as in Definition 5.19, then we define;\\

$I^{\Sigma}_{italian}(p,\gamma_{p}^{j},C,\phi_{\lambda})=I_{italian}^{\Sigma,mobile}(p,\gamma_{p}^{j},C,\phi_{\lambda})+1$
if $p\in Base(\Sigma)$\\

$I^{\Sigma}_{italian}(p,\gamma_{p}^{j},C,\phi_{\lambda})=I_{italian}^{\Sigma,mobile}(p,\gamma_{p}^{j},C,\phi_{\lambda})$
\ \ \ \ \ if $p\notin Base(\Sigma)$\\
\end{defn}

\begin{lemma}
Let notation be as in the previous definition. As in Remarks 5.10,
if $(C^{ns},\Phi)$ is a non-singular model of $C$, with
presentation $\Phi_{\Sigma'}$ such that $Base(\Sigma')$ is
disjoint from the fibre $\Gamma_{[\Phi_{\Sigma'}]}(x,p)$, then we
have that;\\

$I^{\Sigma}_{italian}(p,\gamma_{p}^{j},C,\phi_{\lambda})=Card(C^{ns}\cap\overline{\phi_{\lambda'}}\cap\mathcal
V_{p_{j}})$, $\lambda'\in{\mathcal V}_{\lambda}$ generic in
$Par_{\Sigma}.$\\

where again we have used the notation
$\{\overline{\phi_{\lambda}}:\lambda\in Par_{\Sigma}\}$ to denote
the family of "lifted" forms, as in Remarks 5.10.
\end{lemma}

\begin{proof}
We divide the proof into the following cases;\\

Case 1. $p\notin Base(\Sigma)$.\\

Then, by Definition 5.21, we have that;\\

$I^{\Sigma}_{italian}(p,C,\gamma_{p}^{j},\phi_{\lambda})=I_{italian}^{\Sigma,mobile}(p,C,\gamma_{p}^{j},\phi_{\lambda})$ $(1)$.\\

By the assumption on $\Sigma'$, we have that $p_{j}\notin
Base(\Sigma)$ for the "lifted" family of forms on $C^{ns}$,
corresponding to $\Sigma$, $(\dag)$. By the transfer result, Lemma
2.31, the $g_{n}^{r}$ on $C^{ns}$, used to define
$I_{italian}^{\Sigma,mobile}(p,C,\gamma_{p}^{j},\phi_{\lambda})$,
is obtained from the "lifted" family of form on $C^{ns}$ after
removing all fixed point contributions. Therefore, as by $(\dag)$,
this lifted family has no fixed point contribution
at $p_{j}$, we must have that;\\

$I_{italian}(p_{j},C^{ns},\overline{\phi_{\lambda}})=I_{italian}^{\Sigma,mobile}(p,C,\gamma_{p}^{j},\phi_{\lambda})$ $(2)$\\

Combining $(1)$,$(2)$ and using Lemma 2.10, we have that;\\

$I_{italian}^{\Sigma}(p_{j},C^{ns},\overline{\phi_{\lambda}})=I^{\Sigma}_{italian}(p,C,\gamma_{p}^{j},\phi_{\lambda})$\\

The result for this case now follows immediately from the
definition of
$I_{italian}^{\Sigma}(p_{j},C^{ns},\overline{\phi_{\lambda}})$.\\

Case 2. $p\in Base(\Sigma)$.\\

Then, by Definition 5.21, we have that;\\

$I^{\Sigma}_{italian}(p,C,\gamma_{p}^{j},\phi_{\lambda})=I_{italian}^{\Sigma,mobile}(p,C,\gamma_{p}^{j},\phi_{\lambda})+1$ $(1)$.\\

In this case, we have that $p_{j}\in Base(\Sigma)$ for the
"lifted" family of forms on $C^{ns}$, corresponding to $\Sigma$,
$(\dag)$. Let $I_{p_{j}}$ be the fixed point contribution for this
family, as defined in Lemma 2.14. Then, by a similar argument to the above, we have that;\\

$I_{italian}(p_{j},C^{ns},\overline{\phi_{\lambda}})=I_{p_{j}}+I_{italian}^{\Sigma,mobile}(p,C,\gamma_{p}^{j},\phi_{\lambda})$ $(2)$\\

Using Lemma 2.15, we have that;\\

$I_{italian}(p_{j},C^{ns},\overline{\phi_{\lambda}})=I_{p_{j}}+I_{italian}^{\Sigma}(p_{j},C^{ns},\overline{\phi_{\lambda}})-1$
$(3)$\\

Combining $(1),(2),(3)$ gives that;\\

$I_{italian}^{\Sigma}(p_{j},C^{ns},\overline{\phi_{\lambda}})=I^{\Sigma}_{italian}(p,C,\gamma_{p}^{j},\phi_{\lambda})$\\

Again, the result for this case follows immediately from the
definition of $I_{italian}^{\Sigma}(p_{j},C^{ns},\overline{\phi_{\lambda}})$.\\

\end{proof}

As an easy consequence of the previous lemma, we have that;\\

\begin{lemma}
Let $C\subset P^{w}$ be a projective algebraic curve. Let $\Sigma$
be a linear system, having finite intersection with $C$. Then, if
$\gamma_{p}^{j}$ is a branch centred at $p$ and $\phi_{\lambda}$ belongs to $\Sigma$;\\

$I_{italian}^{\Sigma}(p,C,\gamma_{p}^{j},\phi_{\lambda})=Card(C\cap\gamma_{p}^{j}\cap\phi_{\lambda'})$
for $\lambda'\in{\mathcal V}_{\lambda}$ generic in
$Par_{\Sigma}$.\\

$I_{italian}^{\Sigma,mobile}(p,C,\gamma_{p}^{j},\phi_{\lambda})=Card(C\cap(\gamma_{p}^{j}\setminus
p)\cap \phi_{\lambda'})$ for $\lambda'\in{\mathcal V}_{\lambda}$
generic\\
\indent \indent \indent\indent\indent\ \ \ \ \ \ \ \ \ \ \ \ \ \ \ \ \ \ \ \ \ \ \ \ \ \ \ \ \ \ \ \ \ \ \ \ \ \ \ \ \ \ \ \ \ \ \ \  in $Par_{\Sigma}$.\\

where $\gamma_{p}^{j}$ was given in Definition 5.15.

\end{lemma}

\begin{proof}
The first part of the lemma follows immediately from Lemma 5.22
and the Definition 5.15 of a branch. The second part follows from
Definition 5.21 and the first part.
\end{proof}

We then reformulate the remaining results of Section 2 in terms of
branches. The notation of Lemma 5.23 will be use for the remainder of this section.\\

\begin{lemma}{Non-Existence of Coincident Mobile Points along a
Branch}\\

Let $C\subset P^{w}$ be a projective algebraic curve. Let $\Sigma$
be a linear system, having finite intersection with $C$, such that
$p\in {C\setminus Base(\Sigma)}$. Then, if $\gamma_{p}^{j}$ is a
branch centred at $p$ and $\phi_{\lambda}$ belongs to $\Sigma$;\\

$I_{italian}(p,C,\gamma_{p}^{j},\phi_{\lambda})=I_{italian}^{\Sigma}(p,C,\gamma_{p}^{j},\phi_{\lambda})=I_{italian}^{\Sigma,mobile}(p,C,\gamma_{p}^{j},\phi_{\lambda})$\\

\end{lemma}

\begin{proof}
Let $(C^{ns},\Phi)$ be a non-singular model of $C$, with
presentation $\Phi_{\Sigma'}$, such that $Base(\Sigma')$ is
disjoint from $\Gamma_{[\Phi]}(x,p)$. Let
$\{\overline{\phi_{\lambda}}\}$ be the "lifted" family of
algebraic forms on $C^{ns}$ defined by $\Sigma$. By Lemma 5.22, we
have that;\\

$I_{italian}^{\Sigma}(p,C,\gamma_{p}^{j},\phi_{\lambda})=I_{italian}^{\Sigma}(p_{j},C^{ns},\overline{\phi_{\lambda}})$ $(1)$\\

By Remarks 5.10, we have that;\\

$I_{italian}(p,C,\gamma_{p}^{j},\phi_{\lambda})=I_{italian}(p_{j},C^{ns},\overline{\phi_{\lambda}})$ $(2)$\\

Using the fact that $p\notin Base(\Sigma)$ and the hypotheses on
$\Phi_{\Sigma'}$, it follows that $p_{j}\notin Base(\Sigma)$, for
the lifted system defined by $\Sigma$. Hence, we may apply Lemma
2.10, to obtain that;\\

$I_{italian}^{\Sigma}(p_{j},C^{ns},\overline{\phi_{\lambda}})=I_{italian}(p_{j},C^{ns},\overline{\phi_{\lambda}})$ $(3)$\\

The result follows by combining $(1),(2)$ and $(3)$ and using
Definition 5.21.

\end{proof}

\begin{lemma}{Branch Multiplicity at non-base points witnessed by
transverse intersections along the branch}\\

Let $p\in {C\setminus Base(\Sigma)}$ and let $\gamma_{p}^{j}$ be a
branch centred at $p$. Then, if
$m=I_{italian}(p,C,\gamma_{p}^{j},\phi_{\lambda})$, we can find
$\lambda'\in{\mathcal V}_{\lambda}$, generic in $Par_{\Sigma}$,
and distinct $\{p_{1},\ldots,p_{m}\}=C\cap\phi_{\lambda'}\cap(
{\gamma_{p}^{j}\setminus p})$ such that the intersection of $C$
with $\phi_{\lambda'}$ at each $p_{i}$ is transverse for $1\leq
i\leq m$.
\end{lemma}

\begin{proof}
By Lemmas 5.23 and 5.24, for $\lambda'\in{\mathcal V}_{\lambda}$,
generic in $Par_{\Sigma}$, the intersection
$C\cap\phi_{\lambda'}\cap \gamma_{p}^{j}$ consists of $m$ distinct
points $\{p_{1},\ldots,p_{m}\}$. The condition on $Par_{\Sigma}$
that $\phi_{\lambda}$ passes through $p$ defines a proper closed
subset, hence we may assume these points are all distinct from
$p$. Finally, the transversality result follows from, say Lemma
2.17, using the fact that $\{p_{1},\ldots,p_{m}\}$ cannot lie
inside $Base(\Sigma)$.

\end{proof}

Again, we have analogous results to Lemmas 5.24 and 5.25 for
points in $Base(\Sigma)$;\\

We first require the following;\\

\begin{lemma}
Let $p\in C\cap Base(\Sigma)$ and $\gamma_{p}^{j}$ a branch
centred at $p$. Then there exists an open subset
$U_{\gamma_{p}^{j}}\subset Par_{\Sigma}$ and an integer
$I_{\gamma_{p}^{j}}\geq 1$ such that;\\

$I_{italian}(p,C,\gamma_{p}^{j},\phi_{\lambda})=I_{\gamma_{p}^{j}}$
for
$\lambda\in U_{\gamma_{p}^{j}}$\\

and\\

$I_{italian}(p,C,\gamma_{p}^{j},\phi_{\lambda})\geq
I_{\gamma_{p}^{j}}$ for
$\lambda\in Par_{\Sigma}$\\

\end{lemma}

\begin{proof}
Let $(C^{ns},\Phi)$ be a non-singular model with presentation
$\Phi_{\Sigma'}$ such that $Base(\Sigma')$ is disjoint from
$\Gamma_{[\Phi]}(x,p)$. Then, by the proof of Lemma 5.12, we have
that, for $\lambda\in Par_{\Sigma}$;\\

$I_{italian}(p,\gamma_{p}^{j},C,\phi_{\lambda})=I_{italian}(p_{j},C^{ns},\overline{\phi_{\lambda}})$\\

By properties of Zariski structures;\\

$W_{k}=\{\lambda\in
Par_{\Sigma}:I_{italian}(p_{j},C^{ns},\overline{\phi_{\lambda}})\geq
k\}$\\

are definable and Zariski closed in $Par_{\Sigma}$. The result
then follows by taking $I_{\gamma_{p_{j}}}=min_{\lambda\in
Par_{\Sigma}}I_{italian}(p_{j},C^{ns},\overline{\phi_{\lambda}})$
and using the fact that $Par_{\Sigma}$ is irreducible.

\end{proof}

We can now formulate analogous results to Lemmas 5.24 and 5.25 for
base points;\\

\begin{lemma}
Let $p\in C\cap Base(\Sigma)\cap\phi_{\lambda}$ and
$\gamma_{p}^{j}$ a branch centred at $p$.
Then;\\

$I_{italian}(p,\gamma_{p}^{j},C,\phi_{\lambda})=I_{\gamma_{p}^{j}}+I_{italian}^{\Sigma}(p,\gamma_{p}^{j},C,\phi_{\lambda})-1$\\

and\\

$I_{italian}(p,\gamma_{p}^{j},C,\phi_{\lambda})=I_{\gamma_{p}^{j}}+I_{italian}^{\Sigma,mobile}(p,\gamma_{p}^{j},C,\phi_{\lambda})$\\

\end{lemma}

\begin{proof}
In order to prove the first part of the lemma, suppose that
$m=I_{italian}^{\Sigma}(p,\gamma_{p}^{j},C,\phi_{\lambda})$. By
Lemma 5.23, choosing $\lambda'\in{\mathcal V}_{\lambda}$ generic
in $Par_{\Sigma}$, we can find
$\{p_{1},\ldots,p_{m-1}\}=C\cap\phi_{\lambda'}\cap({{\mathcal
V}_{p}\setminus p})$, distinct from $p$, witnessing this
multiplicity. Using the fact that $\{p_{1},\ldots,p_{m-1}\}$ lie
outside $Base(\Sigma)$, we may apply Lemma 2.17 to obtain that the
intersections at these points are transverse. By the previous
Lemma 5.26, we have that
$I_{italian}(p,\gamma_{p}^{j},C,\phi_{\lambda'})=I_{\gamma_{p}^{j}}$.
Now choose a nonsingular model $(C^{ns},\Phi)$, with presentation
$\Phi_{\Sigma'}$, such that $Base(\Sigma')$ is disjoint from
$\Gamma_{[\Phi]}(x,p)$. By definition 5.15 of a branch, we can
find
$\{p_{1}',\ldots,p_{m}'\}=C^{ns}\cap\overline{\phi_{\lambda'}}\cap({\mathcal
V_{p}\setminus p})$. By properties of specialisations,
$Base(\Sigma')$ is also disjoint from this set. We then have that
the intersections between $C^{ns}$ and
$\overline{\phi_{\lambda'}}$ are also transverse at these points
and that
$I_{italian}(p_{j},C^{ns},\overline{\phi_{\lambda'}})=I_{\gamma_{p}^{j}}$.
It then follows by summability of specialisation, see
\cite{depiro2}, that;\\

$I_{italian}(p_{j},C^{ns},\overline{\phi_{\lambda}})=I_{\gamma_{p}^{j}}+(m-1)$.\\

Again, using the presentation of $(C^{ns},\Phi)$, we obtain
that;\\

$I_{italian}(p,\gamma_{p}^{j},C,\phi_{\lambda})=I_{\gamma_{p}^{j}}+(m-1)$.\\

Hence, the result follows. The second part of the lemma follows
immediately from the Definition 5.21 of
$I_{italian}^{\Sigma}(p,C,\gamma_{p}^{j},\phi_{\lambda})$ at a
base point.
\end{proof}

\begin{lemma}
Let $p\in C\cap Base(\Sigma)$ and let $\gamma_{p}^{j}$ be a branch
centred at $p$. Then, if
$m=I_{italian}^{\Sigma}(p,C,\gamma_{p}^{j},\phi_{\lambda})$, we
can find $\lambda'\in {\mathcal V}_{\lambda}$, generic in
$Par_{\Sigma}$, and distinct
$\{p_{1},\ldots,p_{m-1}\}=C\cap\phi_{\lambda'}\cap({\gamma_{p}^{j}\setminus
p})$ such that the intersection of $C$ with $\phi_{\lambda'}$ at
each $p_{i}$ is transverse for $1\leq i\leq m-1$. If
$m=I_{italian}^{\Sigma,mobile}(p,C,\gamma_{p}^{j},\phi_{\lambda})$,
then the same results for distinct $\{p_{1},\ldots,p_{m}\}$ with
the same properties.
\end{lemma}

\begin{proof}
The first part of the lemma is a straightforward consequence of
Lemma 5.23, properties of infinitesimals, (to show that
$\{p_{1},\ldots,p_{m-1}\}$ lie outside $Base(\Sigma)$) and Lemma
2.17 (to obtain transversality). The second part of the lemma also
follows from Lemma 5.23 and Lemma 2.17 (to obtain transversality).
\end{proof}

We now note the following, connecting the original definition of
$I_{italian}$ with its branched version at non-singular points.

\begin{lemma}
Let $p$ be a nonsingular point of the curve $C$, then there exists
a unique branch $\gamma_{p}$ centred at $p$ and;\\

$I_{italian}(p,C,\phi_{\lambda})=I_{italian}(p,C,\gamma_{p},\phi_{\lambda})$ $(1)$\\

$I_{italian}^{\Sigma,mobile}(p,C,\phi_{\lambda})=I_{italian}^{\Sigma,mobile}(p,C,\gamma_{p},\phi_{\lambda})$ $(2)$\\

\end{lemma}

\begin{proof}
The fact that there exists a unique branch $\gamma_{p}$, centred
at $p$, follows from Lemma 5.4. By definition 5.15 of a branch, we
then have that;\\

$\gamma_{p}=C\cap {\mathcal V}_{p}$\\

Now $(1)$ follows from Lemma 5.17, Lemma 2.10 and Definition 2.6.
While $(2)$ follows from Definition 2.20 and Lemma 5.23.

\end{proof}

\end{section}
\begin{section}{Cayley's Classification of Singularities}

The purpose of this section is to develop a theory of
singularities for algebraic curves based on the work of Cayley. In
order to make this theory rigorous, one first needs to find a
method of parametrising the branches of an algebraic curve. This
is the content of the following theorem;\\

\begin{theorem}{Analytic Representation of a Branch}\\

Let $C\subset P^{w}$ be a projective algebraic curve. Suppose that
$C$ is defined by equations $\{F_{1}(x_{1},\ldots,x_{w}),\ldots
F_{m}(x_{1},\ldots,x_{w})\}$ in affine coordinates
$x_{i}={X_{i}\over X_{0}}$. Let $p\in C$ correspond to the point
$\bar 0$ in this coordinate system. Then there exist algebraic
power series $\{x_{1}(t),\ldots,x_{w}(t)\}$ such that
$x_{1}(t)=\ldots=x_{w}(t)=0$, $F_{j}(x_{1}(t),\ldots,x_{w}(t))=0$
for $1\leq j\leq m$ and with the property that, for any
algebraic function $F_{\lambda}(x_{1},\ldots,x_{w})$;\\

$F_{\lambda}(x_{1}(t),\ldots,x_{w}(t))\equiv 0$ iff $F_{\lambda}$
vanishes on $C$.\\

Otherwise, $F_{\lambda}$ has finite intersection with $C$ and\\

$ord_{t}F_{\lambda}(x_{1}(t),\ldots,x_{w}(t))=I_{italian}(p,\gamma_{p}^{j},C,F_{\lambda})$ $(*)$\\

We refer to the power series as parametrising the branch
$\gamma_{p}^{j}$.

\end{theorem}

\begin{proof}
We first prove the theorem in the case when $C\subset P^{w}$ is a
non-singular projective algebraic curve. By Lemma 4.13, we can
find a plane projective algebraic curve $C_{1}\subset P^{2}$ such
that $\{C,C_{1}\}$ are birational and there exists a corresponding
point $p_{1}\in C_{1}$ such that $p_{1}$ is non-singular. Let
$\Phi_{\Sigma}$ and $\Psi_{\Sigma'}$ be presentations such that
$\Psi_{\Sigma'}=\Phi_{\Sigma}^{-1}$ as a birational map. Without
loss of generality, we may assume that
$V_{\Phi_{\Sigma}}=W_{\Psi_{\Sigma'}}\subset C$ and
$V_{\Psi_{\Sigma'}}=W_{\Phi_{\Sigma}}\subset C_{1}$. Moreover, we
may assume that $\{p,p_{1}\}$ lie in
$\{V_{\Phi_{\Sigma}},V_{\Psi_{\Sigma'}}\}$ and correspond to the
origins of the affine coordinate systems $(x_{1},\ldots,x_{w})$
and $(y_{1},y_{2})$. Let
$\Sigma'=\{\psi_{0},\psi_{1},\ldots,\psi_{w}\}$ and let
$(y_{1}(t),y_{2}(t))$ be an analytic representation of $p_{1}\in
C_{1}$, given by the inverse function theorem. We obtain an
analytic representation of $p\in C$ by the
formula;\\

$(x_{1}(t),\ldots,x_{w}(t))=({\psi_{1}\over\psi_{0}}(y_{1}(t),y_{2}(t)),\ldots,{\psi_{w}\over\psi_{0}}(y_{1}(t),y_{2}(t)))$\\

First, note that as $p\notin Zero(\psi_{0})$, we may assume that
$\psi_{0}(0,0)\neq 0$. Hence, we can invert the power series
$\psi_{0}(y_{1}(t),y_{2}(t))$. This clearly proves that $x_{j}(t)$
is a formal power series in $L[[t]]$. That $x_{j}(0)=0$ for $1\leq
j\leq w$ follows from the corresponding property for
$(y_{1}(t),y_{2}(t))$ and the fact that $p$ is situated at the
origin of the coordinate system $(x_{1},\ldots,x_{w})$. Finally,
we need to check that $x_{j}(t)$ define algebraic power series.
This follows obviously from the fact that $\psi_{j}$ and
$\psi_{0}$ define algebraic functions. Now, suppose that
$\{F_{1},\ldots,F_{m}\}$ are defining equations for $C$. Let
$\{F_{1}',\ldots,F_{m}'\}$ be the corresponding equations written
in homogeneous form for the variables $\{X_{0},\ldots,X_{w}\}$,
where $x_{j}={X_{j}\over X_{0}}$. Let $G(Y_{0},Y_{1},Y_{2})$ be
the defining equation for $C_{1}$. We can homogenise the power
series representation of $p_{1}\in C_{1}$ by
$(Y_{0}(t):Y_{1}(t):Y_{2}(t))=(1:y_{1}(t):y_{2}(t))$. Then we must
have that $G(1:y_{1}(t):y_{2}(t))=0$. Now
$F_{k}'(\psi_{0},\ldots,\psi_{w})$ vanishes identically on
$C_{1}$, hence, by the projective Nullstellensatz, there exists a
homogeneous $H_{k}(Y_{0},Y_{1},Y_{2})$ such that;\\

$F_{k}'(\psi_{0},\ldots,\psi_{w})=H_{k}G$\\

It follows that;\\

$F_{k}'(\psi_{0}(1:y_{1}(t):y_{2}(t)),\ldots,\psi_{w}(1:y_{1}(t):y_{2}(t)))\equiv
0$\\

therefore;\\

$F_{k}'(1:{\psi_{1}(y_{1}(t),y_{2}(t))\over\psi_{0}(y_{1}(t),y_{2}(t))}:\ldots:{\psi_{w}(y_{1}(t),y_{2}(t))\over\psi_{0}(y_{1}(t),y_{2}(t))})\equiv
0$\\

which gives;\\

$F_{k}(x_{1}(t),\ldots,x_{w}(t))\equiv 0$\\

as required. The property that an algebraic function $F_{\lambda}$
vanishes on $(x_{1}(t),\ldots,x_{w}(t))$ iff it vanishes on $C$
can be proved in a similar way to the above argument, invoking
Theorem 2.10 of the paper \cite{depiro4}. Alternatively, it can be
proved directly, using the fact that, as
$(x_{1}(t),\ldots,x_{w}(t))$ define algebraic power series,
$(y_{1}-x_{1}(t),\ldots,y_{w}-x_{w}(t))$ defines the equation of a
curve $C'$ on some etale cover $i:(A^{w}_{et},(\bar
0)^{lift})\rightarrow (A^{w},(\bar 0))$ such that $i(C')\subset
C$. If $F_{\lambda}$ vanishes on $C'$, then it must vanish on an
open subset $U$ of $C$, hence as $F_{\lambda}$ is closed, must
vanish on all of $C$ as required.  Finally, we need to check the
property $(*)$. Suppose that $F_{\lambda}$ is an algebraic
function with $m=ord_{t}F_{\lambda}(x_{1}(t),\ldots,x_{w}(t))$,
passing through $p$. Choose $\Sigma_{1}$ containing $F_{\lambda}$
such that $\Sigma_{1}$ has finite intersection with $C$ and
$p\notin Base(\Sigma_{1})$. It follows, using Lemma 2.12, that
$I_{italian}(p,C,F_{\lambda})=I_{italian}^{\Sigma_{1}}(p,C,F_{\lambda})$,
$(\dag)$. Using Lemma 2.31, we can transfer the system
$\Sigma_{1}$ to a system on $C_{1}$. Let $G_{\lambda}$ be the
corresponding algebraic curve to the algebraic form $F_{\lambda}$.
We must have that $p_{1}\notin Base(\Sigma_{1})$, otherwise, as
$p_{1}$ belongs to the canonical set $V_{\Psi_{\Sigma'}}$, we
would have that $p$ belongs to $Base(\Sigma_{1})$ as well. Hence,
using Lemma 2.12 again, we must have that
$I_{italian}(p_{1},C_{1},G_{\lambda})=I_{italian}^{\Sigma_{1}}(p_{1},C_{1},G_{\lambda})$,
$(\dag\dag)$. By direct calculation, we have that;\\

$G_{\lambda}(y_{1}(t),y_{2}(t))=\psi_{0}^{r}(y_{1}(t),y_{2}(t))F_{\lambda}(x_{1}(t),\ldots,x_{w}(t))$\\

for some $r\leq 0$. As $ord_{t}\psi_{0}(y_{1}(t),y_{2}(t))=0$, we
have that;\\

$ord_{t}G_{\lambda}(y_{1}(t),y_{2}(t))=ord_{t}F_{\lambda}(x_{1}(t),\ldots,x_{w}(t))=m$\\

Now, by Theorem 5.1 of the paper \cite{depiro2} and $(\dag\dag)$,
it follows
that;\\

$I_{italian}(p_{1},C_{1},G_{\lambda})=I_{italian}^{\Sigma_{1}}(p_{1},C_{1},G_{\lambda})=m$\\

Now, using Lemma 2.31 and the fact that $\{p,p_{1}\}$ lie in the
canonical sets $V_{\Phi_{\Sigma}}$ and $W_{\Phi_{\Sigma}}$, we
must have that $I_{italian}^{\Sigma_{1}}(p,C,F_{\lambda})=m$ as
well. Hence, by $(\dag)$, it follows that
$I_{italian}(p,C,F_{\lambda})=m$. As $C$ is a non-singular model
of itself, this proves the claim $(*)$ in this special case.\\

We now assume that $C\subset P^{w}$ is \emph{any} projective
algebraic curve. Suppose that $C^{ns}\subset P^{w'}$ is a
non-singular model of $C$ with birational morphism
$\Phi_{\Sigma'}:C^{ns}\rightarrow C$ such that the branch
$\gamma_{p}^{j}$ corresponds to ${\mathcal V}_{p_{j}}$ in the
fibre $\Gamma_{[\Phi]}(x,p)$, disjoint from $Base(\Sigma')$. As
before, we may assume that $\{p,p_{j}\}$ correspond to the origins
of the coordinate systems $(x_{1},\ldots,x_{w})$ and
$(y_{1},\ldots,y_{w'})$. Let
$\Sigma'=\{\phi_{0},\ldots,\phi_{w}\}$. By the previous argument,
we can find an analytic representation
$(y_{1}(t),\ldots,y_{w'}(t))$ of $p_{j}$ in $C^{ns}$, with the
properties given in the statement of the theorem. As before, we
obtain an analytic representation of the corresponding $p\in C$,
by the formula;\\

$(x_{1}(t),\ldots,x_{w}(t))=({\phi_{1}\over\phi_{0}}(y_{1}(t),\ldots,y_{w'}(t)),\ldots,{\phi_{w}\over\phi_{0}}(y_{1}(t),\ldots,y_{w'}(t)))$\\

One checks that this has the required properties up to the
property $(*)$ by a direct imitation of the proof above, with the
minor modification that the projective Nullstellensatz for
$C^{ns}$ gives that, if $\{G_{1},\ldots,G_{k}\}$ are defining
equations for $C^{ns}$, then, if $F$ vanishes on $C^{ns}$, there
must exist homogeneous $\{H_{1},\ldots,H_{k}\}$ such that
$F=H_{1}G_{1}+\ldots+H_{k}G_{k}$. Alternatively, one can refer the
parametrisation to a non-singular point of a plane projective
curve, in which case the argument up to $(*)$ is identical.\\

We now verify the property $(*)$. Suppose that $F_{\lambda}$ is an
algebraic function with
$ord_{t}F_{\lambda}(x_{1}(t),\ldots,x_{w}(t))=m$. Let
${\overline{F}}_{\lambda}$ be the corresponding function on
$C^{ns}$, obtained from the presentation $\Phi_{\Sigma'}$.  By
Lemmas 5.11 and 5.12;\\

$I_{italian}(p,\gamma_{p}^{j},C,F_{\lambda})=I_{italian}(p_{j},C^{ns},\overline{F_{\lambda}})
\ \ (**)$\\

We claim that $ord_{t}{\overline
F}_{\lambda}(y_{1}(t),\ldots,y_{w'}(t))=m$. This follows by
repeating the argument given above. By the properties of
$(y_{1}(t),\ldots,y_{w'}(t))$ and the result verified in the case
of a non-singular curve, we obtain immediately that
$I_{italian}(p_{j},C^{ns},\overline{F_{\lambda}})=m$. Combined
with $(**)$, this gives the required result.\\

\end{proof}

Using the analytic representation, we obtain the following
classification of singularities due to Cayley;\\

\begin{theorem}{Cayley's Classification of Singularities}\\

Let $C\subset P^{w}$ be a projective algebraic curve which is not
contained in any hyperplane section. Let $\gamma_{p}^{j}$ be a
branch of the algebraic curve centred at $p$. Then we can assign a
sequence of non-negative integers
$(\alpha_{0},\alpha_{1},\ldots,\alpha_{w-1})$, called the
\emph{character} of the branch, which has the following
property;\\

Let $\Sigma$ be the system of hyperplanes passing through $p$.
Then there exists a filtration of $\Sigma$ into subsystems of hyperplanes;\\

$\Sigma_{w-1}\subset\Sigma_{w-2}\subset\ldots\subset\Sigma_{1}\subset\Sigma_{0}=\Sigma$\\

with $dim(\Sigma_{i})=(w-1)-i$, such that, for any hyperplane
$H_{\lambda}$ passing through $p$, we have that;\\

$H_{\lambda}\in\Sigma_{i}$ iff
$I_{italian}(p,\gamma_{p}^{j},C,H_{\lambda})\geq\alpha_{0}+\alpha_{1}+\ldots+\alpha_{i}$. $(\dag)$\\

Moreover, these are the only multiplicities which occur.
\end{theorem}

\begin{proof}
Without loss of generality, we may assume that $p$ is situated at
the origin of the coordinate system $(x_{1},\ldots,x_{w})$. By the
previous theorem, we can find an analytic parameterisation of the
branch $\gamma_{p}^{j}$ of the form;\\

$x_{k}(t)=a_{k,1}t+a_{k,2}t^{2}+\ldots+a_{k,n}t^{n}+\ldots$, for $1\leq k\leq w$\\

Let $\sum_{k=1}^{w}\lambda_{k}x_{k}=0$ be the equation of a
hyperplane $H_{\lambda}$ passing through $p$. Then, we can write
$H_{\lambda}(x_{1}(t),\ldots,x_{k}(t))$ in the
form;\\

$(\sum_{k=1}^{w}\lambda_{k}a_{k,1})t+\ldots+(\sum_{k=1}^{w}\lambda_{k}a_{k,n})t^{n}+O(t^{n})$ (*)\\

Let $\bar a_{n}=(a_{1,n},\ldots,a_{k,n},\ldots,a_{w,n})$. We claim
that we can find a sequence $(\bar a_{m_{1}},\ldots,\bar
a_{m_{w}})$, for $m_{1}<\ldots<m_{i}<\ldots<m_{w}$, such that;\\

(i). $\{\bar a_{m_{1}},\ldots,\bar a_{m_{w}}\}$ is linearly
independent\\

(ii). $V_{m_{i}}=<\bar a_{1},\ldots,\bar a_{m_{i}}>=<\bar a_{m_{1}},\ldots,\bar a_{m_{i}}>$, for $1\leq i\leq w$\\

(iii). If there exist $\{n_{1},\ldots,n_{i}\}$ for
$n_{1}<\ldots< n_{i}$ with $n_{1}\leq
m_{1},\ldots, n_{i}\leq m_{i}$ such that\\

 $V_{m_{i}}=<\bar
a_{n_{1}},\ldots,\bar a_{n_{i}}>$,\\

 then
$n_{1}=m_{1},\ldots,n_{i}=m_{i}$.\\

(iv). $V_{k}=V_{m_{i}}$ for $m_{i}\leq k<m_{i+1}$.\\

The first three properties may be proved by induction on $i$.For
$i=1$, choose the first non-zero vector $\bar a_{m_{1}}$. For the
inductive step, assume we have found $\{\bar a_{m_{1}},\ldots,\bar
a_{m_{i}}\}$ with the required properties $(i)-(iii)$. We claim
that there exists $\bar a_{m_{i+1}}$, with $m_{i+1}>m_{i}$, such
that $\bar a_{m_{i+1}}\notin V_{m_{i}}$, $(**)$.
Suppose not. As $i<w$, the condition;\\

$\bigwedge_{s=1}^{i}(\sum_{k=1}^{w}\lambda_{k}a_{k,m_{s}}=0)$\\

defines a non-empty plane $P$ in the dimension $w-1$ parameter
space $Par_{H}$ of hyperplanes passing through $p$. Choosing
$\lambda\in P$ and using $(*)$, it follows that
$ord_{t}H_{\lambda}(x_{1}(t),\ldots,x_{w}(t))\geq n$ for all
$n>0$. Therefore, $H_{\lambda}(x_{1}(t),\ldots,x_{w}(t))\equiv 0$
and, by Theorem 6.1, $H_{\lambda}$ must contain $C$. This
contradicts the assumption that $C$ is \emph{not} contained in any
hyperplane section. Using $(**)$, choose $m_{i+1}$ minimal such
that $\bar a_{m_{i+1}}\notin V_{m_{i}}$. Properties $(i)$ and
$(ii)$ are trivial to verify. For $(iii)$, assume that $\{\bar
a_{n_{1}},\ldots,\bar a_{n_{i+1}}\}$ are as given in the
hypotheses. Then, the sequence must form a linearly independent
set. Hence, we must have that $V_{m_{i}}=<\bar
a_{n_{1}},\ldots,\bar a_{n_{i}}>$. By the induction hypothesis,
$n_{1}=m_{1},\ldots,n_{i}=m_{i}$. Then, $\bar a_{n_{i+1}}\notin
V_{m_{i}}$, Now, by minimality of $m_{i+1}$, we also have that
$\bar a_{m_{i+1}}=\bar a_{n_{i+1}}$ as required.\\

Property $(iv)$ follows easily from properties $(i)-(iii)$. We
clearly have that $V_{m_{i}}\subseteq V_{k}\subseteq V_{m_{i+1}}$,
for $m_{i}\leq k<m_{i+1}$. If $V_{k}\neq V_{m_{i}}$, then, by
$(i),(ii)$, $V_{k}=V_{m_{i+1}}$. This clearly contradicts $(iii)$.
\\

Now, define;\\

$\Sigma_{i}=\{\lambda\in Par_{H,p}:\bigwedge_{s=1}^{i}(\sum_{k=1}^{w}\lambda_{k}a_{k,m_{s}}=0)\}$, for, $1\leq i\leq w-1$,\\

Then, we obtain a filtration;\\

$\Sigma_{w-1}\subset\ldots\subset\Sigma_{i}\subset\ldots\subset\Sigma_{0}=\Sigma$\\

with $dim(\Sigma_{i})=(w-1)-i$ as in the statement of the theorem.
Define;\\

 $\alpha_{0}=m_{1}$ and $\alpha_{i}=m_{i+1}-m_{i}$ for
$1\leq i\leq w-1$.\\

We need to verify the property $(\dag)$. Suppose that
$H_{\lambda}\in\Sigma_{i}$, for $i\geq 1$. Then $H_{\lambda}$
contains the plane $V_{m_{i}}$ spanned by $\{\bar
a_{m_{1}},\ldots,\bar a_{m_{i}}\}$. Hence, by $(iv)$, it contains
the plane $V_{k}$ for $k<m_{i+1}$. Then, by $(*)$,
$ord_{t}(H_{\lambda}(x_{1}(t),\ldots,x_{w}(t)))\geq m_{i+1}$ and
,by
Theorem 6.1,\\

 $I_{italian}(p,\gamma_{p}^{j},C,H_{\lambda})\geq
m_{i+1}=\alpha_{0}+\ldots+\alpha_{i}$.\\

 Conversely, suppose that\\

$ord_{t}H_{\lambda}(x_{1}(t),\ldots,x_{w}(t))\geq\alpha_{0}+\ldots\alpha_{i}=m_{i+1}$,
for some $i\geq 1$.\\

 Then $H_{\lambda}$ contains the plane
$V_{k}$ for $k<m_{i+1}$. In particular, it contains the plane
$V_{m_{i}}$. Hence $H_{\lambda}\in\Sigma_{i}$. The remaining case
amounts to showing that
$I_{italian}(p,\gamma_{p}^{j},C,H_{\lambda})\geq\alpha_{0}$ for
any hyperplane $H_{\lambda}$ passing through $p$. This follows
immediately from $(*)$, Theorem 6.1 and the fact that $\bar
a_{m_{1}}$ was the first non-zero vector in the sequence $\{\bar
a_{n}:n<\omega\}$. The remark made after the property $(\dag)$
follows immediately from the property $(iv)$ of the sequence
$\{\bar a_{m_{1}},\ldots,\bar a_{m_{w}}\}$.

\end{proof}

\begin{defn}
In accordance with the Italian terminology, we refer to
$\alpha_{0}$ as the order of the branch, $\alpha_{j}$ as the
$j$'th range of the branch, for $1\leq j\leq w-2$, and
$\alpha_{w-1}$ as the final range or class of the branch. We
define $<\bar a_{m_{1}},\ldots,\bar a_{m_{k}}>$ to be the $k$'th
osculatory plane at $p$ for $1\leq k\leq w-1$. We also define the
$w-1$'th osculatory plane to be the osculatory plane. We define
the tuple $(\alpha_{0},\ldots,\alpha_{w-1})$ to be the character
of the branch. Cayley referred to branches of order $1$ as linear
and superlinear otherwise. He referred to branches having a
character of the form $(1,1\ldots,1)$ as ordinary. The Italian
geometers refer to a simple point, which is the origin of an
ordinary branch, as an ordinary simple point. Note that for a
simple (equivalently non-singular) point, the $1$'st osculatory
plane is the same as the tangent line. We will now also use the
terminology, tangent line to a branch, to describe the $1$'st
osculatory plane at any branch. We define a node of a plane curve
to be the origin of at most $2$ ordinary branches with distinct
tangent lines, this definition was used in Theorem 4.16. We also
used, in Theorem 4.16, the fact that a generic point of an
algebraic curve is an ordinary simple point, (this is not true
when the field has non-zero characteristic, see the final section)
a rigorous proof of this result requires duality arguments, we
postpone this proof for another occasion.
\end{defn}

We have the following important results on the projection of a
branch, see Section 4 for the relevant definitions.\\

\begin{theorem}
Let $C\subset P^{w}$ be a projective algebraic curve, as defined
in the previous theorems of this section, and $\gamma_{O}$ a
branch centred at $O$ with character
$(\alpha_{0},\ldots,\alpha_{w-1})$. Let $P$ be chosen generically
in $P^{w}$, then the projection $pr_{P}(\gamma_{O})$ has
character;\\

 $(\alpha_{0},\ldots,\alpha_{w-2})$\\

  If $P$ is situated generically on the osculatory plane, then
 $pr_{P}(\gamma_{O})$ has character;\\

$(\alpha_{0},\ldots,\alpha_{w-2}+\alpha_{w-1})$\\

 More generally, if $P$ is
 situated on the $k$'th osculatory plane for $1\leq k\leq w-2$,
 and not on an osculatory plane of lower order, then the
 projection $pr_{P}(\gamma_{O})$ has character;\\

 $(\alpha_{0},\ldots,\alpha_{k-2},\alpha_{k-1}+\alpha_{k},\alpha_{k+1},\ldots,\alpha_{w-1})$.\\

\end{theorem}

\begin{proof}
First note that, by Lemma 4.9, if $P$ is situated generically on
the $k$'th-osculatory plane for any $k\geq 1$, and $H'$ is any
hyperplane not containing $P$, the projection $pr_{P}$ is
generally biunivocal. Hence, by Lemma 5.7, the projection
$pr_{P}(\gamma_{O})$ is well defined. We first claim that for any
hyperplane $H_{\lambda}$ in $P^{w-1}$;\\

$I_{italian}(O,\gamma_{O},C,pr^{-1}_{P}(H_{\lambda}))=I_{italian}(pr_{P}(O),pr_{P}(\gamma_{O}),pr_{P}(C),H_{\lambda})$ $(*)$\\

This follows by using the proof of Lemma 4.13 and the fact that
the multiplicity is calculated at a branch, to replace the use of
biunivocity. Now, if $P$ is situated in generic position in
$P^{w}$, then\\
 $\{pr_{P}(\bar a_{m_{1}}),\ldots,pr_{P}(\bar
a_{m_{w-1}})\}$ forms a linearly independent sequence passing
through $pr_{P}(O)\in P^{w-1}$ and, for any hyperplane
$H_{\lambda}\subset P^{w-1}$, we have that, for $i\leq w-2$;\\

$<pr_{P}(\bar a_{m_{1}}),\ldots,pr_{P}(\bar a_{m_{i}})>\subset
H_{\lambda}$ iff $<\bar a_{m_{1}},\ldots,\bar a_{m_{i}}>\subset
<P,H_{\lambda}>$\\

It follows by $(*)$ and Theorem 6.2 that;\\

$I_{italian}(pr_{P}(O),pr_{P}(\gamma_{O}),pr_{P}(C),H_{\lambda})\geq\alpha_{0}+\ldots\alpha_{i}$\\

iff\\

$<pr_{P}(\bar a_{m_{1}}),\ldots,pr_{P}(\bar a_{m_{i}})>\subset
H_{\lambda}\ \ (i\leq w-2)$\\

Hence, by Theorem 6.2 again, we have that the branch
$pr_{P}(\gamma_{O})$ has character
$(\alpha_{0},\ldots,\alpha_{w-2})$.\\

If $P$ is situated generically on the $k$'th oscillatory plane,
but not on an oscillatory plane of lower order, then, the
sequence\\
 $\{pr_{P}(\bar a_{m_{1}}),\ldots,pr_{P}(\bar
a_{m_{k-1}}),pr_{P}(\bar a_{m_{k+1}}),\ldots,pr_{P}(\bar
a_{m_{w-1}})\}$ forms a linearly independent set, and, for $1\leq i\leq k-2$;\\

$<pr_{P}(\bar a_{m_{1}}),\ldots,pr_{P}(\bar a_{m_{i}})>\subset
H_{\lambda}$ iff $<\bar a_{m_{1}},\ldots,\bar a_{m_{i}}>\subset
<P,H_{\lambda}>$\\

whereas;\\

$<pr_{P}(\bar a_{m_{1}}),\ldots,pr_{P}(\bar a_{m_{k-1}})>\subset
H_{\lambda}$ iff $<\bar a_{m_{1}},\ldots,\bar a_{m_{k}}>\subset
<P,H_{\lambda}>$\\

and, for $1\leq j\leq (w-2-k)$;\\

$<pr_{P}(\bar a_{m_{1}}),\ldots,pr_{P}(\bar
a_{m_{k-1}}),pr_{P}(\bar a_{m_{k+1}}),\ldots,pr_{P}(\bar
a_{m_{k+j}})>\subset H_{\lambda}$\\

 iff\\

  $<\bar
a_{m_{1}},\ldots,\bar
a_{m_{k+j}}>\subset <P,H_{\lambda}>$\\

The rest of the theorem then follows immediately by the same
argument as given above.

\end{proof}

We also have the following important consequence of Theorem 6.1.

\begin{theorem}{Linearity of Multiplicity at a Branch}\\

Let $C\subset P^{w}$ be a projective algebraic curve and
$\gamma_{p}^{j}$ a branch of $C$, centred at $p$. Let $\Sigma$ be
an independent system having finite intersection with $C$. Then,
for any $k\geq 1$, the condition;\\

$\{\lambda\in
Par_{\Sigma}:I_{italian}(p,\gamma_{p}^{j},C,F_{\lambda})\geq
k\}$ $(*)$\\

is linear and definable.

\end{theorem}

\begin{proof}
That $(*)$ is definable follows from the definition of
multiplicity at a branch and elementary facts about Zariski
structures. Moreover, $(*)$ is definable over the parameters of
$C$ and the point $p$. In order to prove linearity, suppose that
$F_{\lambda}$ and $F_{\mu}$ belong to $\Sigma$ and satisfy $(*)$.
Let $(x_{1}(t),\ldots,x_{w}(t))$ be a parameterisation of the
branch $\gamma_{p}^{j}$ as given by Theorem 6.1. Then;\\

$k\leq min\{ord_{t}F_{\lambda}(x_{1}(t),\ldots,x_{w}(t)),ord_{t}F_{\mu}(x_{1}(t),\ldots,x_{w}(t))\}$\\

Then, for any constant $c$;\\

$(F_{\lambda}+cF_{\mu})(x_{1}(t),\ldots,x_{w}(t))=F_{\lambda}(x_{1}(t),\ldots,x_{w}(t))+cF_{\mu}(x_{1}(t),\ldots,x_{w}(t))$\\

Hence,\\

$ord_{t}(F_{\lambda}+cF_{\mu})(x_{1}(t),\ldots,x_{w}(t))\geq k$\\

as well. This shows that the pencil of curves generated by
$\{F_{\lambda},F_{\mu}\}$ satisfies $(*)$. Let $W$ be the closed
projective subvariety of $Par_{\Sigma}$ defined by the condition
$(*)$. Then, $W$ has the property that for any $\{a,b\}\subset W$,
$l_{ab}\subset W$. It follows easily that $W$ defines a plane $H$
in $Par_{\Sigma}$ as required. (Use the fact that for any tuple
$\{a_{1},\ldots,a_{n}\}$ in $W$, the plane
$H_{a_{1},\ldots,a_{n}}\subset W$ and a dimension argument)

\end{proof}

\begin{rmk}
Note that, given $\Sigma$ of dimension $n$ as in the statement of
the theorem, we can, using the above theorem, find a sequence
$\{\beta_{0},\ldots,\beta_{n-1}\}$ and a filtration;\\

$\Sigma_{n-1}\subset\ldots\subset\Sigma_{i}\subset\ldots\subset\Sigma_{0}\subset\Sigma$\\

with $dim(\Sigma_{i})=(n-1)-i$ such that;\\

$I_{italian}(p,\gamma_{p}^{j},C,F_{\lambda})\geq
\beta_{0}+\ldots+\beta_{i}$ iff $F_{\lambda}\in\Sigma_{i}$\\

and these are the only multiplicities which can occur. Note also
that, as an easy consequence of the theorem, given any tuple
$(\beta_{0},\ldots,\beta_{n-1})$ and $\Sigma$ as above;\\

$\{x\in C:x\ has\ character\ (\beta_{0},\ldots,\beta_{n-1})\ with\
respect\ to\ \Sigma\}$\\

is constructible and defined over the field of definition of $C$
and $\Sigma$. In particular, it follows from the previous
Definition 6.3, in characteristic $0$, that there exist only
finitely many points on $C$ which are the origins of non-ordinary branches.\\

\end{rmk}

We have the following important characterisation of
multiplicity at a branch;\\

\begin{theorem}{Multiplicity at a Branch as a Specialised
Condition}\\

Let $C$ and $\Sigma$ be as in the statement of Theorem 6.5 and the
following remark. Fix independent generic points (over $L$)
$\{p_{0j},\ldots,p_{ij}\}$ in $\gamma_{p}^{j}$. Then the system
$\Sigma_{i}$ may be obtained by specialising the condition;\\

$\{\lambda\in Par_{\Sigma}:F_{\lambda}=0$ passes
through $\{p_{0j},\ldots,p_{ij}\}\}$.\\

That is, if $F_{\lambda}\in\Sigma_{i}$, there exists
$\lambda'\in{\mathcal V}_{\lambda}$ such that $F_{\lambda'}$
passes through $\{p_{0j},\ldots,p_{ij}\}$, while, if
$F_{\lambda'}$ passes through $\{p_{0j},\ldots,p_{ij}\}$, then its
specialisation $F_{\lambda}$ belongs to $\Sigma_{i}$.\\

\end{theorem}

\begin{proof}
We will first assume that the curve $C$ is non-singular, (hence,
by Lemma 5.4, there exists a single branch at $p$). Consider the
cover $F_{i}\subset ({C\setminus p})\times
Par_{\Sigma}$ given by;\\

$F_{i+1}(x,\lambda)\equiv (x\in C\cap F_{\lambda})\wedge (F_{\lambda}\in\Sigma_{i})$ $(i\geq 0)$\\

For generic $q\in C$, the fibre $F_{0}(q,y)$ has dimension
$n-1-(i+1)$. Hence, we obtain an open subset $U\subset C$ such
that $F_{i+1}\subset (U\setminus p)\times Par_{\Sigma}$ is regular
with fibre dimension $n-1-(i+1)$. Let $\bar F_{i+1}$ be the
closure of $F_{i+1}$ in $U\times Par_{\Sigma}$. We claim that
$\Sigma_{i+1}$ is defined by the fibre ${\bar F_{i+1}}(p,y)$.
First observe that, as $p$ has codimension $1$ in $U$, $dim({\bar
F_{i+1}}(p,y))\leq n-1-(i+1)$. As $p$ is non-singular, each
component of the fibre ${\bar F_{i+1}}(p,y)$ has dimension at
least $n-1-(i+1)$. Suppose that ${\bar F_{i+1}}(p,\lambda)$ holds,
then, by regularity of $p$ for the cover $\bar F_{i+1}$, given
$p'\in U\cap {\mathcal V}_{p}$ generic, we can find
$\lambda'\in{\mathcal V}_{\lambda}$ such that
$F_{i+1}(p',\lambda')$, that is $F_{\lambda'}$ passes through $p'$
and $F_{\lambda'}$ belongs to $\Sigma_{i}$. By definition of
$\Sigma_{i}$, we have that $I_{italian}(p,C,F_{\lambda'})\geq
\beta_{0}+\ldots+\beta_{i}$. As $p'$ is distinct from $p$, by
summability of specialisation, see the paper \cite{depiro2}, we
must have that
$I_{italian}(p,C,F_{\lambda})\geq\beta_{0}+\ldots+\beta_{i}+1$.
Therefore, in fact, by the above theorem,
$I_{italian}(p,C,F_{\lambda})\geq\beta_{0}+\ldots+\beta_{i+1}$ and
$F_{\lambda}$ belongs to $\Sigma_{i+1}$. By dimension
considerations, it follows that $\bar F_{i+1}(p,y)$ defines the
system $\Sigma_{i+1}$ as required.\\

We now prove one direction of the theorem by induction on $i\geq
0$. Suppose that $\{p_{0j},\ldots,p_{i+1,j}\}$ are given
independent generic points in $C\cap{\mathcal V}_{p}$. By the
above, if $F_{\lambda}$ belongs to $\Sigma_{i+1}$, then there
exists $\lambda'\in{\mathcal V}_{\lambda}$ such that
$F_{\lambda'}$ belongs to $\Sigma_{i}$ and passes through
$p_{i+1,j}$. Moreover, as all the covers $F_{i}$ are defined over
the field of definition $L$ of $C$, we may take $\lambda'$ to lie
in the field $L_{1}=L(p_{i+1,j})^{alg}$. Hence, $F_{\lambda'}$
does not pass through any of the other independent generic points
$\{p_{0j},\ldots,p_{ij}\}$. Let
$L_{2}=L(p_{0j},\ldots,p_{ij})^{alg}$. As $dim(p_{kj}/L)=1$, for
$1\leq k\leq i$, we may, without loss of generality, assume that
$L_{2}$ is linearly disjoint from $L_{1}$ over $L$. Hence, by the
amalgamation property for the universal specialisation, see the
paper \cite{depiro3}, we have that $\{p_{0j},\ldots,p_{ij}\}$
still belong to ${\mathcal V}_{p}\cap C$ when taking $P(L_{1})$ as
the standard model. Now, we consider the subsystem
$\Sigma'\subset\Sigma$ defined by;\\

$\Sigma'=\{F_{\lambda}:F_{\lambda}\ passes\ through\
p_{i+1,j}\}$\\

As $p_{i+1,j}$ was chosen to be generic, it cannot be a base point
for any of the subsystems;\\

$\Sigma_{n-1}\subset\ldots\subset\Sigma_{i}\subset\ldots\subset\Sigma_{0}\subset\Sigma$\\

Hence, we obtain a corresponding filtration;\\

$\Sigma_{n-2}\cap\Sigma'\subset\ldots\subset\Sigma_{i}\cap\Sigma'\subset\ldots\subset\Sigma_{0}\cap\Sigma'\subset\Sigma'$\\

with the properties in Remarks 6.6. We now apply the induction
hypothesis to $F_{\lambda'}\in\Sigma_{i}\cap\Sigma'$. We can find
$\lambda''\in{\mathcal V}_{\lambda'}$ such that $F_{\lambda''}$
passes through $\{p_{0j},\ldots,p_{ij}\}$ and $F_{\lambda''}$
belongs to $\Sigma'$. Hence $F_{\lambda''}$ passes through
$\{p_{0j},\ldots,p_{i+1,j}\}$. Finally, note that
$\lambda''\in{\mathcal V}_{\lambda}$, if one considers $P(L)$
rather than $P(L_{1})$ as the standard model. Hence, one direction
of the theorem is proved.\\

The converse direction may also be proved by induction on $i\geq
0$. Suppose that $F_{\lambda''}$ passes through independent
generic points\\
 $\{p_{0j},\ldots,p_{i+1,j}\}$ in
$\gamma_{p}^{j}$. As before, we may consider $p_{i+1,j}$ as
belonging to the standard model $P(L_{1})$ and
$\{p_{0j},\ldots,p_{ij}\}$ as belonging to ${\mathcal V}_{p}\cap
C$, relative to $P(L_{1})$. We again consider the subsystem
$\Sigma'\subset\Sigma$ as defined above. Let $\lambda'$ be the
specialisation of $\lambda''$ relative to $P(L_{1})$. Then,
$F_{\lambda'}$ belongs to $\Sigma'$, as $\Sigma'$ is defined over
$p_{i+1,j}$. Moreover, by the inductive hypothesis applied to
$\Sigma'$, $F_{\lambda'}$ also belongs to $\Sigma_{i}$. We now
apply the argument at the beginning of this proof, with $P(L)$ as
the standard model, to obtain that $F_{\lambda}$ belongs to
$\Sigma_{i+1}$, where $\lambda$ is the specialisation of
$\lambda'$, relative to $P(L)$. Clearly $\lambda''$ specialises to
$\lambda$, hence the converse direction is proved.\\

We still need to consider the case for arbitrary $C\subset P^{w}$.
Let $C^{ns}\subset P^{w'}$ be a non-singular model of $C$ and
suppose that the presentation $\Phi_{\Sigma'}$ of
$(C^{ns},\Phi^{ns})$ has $Base(\Sigma')$ disjoint from the fibre
$\Gamma_{[\Phi^{ns}]}(y,p)$. Let $\gamma_{p}^{j}$ correspond to
the infinitesimal neighborhood ${\mathcal V}_{p_{j}}$ of $p_{j}$
in $\Gamma_{[\Phi^{ns}]}(y,p)$ and let $\{{\overline
F}_{\lambda}\}$ be the system $\Sigma$ of lifted forms on $C^{ns}$
corresponding to the space of forms $\{F_{\lambda}\}$ in $\Sigma$.
By Lemma 5.12, it follows that, for $\lambda\in Par_{\Sigma}$,
$I_{italian}(p,\gamma_{p}^{j},C,F_{\lambda})=I_{italian}(p_{j},C^{ns},\overline{F_{\lambda}})$,
hence the character $(\beta_{0},\ldots,\beta_{n})$ of the branch
$\gamma_{p}^{j}$ with respect to the system $\Sigma$ is the same
as the character of the branch $\gamma_{p_{j}}$ with respect to
the lifted system $\Sigma$. Moreover, we obtain the same
filtration of $Par_{\Sigma}$, as given in Remark 6.6, for both
systems with respect to the branches
$\{\gamma_{p}^{j},\gamma_{p_{j}}\}$ $(\dag)$. Now, suppose that we
are given independent generic points $\{p_{0j},\ldots,p_{ij}\}$ in
$\gamma_{p}^{j}$. Then, by definition of $\gamma_{p}^{j}$, we can
find corresponding independent generic points
$\{p_{0j}',\ldots,p_{ij}'\}$ in $\gamma_{p_{j}}$. Now suppose that
$F_{\lambda}\in\Sigma_{i}$, then the corresponding
$\overline{F_{\lambda}}$ belongs to $\Sigma_{i}$. By the proof of
the above theorem for non-singular curves, we can find
$\lambda'\in{\mathcal V}_{\lambda}$ such that
$\overline{F_{\lambda'}}$ passes through
$\{p_{0j}',\ldots,p_{ij}'\}$. Then, by definition, the
corresponding $F_{\lambda}$ passes through
$\{p_{0j},\ldots,p_{ij}\}$. Conversely, suppose that we can find
$\lambda'\in{\mathcal V}_{\lambda}$ such that $F_{\lambda'}$
passes through $\{p_{0j},\ldots,p_{ij}\}$. Then the corresponding
$\overline {F_{\lambda'}}$ passes through
$\{p_{0j}',\ldots,p_{ij}'\}$. By the proof for non-singular
curves, the specialisation $\overline{F_{\lambda}}$ belongs to
$\Sigma_{i}$. Hence, by the observation $(\dag)$ above, the
corresponding $F_{\lambda}$ belongs to $\Sigma_{i}$ as well. The
theorem is then proved.
\end{proof}

\begin{rmk}
One can give a slightly more geometric interpretation of the
preceding theorem as follows;\\

Consider the cover $F\subset C^{i+1}\times Par_{\Sigma}$ given
by;\\

$F(p_{1},\ldots,p_{i+1},\lambda)\equiv\{p_{1},\ldots,p_{i+1}\}\subset
C\cap F_{\lambda}=0$\\

Generically, the cover $F$ over $C^{i+1}$ has fibre dimension
$n-(i+1)$. For a tuple $(p,\ldots,p)\in\Delta^{i+1}$, the
dimension of the fibre $F(p,\ldots,p)$ is $n-1$. By the above,
$\Sigma_{i}\subset F(p,\ldots,p)$, which has dimension $n-(i+1)$,
is regular for the cover, in the sense of the above theorem.\\

The theorem may be construed as a generalisation of an intuitive
notion of tangency. $i+1$ independent generic points on the branch
$\gamma_{p}^{j}$ determine a projective plane $H_{i}$ of dimension
$i$. As these $i+1$ points converge independently to $p$, the
plane $H_{i}$ converges to the $i$'th osculatory plane at $p$. As
with the proofs we have given of many of the original arguments in
\cite{Sev}, the method using infinitesimals in fact reverses this
type of thinking in favour of a more visual approach. In this
case, we have shown that, by moving the $i$'th osculatory plane
\emph{away} from $p$, we can cut the branch
$\gamma_{p}^{j}$ in $i+1$ independent generic points. \\

The theorem also provides an effective method of computing
osculatory planes at a branch $\gamma_{p}^{j}$ for $p\in C$.

\end{rmk}

We now require the following lemma;\\

\begin{lemma}
Let $F_{\lambda'}$ have finite intersection with $C$, where the
parameter $\lambda'$ is taken inside the non-standard model
$P(K)$. Then there exists a maximally independent set
$\{p_{0j},\ldots,p_{ij}\}$ of generic intersections (over $L$)
inside $\gamma_{p}^{j}$.
\end{lemma}

\begin{proof}
Let $W$ be the finite set of intersections inside $\gamma_{p}^{j}$
of $F_{\lambda'}$ with $C$. As $C$ is strongly minimal, if
$dim(W/L)=i+1$, then there exists a basis
$\{p_{0j},\ldots,p_{ij}\}$ of $W$ over $L$. In particular, we have
that $\{p_{0j},\ldots,p_{ij}\}$ are generically independent points
of $C$ and are maximally independent in $W$.
\end{proof}

We then have the following theorem;\\

\begin{theorem}{Intersections along a Branch}\\

Let hypotheses be as in the previous theorem. Let $i$ be maximal
such that $F_{\lambda'}$ belongs to $\Sigma_{i}$ and suppose that
$F_{\lambda'}$ intersects $\gamma_{p}^{j}$ in the maximally
independent set of generic points $\{p_{i+1,j},\ldots,p_{i+r,j}\}$
over $L$, where $r\leq (n-1)-i$. Then, if the branch
$\gamma_{p}^{j}$ has character $(\beta_{0},\ldots,\beta_{n})$ with
respect to the system $\Sigma$, $F_{\lambda'}$ intersects
$(\gamma_{p}^{j}\setminus p)$ in at least
$\beta_{i+1}+\ldots+\beta_{i+r}$ points, counted with
multiplicity.

\end{theorem}

\begin{proof}
We assume first that $C$ is non-singular. Let $F_{\lambda}$ be the
specialisation of $F_{\lambda'}$ relative to the standard model
$P(L)$. Then by Theorem 6.7, $F_{\lambda}$ belongs to
$\Sigma_{i+r}$ (replace the system $\Sigma$ in Theorem 6.7 by the
system $\Sigma_{i}$.) Hence, $I_{italian}(p,C,F_{\lambda})\geq
\beta_{0}+\ldots+\beta_{i+r}$, whereas
$I_{italian}(p,C,F_{\lambda'})=\beta_{0}+\ldots+\beta_{i}$. It
follows immediately, by summability of specialisation, see the
paper \cite{depiro2}, that the total multiplicity of intersections
of $F_{\lambda'}$ with $C$ inside the branch
$(\gamma_{p}^{j}\setminus p)$ is at least;\\

$(\beta_{0}+\ldots+\beta_{i+r})-(\beta_{0}+\ldots+\beta_{i})=\beta_{i+1}+\ldots+\beta_{i+r}$\\

as required. If $C$ is singular, let $(C^{ns},\Phi^{ns})$ be a
non-singular model, with a presentation $\Phi_{\Sigma'}$ such that
$Base(\Sigma')$ is disjoint from $\Gamma_{[\Phi]}(x,p)$. Then,
given the maximally independent set of generic points\\
$\{p_{i+1,j},\ldots,p_{i+r,j}\}$, in $\gamma_{p}^{j}$, for the
intersection of $C$ with $F_{\lambda'}$, we obtain a maximally
independent set for the intersection ${\overline F}_{\lambda'}\cap
C\cap{\mathcal V}_{p}$. By the above, and the fact that the
character of the branch $\gamma_{p}^{j}$ with respect to
$\{F_{\lambda}\}$ equals the character of the branch
$\gamma_{p_{j}}$ with respect to ${\overline F}_{\lambda}$, we
obtain that ${\overline F}_{\lambda'}$ intersects the branch
$(\gamma_{p_{j}}\setminus p_{j})$ in at least
$\beta_{i+1}+\ldots+\beta_{i+r}$ points with multiplicity. Hence,
using for example Lemma 5.12, and the fact that $Base(\Sigma')$ is
disjoint from $\gamma_{p_{j}}$, we obtain that $F_{\lambda'}$
intersects $(\gamma_{p}^{j}\setminus p)$ in at least
$\beta_{i+1}+\ldots+\beta_{i+r}$ points with multiplicity as well.
\end{proof}

\begin{rmk}
Note that, in the statement of the theorem, one cannot obtain that
$F_{\lambda'}$ intersects the branch $(\gamma_{p}^{j}\setminus p)$
in exactly $\beta_{i+1}+\ldots+\beta_{i+r}$ points, with
multiplicity. For example, consider the algebraic curve $C$ given
in affine coordinates by $x^{3}-y^{2}=0$. At $(0,0)$, this has a
cusp singularity with character $(2,1)$. The tangent line or
$1$'st osculatory plane, is given by $y=0$. The line
$y-\epsilon=0$, where $\epsilon$ is an infinitesimal, cuts the
branch of $C$ at $(0,0)$ in exactly $3=2+1$ points. However, the
total transcendence degree (over $L$) of these points is clearly
$1$. Neither can one obtain that $F_{\lambda'}$ intersects the
branch $(\gamma_{p}^{j}\setminus p)$ transversely. For example,
consider the algebraic curve $C$ given in affine coordinates by
$y-x^{2}=0$. Let $\Sigma$ be the $2$-dimensional system consisting
of (projective) lines. As $(0,0)$ is an ordinary simple point of
$C$, it has character $(1,1)$, which is also the character of
$(0,0)$ with respect to the system $\Sigma$. Again, the tangent
line or $1$'st osculatory plane, is given by $y=0$. The line
$y=(2\epsilon)x-\epsilon^{2}$ cuts the branch of $C$ at $(0,0)$ in
exactly one point $(\epsilon,\epsilon^{2})$, with multiplicity
$2$, and specialises to $y=0$.
\end{rmk}

\begin{rmk}
The above theorem is critical in calculations relating to the
transformation of branches by duality. We save this point of view
for another occasion.
\end{rmk}

We finish this section by applying the above ideas to examine the
behaviour of hyperplane systems on an arbitrary projective algebraic curve;\\

\begin{lemma}

Let $C\subset P^{w}$ be \emph{any} projective algebraic curve and let;\\

$V^{*}\subset NonSing(C)\times P^{w*}$ be $\{(x,\lambda):x\in
NonSing(C)\wedge H_{\lambda}\supset l_{x}\}$\\

Then $V^{*}$ defines an irreducible algebraic variety and, if
$\bar V^{*}\subset C\times P^{w*}$ defines its Zariski closure,
then, for a singular point $p$, which is the origin of branches
$\{\gamma_{p}^{1},\ldots,\gamma_{p}^{m}\}$, the fibre $\bar
V^{*}(p)$ consists exactly of the parameters for hyperplanes
$H_{\lambda}$, which contain at least one of the tangent lines
$l_{\gamma_{p}^{j}}$, ($1\leq j\leq m$).

\end{lemma}

\begin{proof}
The fact that $V^{*}$ defines an irreducible algebraic variety
follows easily from arguments given in Section 1. We let $\Sigma$
be the linear system defined by the set of hyperplanes
$\{H_{\lambda}:\lambda\in P^{w*}\}$ and let $C^{ns}$ be a
nonsingular model of $C$, with presentation $\Phi_{\Sigma'}$, such
that $Base(\Sigma')$ is disjoint from the fibres
$\{\Gamma_{[\Phi]}(x,p_{1}),\ldots,\Gamma_{[\Phi]}(x,p_{n})\}$,
where $\{p_{1},\ldots,p_{n}\}$ denote the singular points of $C$,
$(*)$. Let $\{V_{\Phi_{\Sigma'}},W_{\Phi_{\Sigma'}}\}$ be the
canonical sets associated to this presentation and let $\Sigma$
also denote the lifted system of forms on $C^{ns}$. We now lift
the cover $V^{*}$ to $C^{ns}$ by defining $V^{*,lift}\subset V_{\Phi_{\Sigma'}}\times P^{w*}$;\\

$V^{*,lift}=\{(x,\lambda):V^{*}(\Phi_{\Sigma'}(x),\lambda)\}$\\

Using the fact that $V^{*}$ defines an algebraic variety, it is
clear that $V^{*,lift}$ defines an algebraic variety as well. Let
$\bar V^{*,lift}$ be the Zariski closure of this cover inside
$C^{ns}\times P^{w}$. Now let $p$ be a singular point of $C$, let
$\gamma_{p}^{j}$ be a branch centred at $p$ and let $p_{j}$ be the
corresponding point of $C^{ns}$. By the assumption $(*)$, the
character of the branch $\gamma_{p}^{j}$, with respect to the
system $\Sigma$, is the same as the character of the branch
$\gamma_{p_{j}}$, with respect to the lifted system $\Sigma$. In
particular, the filtration given by Remarks 6.6 is the same. As in
Remarks 6.6, we let $\Sigma_{1}\subset\Sigma$ consist of the set
of hyperplanes $\{H_{\lambda}:H_{\lambda}\supset
l_{\gamma_{p}^{j}}\}$. We now claim that the fibre $\bar
V^{*,lift}(p_{j},z)$ defines $\Sigma_{1}$, $(**)$. The proof is
similar to the beginning of Theorem 6.7. Using the fact that
$p_{j}$ is non-singular, we have that $dim(\bar
V^{*,lift}(p_{j},z))=w-2$, $(\dag)$, and is regular for the cover
$(\bar V^{*,lift}/C^{ns})$. Now, suppose that $\bar
V^{*,lift}(p_{j},\lambda)$ holds, then, given generic
$p_{j}'\in({\mathcal V}_{p_{j}}\cap V_{\Phi_{\Sigma'}})$, we can
find $\lambda'\in {\mathcal V}_{\lambda}$ such that $\bar
V^{*,lift}(p_{j}',\lambda')$. In particular, using the assumption
$(*)$ again, the corresponding form $\overline{H_{\lambda'}}$ must
contain the tangent line $l_{p_{j}'}$. Now, applying the result of
Theorem 6.7 to the linear system $\Sigma_{p_{j}'}\subset\Sigma$ of
 forms passing through $p_{j}'$, taking $\{p_{j},p_{j}'\}$ as belonging
 to the standard model, we can find $p_{j}''\in{\mathcal V}_{p_{j}'}$,
  generically independent from $\{p_{j},p_{j}'\}$, and $\lambda''\in{\mathcal
V}_{\lambda'}$ such that $\overline{H_{\lambda''}}$ passes through
$\{p_{j}',p_{j}''\}$. As the pair $\{p_{j}',p_{j}''\}$ are
generically independent over $p_{j}$ and belong to ${\mathcal
V}_{p_{j}}$, taking only $p_{j}$ as belonging to the standard
model, again applying Theorem 6.7, we obtain that the
specialisation $\overline{H_{\lambda}}$ of
$\overline{H_{\lambda''}}$ belongs to $\Sigma_{1}$. Hence,
$\Sigma_{1}\subset \bar V^{*,lift}(p_{j},y)$ and the result $(**)$
then follows from this and the dimension consideration $(\dag)$.
We now project the cover $\bar V^{*,lift}$ to a cover
$W^{*}\subset C\times P^{w}$, by defining;\\

$W^{*}=\{(y,\lambda):y\in C\wedge\exists
x[\Gamma_{[\Phi]}(x,y)\wedge \bar V^{*,lift}(x,\lambda)]\}$\\

By construction, we have that $W^{*}$ is an irreducible closed
projective variety, such that its restriction
$W^{*}|W_{\Phi_{\Sigma'}}$ agrees with $V^{*}$. Hence, $W^{*}=\bar
V^{*}$. Moreover, for a singular point $p$ of $C$, the fibre
$W^{*}(p)$ consists exactly of the parameters for hyperplanes
$H_{\lambda}$, which contain at least one of the tangent lines
$l_{\gamma_{p}^{j}}$, where
$\{\gamma_{p}^{1},\ldots,\gamma_{p}^{j},\ldots,\gamma_{p}^{m}\}$
are the branches centred at $p$. This completes the proof.\\

\end{proof}

\begin{lemma}

Let $C\subset P^{w}$ be \emph{any} projective algebraic curve and let $O$ be
a fixed point of $C$, possibly singular, let;\\

$V^{*}\subset (C\setminus\{O\})\times P^{w*}$ be $\{(x,\lambda):x\in
C\setminus\{O\}\wedge H_{\lambda}\supset l_{Ox}\}$\\

Then $V^{*}$ defines an irreducible algebraic variety and, if
$\bar V^{*}\subset C\times P^{w*}$ defines its Zariski closure,
then, if $O$ is the origin of branches
$\{\gamma_{O}^{1},\ldots,\gamma_{O}^{m}\}$, the fibre $\bar
V^{*}(O)$ consists exactly of the parameters for hyperplanes
$H_{\lambda}$, which contain at least one of the tangent lines
$l_{\gamma_{O}^{j}}$, ($1\leq j\leq m$).

\end{lemma}

\begin{proof}
The proof is similar to Lemma 6.13. Using the same notation there, we choose a presentation $\Phi_{\Sigma'}$ of a nonsingular model $C^{ns}$. We define $\Sigma_{1}\subset\Sigma$ to consist of the set of hyperplanes $\{H_{\lambda}:H_{\lambda}\supset l_{\gamma_{O}^{j}}\}$, where $\gamma_{O}^{j}$ is one of the branches centred at $O$. We define $V^{*,lift}$ and $\bar V^{*,lift}$ as in the previous Lemma 6.13, using $\Phi_{\Sigma'}$, (with the corresponding modification of $V^{*}$). The aim is then to show the corresponding statement $(**)$ of Lemma 6.13, that the fibre $\bar V^{*}(O_{j},z)$ defines $\Sigma_{1}$, where $O_{j}$ corresponds to $\gamma_{O}^{j}$ in $C^{ns}$. As before, we have that $dim(\bar V^{*}(O_{j},z))=w-2$, $(\dag)$, and that $O_{j}$ is regular for the cover $(\bar V^{*,lift}/C^{ns})$. If $\bar V^{*,lift}(O_{j},\lambda)$ holds, then, given generic $O_{j}'\in({\mathcal V}_{O_{j}}\cap V_{\Phi_{\Sigma'}})$, we can find $\lambda'\in {\mathcal V}_{\lambda}$ such that $\bar V^{*,lift}(O_{j}',\lambda')$. The corresponding form $H_{\lambda'}$ then contains $l_{OO'}$ for some $O'$, generically independent from $O$, lying on $\gamma_{O}^{j}$. Applying Theorem 6.7 to the set of hyperplanes $\Sigma_{0}$ passing through $O$, we obtain that the specialisation $H_{\lambda}$ contains the tangent line $l_{\gamma_{O}^{j}}$. Combining this result with the dimension consideration $(\dag)$, we obtain $(**)$ as required. The result then follows by the same argument as Lemma 6.13.
\end{proof}
\end{section}

\begin{section}{Some Remarks on Frobenius}

When the field has non-zero characteristic, many of the above
arguments are complicated by the Frobenius morphism. However, we
take the point of view that this is an exception rather than a
general rule, hence the results are true if we exclude unusual
cases. We will consider each of the previous sections separately
and point out where to make these modifications. We briefly remind
the reader that given, algebraic curves $C_{1}$ and $C_{2}$, by a
generally biunivocal map, denoted for this section only using the
repetition of notation, $\phi:C_{1}\leftrightsquigarrow C_{2}$, we
mean a morphism $\phi$, defined on an open subset $U\subset
C_{1}$, such that $\phi$ defines a bijective correspondence
between $U$ and an open subset $V$ of $C_{2}$. In characteristic
$0$, a generally biunivocal map is birational, in the sense of
Definition 1.19. However, this is not true when the field has
non-zero characteristic, Frobenius being a counterexample.
 \\

Section 1. The results of this section hold in arbitrary
characteristic.\\

Section 2. We encounter the first problem in Theorem 2.3, the
proof of which depends on Lemma 2.10. Unfortunately, Lemma 2.10 is
\emph{not} true in arbitrary characteristic. However, as we will
explain below, Lemma 2.10 is true for a linear system $\Sigma$
which defines a birational morphism $\Phi_{\Sigma}$ on $C$. As
this was assumed in Theorem 2.3, its proof \emph{does} hold in
non-zero characteristic.\\

Lemma 2.10 does \emph{not} hold in arbitrary characteristic. Let
$C$ be the algebraic curve defined by $y=0$ in affine coordinates
$(x,y)$. Consider the linear system $\Sigma$ of dimension $1$
defined by $\phi_{t}(x,y):=(y=x^{2}+t)$ in characteristic $2$.
Then $\phi_{t}$ is tangent to $C$ at $(t^{1/2},0)$ for all $t$. In
particular, $(0,0)$ is a coincident mobile point for the linear
system. The reason for the failure of the lemma is that the
function $F(x,y)=y-x^{2}$ defines a Zariski unramified morphism on
$y=0$ at $(0,0)$, which is not etale, it is just the Frobenius map
in characteristic $2$. One can avoid such cases by insisting that
the linear system $\Sigma$ under consideration defines a separable
morphism on $C$ $(*)$. With this extra requirement and a result
from \cite{depiro1} (Theorem 6.11) that any locally Zariski
unramified separable morphism between curves is locally etale, the
proof of Lemma 2.10 holds.\\

The remaining results of the section are unaffected, with the
restriction $(*)$ on $\Sigma$ in non-zero characteristic. In
particular, Lemma 2.30 holds with this restriction. The proof of
the Lemma gives the existence of a generally biunivocal morphism
$\phi$. By seperability, this induces an isomorphism of the
function fields of the respective curves. By an elementary
algebraic and model theoretic argument, see for example \cite{H},
(Theorem 4.4 p25), one can then invert the morphism $\phi$ in the
sense of Definition 1.19. Therefore, $\phi$ will define a birational map.\\

Section 3. The proof of Lemma 3.2 requires results from Section 2
which may not hold in certain exceptional cases. However, the
Lemma is \emph{still} true in arbitrary characteristic. One should
replace the use of Lemma 2.30 by invoking general results for
plane curves in \cite{depiro2}. Lemma 3.6 and Theorem 3.3 also
holds, if we replace birational with biunivocal. In order to
obtain the full statement of Theorem 3.3 in arbitrary
characteristic, one can use the following argument;\\

We obtain from the argument of Theorem 3.3, in arbitrary
characteristic, a generally biunivocal morphism $\phi$ from
$C\subset P^{2}$ to $C_{1}\subset P^{w}$. This induces an
inclusion of function fields $\phi^{*}:L(C_{1})\rightarrow L(C)$.
We may factor this extension as $L(C_{1})\subset L(D)\subset
L(C)$, with $D$ an algebraic curve, $L(D)\subset L(C)$ a purely
inseperable extension and $L(C_{1})\subset L(D)$ a seperable
extension. We, therefore, obtain rational maps
$\phi_{1}:C\rightsquigarrow D$ and $\phi_{2}:D\rightsquigarrow
C_{1}$, such that $\phi_{2}\circ\phi_{1}$ is equivalent to $\phi$
as a biunivocal map between $C$ and $C_{1}$. Now, using the method
of \cite{depiro1} (Remarks 6.5), we may find an algebraic curve
$C'\subset P^{2}$ (apply some power of Frobenius to the
coefficients defining $C$) and a morphism $Frob^{n}:C\rightarrow
C'$, together with a birational map
$\phi_{3}:D\leftrightsquigarrow C'$ such that $Frob^{n}$ and
$\phi_{3}\circ\phi_{1}$ are equivalent as biunivocal maps between
$C$ and $C'$. We now obtain a seperable rational map
$\phi_{4}=\phi_{2}\circ\phi_{3}^{-1}:C'\rightsquigarrow C_{1}$,
such that $\phi_{4}\circ Frob^{n}$ and $\phi$ are equivalent as
biunivocal morphisms. Now let $U\subset NonSing(C)$ be an open set
on which $\phi$ and $\phi_{4}\circ Frob^{n}$ are defined and agree
as morphisms. By an elementary application of the chain rule and
the fact that the differential of the Frobenius morphism is
identically zero, one obtains that, for any $x\in U$,
$(D\phi)_{x}$ contains the tangent line $l_{x}$ of $C$. By the
methods in the introduction of Section 1, this is in fact a closed
condition on $(D\phi)$, hence, in fact $(D\phi)_{x}$ contains the
tangent line $l_{x}$ of $C$, for $x\in NonSing(C)$, at any point
where $\phi$ is defined. We can summarise this more generally in
the following lemma;

\begin{lemma}
Let $\phi:C\rightsquigarrow P^{w}$ be an inseperable rational map,
then, for \emph{any} nonsingular point $x$ of $C$ at which $\phi$
is defined, $(D\phi)_{x}$ contains the tangent line $l_{x}$ of
$C$.
\end{lemma}

By Remark 3.7, this property is excluded for a transverse
$g_{n}^{r}$ as used in Theorem 3.3.\\

Section 4. The projection construction defined at the beginning of
the section may fail to define a separable morphism in non-zero
characteristic. However, using Lemma 7.1 and methods from Section
1, one can easily show that this only occurs for projective curves
$C$ with the property that, for every $x\in NonSing(C)$, the
tangent line $l_{x}$ passes through a given point $P$. In this
case, the projection of $C$ from $P$ onto any hyperplane will be
inseparable. In \cite{H} such curves are called \emph{strange}.
Non-singular strange curves were completely classified by Samuel in \cite{S};\\

\begin{theorem}
The only strange non-singular projective algebraic curves are the
line and the conic in characteristic $2$.
\end{theorem}

However, there are examples of other \emph{singular} strange
projective algebraic curves in $P^{w}$, for $w\geq 2$, not
contained in any hyperplane section. For example, the curves
$Fr_{w}\subset P^{w}$ obtained by iterating Frobenius, given parametrically by;\\

 $(t,t^{p},t^{p^{2}},\ldots,t^{p^{w-1}})$ in characteristic $p$.\\

For these examples, Lemma 4.2 fails. In order to see this, pick
independent points $\{T,S\}$ on $Fr_{w}$ given by
$(t,t^{p},\ldots,t^{p^{w-1}})$ and\\
$(s,s^{p},\ldots,s^{p^{w-1}})$.
Then, the equation of the chord $l_{TS}$ is given parametrically by;\\

$(t+\lambda s,t^{p}+\lambda s^{p},\ldots,t^{p^{w-1}}+\lambda
s^{p^{w-1}})$\\

If $t+\lambda s=v$, and $V$ is given by
$(v,v^{p},\ldots,v^{p^{w-1}})$, then we have that the chord
$l_{TS}$ meets $V$, distinct from $\{T,S\}$, iff we can solve
$\lambda^{p-1}=1,\ldots,\lambda^{p^{w-1}-1}=1$ for $\lambda\neq
1$. This is clearly possible if $p\geq 3$. In this case, we would
have that the chord $l_{TS}$ intersects $Fr_{w}$ in at least $p$
points.\\

 Lemma 4.2 holds in arbitrary characteristic, if we exclude singular strange projective curves,
 however the proof should be modified as it involves Lemma 2.10
applied to a projection. If $C$ is a non-singular strange curve,
using the classification given above, the theorem has no content
as we assumed that $C$ was not contained in any hyperplane
section.\\

\begin{lemma}{Lemma 4.2 in arbitrary characteristic, excluding singular strange projective curves}\\

Let $C\subset P^{w}$, for $w\geq 3$, not contained in any
hyperplane section and such that $C$ is \emph{not} a singular
strange projective curve. Suppose that $\{A,B\}$ are independent
generic points of $C$, then the line $l_{AB}$ does not otherwise
meet the curve $C$.

\end{lemma}

\begin{proof}
We use the same notation as in Lemma 4.2. Let $pr_{P}$ be the
projection defined in this lemma. Suppose that $pr_{P}$ is
inseperable, then, by the above remarks $C$ is a strange
projective algebraic curve, such that all its tangent lines
$l_{x}$, for $x\in NonSing(C)$, pass through $P$. Hence, we can
assume that $pr_{P}$ is seperable. We now show that the degenerate
case $(\dag)$ cannot occur. Suppose that $pr_{P}(l_{A})$ is a
point. We have that $dim_{P}(A)=1$, hence we can find an open
$W\subset NonSing(C)$, defined over $P$, such that, for $x\in W$,
$l_{x}$ passes through $P$. In particular, as $dim_{P}(B)=1$,
$l_{B}$ passes through $P$, hence we must have that
$l_{A}=l_{B}=AB$. As $A$ and $B$ were independent, it follows
easily that $C$ must be a line $l$, which is a contradiction. We
can now follow through the proof of Lemma 4.2 to obtain that;\\

There exists an open $W\subset NonSing(C)$, defined over the field
of definition of $C$, such that, for $y\in pr_{B}(W)$, the $l_{y}$ intersect in a point $Q$. $(****)$\\

It follows that, for $x\in W$, the $l_{x}$ intersect $l_{B}$. In
particular, $l_{A}$ intersects $l_{B}$. If $l_{A}=l_{B}$, we
obtain that $C$ is a line, hence we may assume that $l_{A}\cap
l_{B}=Q$. As $B$ was generic, we can find an open subset
$W'\subset NonSing(C)$, defined over $AB$, such that, for $x\in
W'$, the $l_{x}$ intersect $l_{A}$ and $l_{B}$. Then, either, for
such $x\in W'$, the $l_{x}$ all pass through $Q$ or the $l_{x}$
all lie in the plane $P_{AB}$ defined by $l_{A}$ and $l_{B}$. In
the first case, we have that $C$ is a strange curve, contradicting
the hypotheses. In the second case, we use the fact that the plane
$Pl=P_{AB}$ must be defined over the field of definition of $C$
and then, by the fact that the generic chord $l_{AB}$ lies in
$Pl$, that $C$ must be contained in $Pl$ as well, contradicting
the hypotheses.
\end{proof}

Lemma 4.5 is true if we exclude singular strange projective
curves. In order to obtain the corresponding result for a singular
strange projective curve $C$, pick a generic point point $P\in
P^{w}$. Let $x\in C$ be generic and independent from $P$. We claim
that $l_{Px}$ does not otherwise meet the curve $(*)$. If not, we
can find $y\in C$, distict from $x$, such that $P\in l_{xy}$.
Hence, $dim(P/xy)\leq 1$ and $dim(P/x)\geq 3$.  Now calculate
$dim(Pxy)$ in two different ways;\\

(i). $dim(Pxy)=dim(P/xy)+dim(xy)\leq 1+2=3$\\

(ii).$dim(Pxy)=dim(y/Px)+dim(Px)\geq 0+4=4$\\

This clearly gives a contradiction. It follows, using $(*)$, by an
elementary model theoretic argument, that the projection $pr_{P}$
onto any hyperplane $H$ will be generally biunivocal on $C$. Lemma
4.6 may also be easily modified to include the case of singular
strange curves. Theorem 4.8 holds in arbitrary characteristic by
the modifications of Lemma 4.5 and Lemma 4.6 and by ensuring that
the projections $pr_{P}$ always define seperable morphisms. In the
case of strange curves, we can always ensure this by picking the
centre of projection $P$ to be disjoint from the \emph{bad} point
$Q$, defined as the intersection of the tangent lines. Lemma 4.9
is still true in arbitrary characteristic but the proof needs to
be modified in order to take into account singular strange
projective curves, (we implicitly used Lemma 4.2 in the proof).
Using the same notation as in the lemma, given $x\in C$, using the
same argument, we can find $P\in P^{w}$ generic, such that
$l_{xP}$ does not otherwise meet the curve. Now using the
modification of Lemma 4.5, the projection from $P$ will be
generally biunivocal and, by construction, biunivocal at $x$. We
can then obtain the lemma by repeating this argument. Lemma 4.12
holds in arbitrary characteristic provided the projection $pr$ is
seperable. As we have already remarked, this can always be
arranged in non-zero characteristic. Lemma 4.14 holds in arbitrary
characteristic, using the previous modified lemmas, and the fact
that a seperable biunivocal map, between $C$ and $pr(C)$, may be
inverted to give a birational map including the nonsingular points
of $pr(C)$. It follows that Theorem 4.15 holds in arbitrary
characteristic as well, by the modifications from Section 3.
Finally, Theorem 4.16 holds by checking the result for certain
further unusual curves, depending on generalisations of results in
later sections, (see $(\dag)$ below) . We should note that,
without these generalisations, Theorem 4.16 still holds if one
accepts the weaker definition of a node as the origin of $2$
linear branches (see Definition 6.3).\\

Section 5. The results of this section hold in arbitrary
characteristic up to Lemma 5.24. We only make the remark that it
is always possible to choice a \emph{maximal} linear system such
that it defines a separable morphism on a projective curve $C$.
The proof of Lemma 5.24 has the same complications as Lemma 2.10.
Again, we can avoid these complications and recover the remaining
results of the section in arbitrary characteristic, by the
assumption on the linear system $\Sigma$ that it defines a
seperable morphism on $C$. \\

Section 6. The results up to Definition 6.3 hold, by appropriate
choices of linear systems $\Sigma$. In Definition 6.3, the claim
that a generic point of an algebraic curve is an ordinary simple
point does not hold in arbitrary characteristic, $(\dag)$. An
example is given by the plane quartic curve
$F(x,y,z)=x^{3}y+y^{3}z+z^{3}x=0$ over a field of characteristic
$3$. Every point of this curve is an inflection point, that is a
point with character $(1,2)$. In
this case, the natural duality map;\\

$DF:C\rightarrow C^{*}$\\

$[x:y:z]\mapsto [F_{x}:F_{y}:F_{z}]=[z^{3}:x^{3}:y^{3}]$\\

 is purely inseparable. In order to prove Theorem 4.16 in arbitrary
characteristic, one needs to classify such exceptional curves.
This can be done, using work of Plucker on the transformation of
branches by duality, we save this point of view for another
occasion.
\end{section}

\end{document}